\documentclass[12pt]{article}

\usepackage{amsmath,amssymb,latexsym, amsbsy, amsfonts, amscd, amsthm, mathrsfs, xy,comment, xcolor}
\usepackage{multirow}
\usepackage{hyperref}
\usepackage{enumitem, empheq, url}
 \usepackage[bibtex-style]{amsrefs}
\xyoption{all}
\usepackage{tikz-cd}
\usepackage{rotating}
\newcommand*{\isoarrow}[1]{\arrow[#1,"\rotatebox{90}{\(\sim\)}"]}

\theoremstyle{plain}

\newtheorem{theorem}{Theorem}[section]
\newtheorem{conjecture}[theorem]{Conjecture}
\newtheorem{prop}[theorem]{Proposition}
\newtheorem{corollary}[theorem]{Corollary}

\newtheorem{lemma}[theorem]{Lemma}

\newcommand{\longtwoheadrightarrow}{\relbar\joinrel\twoheadrightarrow}

\theoremstyle{definition}
\newtheorem{definition}[theorem]{Definition}
\newtheorem{remark}[theorem]{Remark}

\long\def\symbolfootnote[#1]#2{\begingroup
\def\thefootnote{\fnsymbol{footnote}}\footnote[#1]{#2}\endgroup}

\DeclareMathOperator{\GL}{GL}

\def\PP{{\mathbf P}}

\def\A{\mathbf{A}}
\def\cA{\mathcal{A}}

\def\sgn{\mathrm{sgn}}
\def\N{\mathrm{N}}
\def\1{\mf{1}}

\DeclareMathOperator{\id}{id}
\DeclareMathOperator{\KS}{SKu}

\DeclareMathOperator{\Fitt}{Fitt}
\DeclareMathOperator{\Ann}{Ann}

\DeclareMathOperator{\Ind}{Ind}

\DeclareMathOperator{\nr}{nr}
\DeclareMathOperator{\chr}{char}
\DeclareMathOperator{\Frac}{Frac}
\DeclareMathOperator{\RBS}{RBS}
\DeclareMathOperator{\res}{res}

\DeclareMathOperator{\cyc}{cyc}
\DeclareMathOperator{\adj}{adj}

\DeclareMathOperator{\cond}{cond}

 \DeclareMathOperator{\Norm}{Norm}
\DeclareMathOperator{\rec}{rec} 
\DeclareMathOperator{\Hom}{Hom} \DeclareMathOperator{\End}{End}

\DeclareMathOperator{\coker}{coker} \DeclareMathOperator{\Tr}{Tr}
\DeclareMathOperator{\ord}{ord}

\DeclareMathOperator{\cusps}{cusps}
 \DeclareMathOperator{\real}{Re}
 
 \DeclareMathOperator{\Cl}{Cl}
\DeclareMathOperator{\lcm}{lcm} 
 
 \DeclareMathOperator{\Frob}{Frob}

\DeclareMathOperator{\Ext}{Ext}
\DeclareMathOperator{\Sel}{Sel}
\DeclareMathOperator{\tr}{tr}

\DeclareMathOperator{\rank}{rank}
\DeclareMathOperator{\cont}{cont}

\newcommand{\bpsi}{\pmb{\psi}}

\newcommand{\mat}[4]{\begin{pmatrix}{#1} & {#2} \\ {#3} & {#4}
\end{pmatrix}}

\newcommand{\stack}[2]{\genfrac{}{}{0pt}{}{#1}{#2}}

\newcommand{\mf}{\mathfrak }

\newcommand{\mscr}{\mathscr}

\def\fa{\mathfrak{a}}
\def\fn{\mathfrak{n}}
\def\ft{\mathfrak{N}}
\def\fp{\mathfrak{p}}
\def\fq{\mathfrak{q}}

\def\fm{\mathfrak{m}}
\def\ft{\mathfrak{t}}

\def\fr{\mathfrak{r}}

\def\fl{\mathfrak{l}}
\def\fP{\mathfrak{P}}

\def\fm{\mathfrak{m}}
\def\fM{\mathfrak{M}}

\def\fd{\mathfrak{d}}

\def\fO{\mathfrak{O}}

\def\T{\mathbf{T}}
\def\Z{\mathbf{Z}}

\def\Q{\mathbf{Q}}
\def\C{\mathbf{C}}

\def\R{\mathbf{R}}

\def\bdf{\begin{defn}}
\def\edf{\end{defn}}

\def\cH{\mathcal{H}}

\def\cO{\mathcal{O}}

\def\cC{\mathcal{C}}

\def\fb{\mathfrak{b}}
\def\fc{\mathfrak{c}}

\def\Gal{{\rm Gal}}

\def\ab{{\rm ab}}

\def\cF{{\cal F}}

\def\ram{\text{ram}}
\def\ab{\text{ab}}

\def\sL{{\mscr L}}

\begin{document}
\baselineskip 15.8pt

\date{May 14, 2022}

\title{On the Brumer--Stark Conjecture}
\author{Samit Dasgupta \\ Mahesh Kakde}

\maketitle

\begin{abstract}

Let $H/F$ be a finite abelian extension of number fields with $F$ totally real and $H$ a CM field.
Let $S$ and $T$ be disjoint finite sets of places of $F$ satisfying the standard conditions. The Brumer--Stark conjecture states that the Stickelberger element $\Theta^{H/F}_{S, T}$ annihilates the $T$-smoothed class group $\Cl^T(H)$.  We prove this conjecture away from $p=2$, that is, after tensoring with $\Z[1/2]$.  We prove a stronger version of this result conjectured by Kurihara that gives a formula for the 0th Fitting ideal of the minus part of the Pontryagin dual of $\Cl^T(H) \otimes \Z[1/2]$ in terms of Stickelberger elements.
   We also show that this stronger result implies Rubin's higher rank version of the Brumer--Stark conjecture, again away from 2.

Our technique is a generalization of Ribet's method, building upon on our earlier work on the Gross--Stark conjecture.  Here we work with group ring valued Hilbert modular forms as introduced by Wiles.  A key aspect of our approach is the construction of  congruences between cusp forms and Eisenstein series that are stronger than usually expected, arising as shadows of the trivial zeroes of $p$-adic $L$-functions.  These stronger congruences are essential to proving that the cohomology classes we construct are unramified at $p$.  
\end{abstract}

\tableofcontents

\section{Introduction}

Let $F$ be a totally real field of degree $n$ over $\Q$.  Let $H$ be a finite abelian extension of $F$ that is a CM field.  Write $G = \Gal(H/F)$.   
 Associated to any character $\chi \colon G \longrightarrow \C^*$  one has the Artin $L$-function
\[ 
L(\chi, s) = \prod_{\fp} \frac{1}{1 - \chi(\fp) \N\fp^{-s}}, \qquad \real(s) > 1, 
\]
where the product ranges over  the maximal ideals $\fp \subset \cO_F$.  We adopt the convention that $\chi(\fp) = 0$ if $\chi$ is ramified at $\fp$.  The Artin $L$-function $L(\chi, s)$ has a meromorphic continuation to $\C$ that is analytic if $\chi \neq 1$, and has only a single simple pole at $s=1$ if $\chi=1$.

Let $\Sigma, \Sigma'$ denote disjoint finite sets of places of $F$ with $\Sigma \supset S_\infty$, the set of infinite places of $F$.  We do not impose any other conditions on $\Sigma, \Sigma'$.

The ``$\Sigma$-depleted, $\Sigma'$-smoothed" $L$-function of $\chi$ is defined by
\[ 
L_{\Sigma,\Sigma'}(\chi, s) = L(\chi, s) \prod_{\fp \in \Sigma \setminus S_\infty} (1 - \chi(\fp)\N\fp^{-s}) \prod_{\fp \in \Sigma'} (1 - \chi(\fp)\N\fp^{1-s}). 
\]
These $L$-functions can be packaged together into a Stickelberger element
\[ 
\Theta_{\Sigma, \Sigma'}^{H/F}(s) \in \C[G] 
\] 
defined by (we drop the superscript $H/F$ when unambiguous)
\[ 
\chi(\Theta_{\Sigma,\Sigma'}(s)) = L_{\Sigma, \Sigma'}(\chi^{-1}, s) \qquad \text{ for all } \chi \in \hat{G}. 
\]

A classical  theorem of Siegel, Klingen and Shintani implies that the specialization  \[ \Theta_{\Sigma, \Sigma'} = \Theta_{\Sigma, \Sigma'}(0) \] lies in $\Q[G]$.  For an integral statement, we must impose conditions on the depletion and smoothing sets.  Let $S, T$ denote disjoint finite sets of places of $F$ with $S\supset S_\infty \cup S_{\ram}$, where $S_{\ram}$ denotes the set of finite primes of $F$ ramified in $H$.  We impose the following condition on $T$.
\begin{equation}
\label{e:drcond}
  \parbox{\dimexpr\linewidth-4em}{
    \strut
Let $T_H$ denote the set of primes of $H$ above those in $T$.  The group of roots of unity $\zeta \in \mu(H)$ such that $\zeta \equiv 1 \!\!\!\!\!\pmod{\fp}$ for all $\fp \in T_H$ is trivial.
    \strut
  }
  \end{equation}
If $T$ contains two primes of different residue characteristic, or one prime of residue characterstic larger than $[F:\Q] + 1$, then this condition automatically holds. 
  A celebrated theorem of  Deligne--Ribet \cite{dr} and Cassou-Nogu\`es \cite{cn} states that \begin{equation} \label{e:dr}
 \Theta_{S,T} \in \Z[G]. \end{equation}

Let $\Cl^T(H)$ denote the ray class group of $H$ with conductor equal to the product of primes  in $T_H$.  This is defined as follows.
Let $I_T(H)$ denote the group of fractional ideals of $H$ relatively prime to the primes in $T_H$.
  Let $P_{T}(H)$ denote the subgroup of $I_T(H)$ generated by principal ideals $(\alpha)$ where $\alpha \in \cO_{H}$ satisfies $\alpha \equiv 1 \pmod{\fp}$ for all $\fp \in T_H$.  Then \[ \Cl^T(H) = I_T(H)/P_{T}(H). \]  This $T$-smoothed class group is naturally a $\Z[G]$-module.  
  
  Further, for later use, we define $\Cl^T_S(H)$ as follows. Let $S$ be a finite set of primes of $F$ disjoint from $T$. The group $\Cl^T_S(H)$ is defined as the quotient of $\Cl^T(H)$ by the subgroup generated by classes of primes of $H$ lying above primes in $S$.
  
  The following conjecture stated by Tate (\cite{tatebook}*{Conjecture~IV.6.2}) is often called the Brumer--Stark conjecture.  Note that the actual conjecture stated by Tate is very slightly stronger---see the discussion following (\ref{e:actual}) below.  This discrepancy disappears when 2 is inverted as it is in our results.

\begin{conjecture}[``The Brumer--Stark Conjecture"] \label{c:bs}  We have
\begin{equation} \label{e:bs}
 \Theta_{S,T} \in \Ann_{\Z[G]}(\Cl^T(H)). 
 \end{equation}
\end{conjecture}

A corollary of our main result is the prime-to-2 part of the Brumer--Stark conjecture.
\begin{theorem} \label{t:ptt}
 We have
\begin{equation} \label{e:ptt}
 \Theta_{S,T} \in \Ann_{\Z[G]}(\Cl^T(H)) \otimes \Z[\textstyle\frac{1}{2}]. \end{equation}
\end{theorem}
 
Let us briefly describe the history of the Brumer--Stark conjecture as well as its significance.  In 1890, Stickelberger proved (\ref{e:bs}) when $F = \Q$ by computing the ideal factorization of Gauss sums in cyclotomic fields \cite{stickelberger}.
In the late 1960s, Brumer defined and studied the Stickelberger element $\Theta_S = \Theta_{S, \emptyset}$ for arbitrary totally real fields $F$, generalizing Stickelberger's  construction.
Brumer conjectured that any element of $(\Theta_S \cdot \Z[G]) \cap \Z[G]$  annihilates $\Cl(H)/\overline{\Cl(F)}$, where $\overline{\Cl(F)}$ denotes the image of $\Cl(F)$ in $\Cl(H)$ under the natural map induced by extension of ideals.  This conjecture was not published by Brumer, but was described in lectures and became well-known to researchers in the field \cite{coates}.  Brumer's conjecture is explicitly stated for real quadratic  $F$  in the 1970 Ph.D.~thesis of Rideout  \cite{rideout}*{Theorem 1.15}.  See also the paper of Coates--Sinnott, where Brumer's ideas are discussed \cite{cs}*{Pp.~254 and 256}.

Throughout the 1970's Stark conducted a series of deep investigations into refinements of the analytic class number formula.  His ``rank one abelian conjecture," stated in \cite{stark}, proposed the existence of units $u$ in abelian extensions $H/F$ whose absolute values at all conjugates of a given archimedean place $w$ of $H$ are described explicitly in terms of the first derivatives at 0 of the $L$-functions of the extension $H/F$.  In addition, Stark observed in the cases he studied the following interesting condition: if $e = \#\mu(H)$ denotes the number of roots of unity in $H$, then the extension $H(u^{1/e})/F$ is abelian.  See Stark's pleasant exposition \cite{starkpc} for a description 
of the origin of his work on these conjectures, and in particular his discovery of this ``abelian" condition (\S4).

Tate realized that Brumer's conjecture and Stark's conjecture could be stated simultaneously in the same notational framework using an arbitrary place $v$ of $F$ that splits completely in $H$; when $v$ is finite one recovers Brumer's conjecture, and when $v$ is infinite one recovers Stark's rank one abelian conjecture.  Tate introduced the smoothing set $T$ and noted that Stark's abelian condition  can be interpreted as the statement that $ \Theta_{S,T}$ annihilates $\Cl^T(H)$, and not just the class group $\Cl(H)$.  Because of the incorporation of Stark's abelian condition into the conjecture, he called the conjecture the Brumer--Stark conjecture \cite{tatebook}*{\S4.6}.

\medskip

The Brumer--Stark conjecture can be related to Hilbert's 12th problem as follows.  Let $\fp \not \in S \cup T$ denote a prime of $F$ that splits completely in $H$.  Pick a prime $\fP$ of $H$ above $F$ and define $\zeta_{S,T}(\sigma)$ by writing $\Theta_{S,T} = \sum_{\sigma \in G} \zeta_{S,T}(\sigma)[\sigma^{-1}].$
  Conjecture~\ref{c:bs} implies that the ideal
\begin{equation} \label{e:actual}
 \fP^{\Theta_{S,T}} = \prod_{\sigma \in G} \sigma^{-1}(\fP)^{\zeta_{S,T}(\sigma)} \end{equation}
is a principal ideal $(u)$ generated by an element $u \equiv 1 \pmod{\fp \cO_H}$ for all $\fp \in T$.  
A very mild refinement of Conjecture~\ref{c:bs}, which was the actual statement proposed by Tate, is that the generator $u$ can be chosen to satisfy $\overline{u} = u^{-1}$, where $\overline{u}$ denotes the image of $u$ under the complex conjugation of $H$.  (In any case the quotient $v = u/\overline{u}$ for any generator $u$ would satisfy $\overline{v} = v^{-1}$ and generate the ideal $\fP^{2\Theta_{S,T}}$, so this refinement only concerns a factor of 2.)
The element $u$ satisfying these properties is unique and is 
called a Brumer--Stark unit.   This is a canonical $\fp$-unit in $H$ with valuations at primes above $\fp$ determined by the $L$-functions of the extension $H/F$:
\[  \sum_{\sigma \in G} \chi(\sigma) \ord_{\sigma^{-1}(\fP)}(u) = L_{S, T}(\chi, 0)  \]
 for all $\chi \in \hat{G}.$  The conjectural existence of the elements $u \in H$  suggests the possibility of an explicit class field theory for the ground field $F$.  This perspective is explored further in our forthcoming work~\cite{dk}, where we prove an explicit $p$-adic analytic formula for Brumer--Stark units and give applications to Hilbert's 12th problem for $F$. 

\subsection{Main Result}

Kurihara stated a refinement of the prime-to-2 part of Conjecture~\ref{c:bs} known as the Strong Brumer--Stark conjecture.    Let 
 \[ \Cl^T(H)^\vee = \Hom_{\Z}(\Cl^T(H), \Q/\Z) \]
 denote the Pontryagin dual of $\Cl^T(H)$ endowed with the contragradient $G$-action: \[ \sigma(f)(c)= f(\sigma^{-1}c). \]
 Let $x \mapsto x^\#$ denote the involution on $\Z[G]$ induced by $g \mapsto g^{-1}$ for $g \in G$. 
 Finally, for a $\Z[\frac{1}{2}][G]$-module $M$, let 
 \[ M^- = M/(\sigma+1) \cong \{m \in M \colon \sigma m = -m \}, \]
 where $\sigma \in G$ denotes the unique complex conjugation of $H$.
 If $M$ is only a $\Z[G]$-module, we let $M^- = (M \otimes_\Z \Z[\frac{1}{2}])^-$.  In particular $\Z[G]^- = \Z[\frac{1}{2}][G]/(\sigma + 1)$.  The following is a corollary of our main result. 
 
 \begin{theorem}[``Strong Brumer--Stark", Conjecture of Kurihara] \label{t:sbs}  We have
\begin{equation} \label{e:sbs}
 \Theta_{S,T}^\# \in \Fitt_{\Z[G]^-}(\Cl^T(H)^{\vee,-}). \end{equation}
\end{theorem}

Here $\Fitt$ denotes the 0th Fitting ideal. 
The Fitting ideal of $\Cl(H)$ and its smoothed version $\Cl^T(H)$ have been the subject of significant study for many years.  Experts have noted that the inclusion $ \Theta_{S,T} \in \Fitt_{\Z[G]^-}(\Cl^T(H)^{-})$ holds in important special instances, but is {\em false} in general.
This is studied in detail in \cite{gk}, where it is suggested  that the Fitting ideal of the {\em Pontryagin dual} of the class group  is better behaved than the class group itself.  See also \cite{pop} for a discussion of these issues. 

 Theorem~\ref{t:sbs} is seen to imply the prime-to-2 part of the Brumer--Stark conjecture (Theorem~\ref{t:ptt})
by combining the following observations: (a) the Fitting ideal of a module is contained in its annihilator; (b) for a module $M$ with finitely many elements one has $\Ann(M^\vee) = \Ann(M)^\#$; (c) the complex conjugation $\sigma$ acts as $-1$ on $\Theta_{S,T}$, so the element $\Theta_{S,T}$ annihilates a $\Z[\frac{1}{2}][G]$-module $M$ if and only if it annhilates $M^{-}$.

Our main result is the proof of even stronger refinement of the prime-to-2 part of the Brumer--Stark conjecture, which was also originally conjectured by Kurihara.  This result gives an exact formula for \[ \Fitt_{\Z[G]^-}(\Cl^T(H)^{\vee, -}) \] in terms of Stickelberger elements, as follows.
Let $S = S_{\ram} \cup S_\infty$.  For  $v \in S_{\ram}$, let $I_v \subset G_v \subset G$ denote the inertia and decomposition groups, respectively, associated to $v$. Let
\[ e_v = \frac{1}{\# I_v} \N I_v  = \frac{1}{\# I_v} \sum_{\sigma \in I_v} \sigma \in \Q[G] \] denote the idempotent that represents projection onto the characters unramified at $v$.  Let $\sigma_v \in G_v$ denote any representative of the Frobenius coset of $v$.  The element  $1 - \sigma_v e_v \in \Q[G]$ is independent of choice of representative.
Following \cite{greither}, we define the Sinnott--Kurihara ideal, {\em a priori}  a fractional ideal of $\Z[G]$, by
\[ \KS^T(H/F) = (\Theta^\#_{S_\infty, T}) \prod_{v \in S_{\ram}} ( \N I_v, 1 - \sigma_v e_v ) . \]
Kurihara showed using the Deligne--Ribet/Cassou-Nogu\'es theorem that $\KS^T(H/F) \subset \Z[G]$ (see Lemma~\ref{l:sku} below).  The following is our main result.

\begin{theorem}[Kurihara {\cite[Conjecture 3.2]{kuriharaunp}}] \label{t:sk}
We have 
\[ \Fitt_{\Z[G]^-}(\Cl^T(H)^{\vee, -})  = \KS^T(H/F)^-. \]
\end{theorem}
Theorem~\ref{t:sk} implies Strong Brumer--Stark (Theorem~\ref{t:sbs}),
and hence the prime-to-2 part of Brumer--Stark  (Theorem~\ref{t:ptt}), since
\begin{equation} \label{e:tstd}
\Theta^\#_{S, T} = \Theta^\#_{S_\infty, T}  \prod_{v \in S_{\ram}} (1 - \sigma_v e_v) \in  \KS^T(H/F). 
\end{equation}
 Greither proved a version of Theorem~\ref{t:sk} under the assumption of the Equivariant Tamagawa Number Conjecture \cite{greither}.

The partial progress that had previously been obtained toward the Brumer--Stark conjecture applied the Iwasawa Main Conjecture for totally real fields proven by Wiles~\cite{wiles}. 
Greither proved some special cases  of the Brumer--Stark conjecture \cite{greithernice} using the techniques of horizontal Iwasawa theory introduced by Wiles \cite{wilesb} under the assumption that the Iwasawa $\mu$-invariant $\mu_p(F)$ vanishes for each odd prime $p$.
  Greither and Popescu   \cite{gp} proved the $p$-part of Theorem~\ref{t:sbs} for odd primes $p$ assuming that $\mu_p(F)=0$ and that $S$ contains all the primes above $p$.  Burns, Kurihara, and Sano refined the Greither--Popescu result \cite{bks2}.  
Recently, Burns proved the $p$-part of Theorem~\ref{t:sbs} assuming that $\mu_p(F)=0$  and that the Gross--Kuzmin Conjecture \cite[Conjecture 1.15]{gross} holds for $(H, p)$ (i.e.\ the non-vanishing of Gross's $p$-adic regulator) \cite{burns}.

\subsection{The Rubin--Stark Conjecture}

 Theorem~\ref{t:sk} has as an important corollary the prime-to-2 part of Rubin's higher rank generalization of the Brumer--Stark conjecture.  Let us recall Rubin's conjecture, stated originally in the beautiful paper \cite{rubin}. We define $H_T^*$ to be the set of elements $h \in H^*$ such that $\ord_w(h - 1) > 0$ for all $w \in T_H$. 

Next we choose $r$ finite primes $S = \{v_1, \dots, v_r\}$ of $F$ that split completely in $H$.  
Define \begin{equation} \label{e:udef}
U_{S,T} = \{ u \in H_T^*: |u|_w = 1 \text{ for all finite primes } w \not\in S_H \} \end{equation}
and write $\Q U_{S,T} = U_{S,T} \otimes_\Z \Q$.   
Choose a prime $w_j$ of $H$ above each $v_j$.
The map
\begin{equation} \label{e:ordg} \begin{tikzcd}
\ord_G \colon \bigwedge\nolimits^r_{\Q[G]}  \Q U_{S,T}^- \ar[r] &  \Q[G]^- \end{tikzcd}
\end{equation}
induced  by 
\begin{equation} \ord_G(u_1 \wedge \cdots \wedge u_r)  = \det\left( \sum_{\sigma \in G} [\sigma^{-1}] \ord_{w_j}(\sigma(u_i))\right)_{i,j = 1, \dotsc, r} \label{e:ordgdef} \end{equation}
is a $\Q[G]$-module isomorphism.

Define the Rubin--Brumer--Stark element \[ u_{\RBS} \in  \bigwedge\nolimits^r_{\Q[G]} \Q U_{S,T}^- \subset \bigwedge\nolimits^r_{\Q[G]} \Q U_{S,T}  \] by \[ \ord_G(u_{\RBS})  = \Theta_{S,T}.\] Rubin  conjectured that $u_{\RBS}$ lies in a certain $\Z[G]$-lattice that is nowadays called ``Rubin's lattice,"  whose definition we now recall.  For $i = 1, \dotsc, r$, consider $\Z[G]$-module homomorphisms $\varphi_i: U_{S,T} \longrightarrow \Z[G]$.
Let
\[ \begin{tikzcd}
 \varphi \colon \bigwedge\nolimits^r_{\Q[G]} \Q U_{S,T} \ar[r] &  \Q[G] 
\end{tikzcd} \]  be the map induced  by
\[ \varphi(u_1 \wedge \cdots \wedge u_r)  = \det(\varphi_i(u_j)). \]  
The {\em $r$th exterior power bidual} of $U_{S,T}$, denoted  $\bigcap_{\Z[G]}^r U_{S,T}$,  
 is the set of  $u \in  \bigwedge_{\Q[G]}^r \Q U_{S,T}$
 such that $\varphi(u) \in \Z[G]$ for all $r$-tuples $(\varphi_1, \dotsc, \varphi_r) \in \Hom_{\Z[G]}(U_{S,T}, \Z[G])^{\oplus r}$.  
 Rubin's lattice is defined by
 \[ \sL = \left(  \bigwedge\nolimits^r_{\Q[G]} \Q U_{S,T}^-  \right) \cap \bigcap\nolimits_{\Z[G]}^r U_{S,T}. \]
 The exterior power bidual terminology  was introduced by Burns and Sano \cite{burnssano},
 who studied and developed Rubin's construction in greater generality.

\begin{conjecture}[Rubin] \label{c:rubin}  We have $u_{\RBS} \in \sL.$
\end{conjecture}

Note that  $u_{\RBS}$ depends on the choice of the $w_j$ only up to multiplication by an element of $G$, and the validity of Conjecture~\ref{c:rubin} is independent of this choice.  The Brumer--Stark conjecture is easily seen to be equivalent to the  rank $r=1$ case of Rubin's conjecture. 
In \S\ref{s:rubin} we show that Theorem~\ref{t:sbs}  implies the prime-to-2 part of Rubin's Conjecture:

\begin{theorem}  \label{t:r}
We have $u_{RBS} \in \sL \otimes_\Z \Z[\frac{1}{2}].$
\end{theorem}

\subsection{Summary of Proof}

We now sketch the proof of Theorem~\ref{t:sk}.  For simplicity we consider the case that $H/F$ is unramified at all finite primes (i.e.\ has conductor 1).
In this case the $\Z[G]^-$-module $\Cl^T(H)^-$ has a quadratic presentation, meaning that it has a finite $\Z[G]^-$-module presentation with the same number of generators and relations (see \S\ref{s:fitting}).  This implies that $\Fitt_{\Z[G]^-}(\Cl^T(H)^-)$  is principal.  Suppose  we can show that 
 \begin{equation} \label{e:fct}
  \Fitt_{\Z[G]^-}(\Cl^T(H)^-) \subset (\Theta_{S_\infty, T}). \end{equation}
The analytic class number formula implies that
\begin{equation} \label{e:cnf}
  \#\Cl^T(H)^- \doteq \prod_{\psi \text{ odd}} \psi(\Theta_{S_\infty, T}), \end{equation}
where $\doteq$ denotes equality up to a power of 2. In particular, the product in (\ref{e:cnf}) lies in $\Z$.  An elementary argument shows that (\ref{e:cnf}) implies that the inclusion (\ref{e:fct}) must be an equality (see \S\ref{s:fitting} for a description of this argument).
 Theorem~\ref{t:sk} follows from this since one can show that \[   \Fitt_{\Z[G]^-}(\Cl^T(H)^{\vee,-}) =  \Fitt_{\Z[G]^-}(\Cl^T(H)^{-})^\# \]
in this special setting ($H/F$ unramified at finite places).
 
 The inclusion (\ref{e:fct}) is proved using Ribet's method, which was originally invented by Ribet to prove the converse of Herbrand's Theorem in the seminal work \cite{ribet}.  
Our application of Ribet's Method owes a great debt to the techniques introduced by Wiles in \cite{wiles}.
We reintroduce the theory of {\em group ring valued Hilbert modular forms}. 
 These were considered by Wiles in \cite{wiles}, and this theory is developed further by Silliman in \cite{dks} and in this paper.
 
 Since $H/F$ has conductor 1, class field theory canonically identifies $G$ as a quotient of the narrow class group $\Cl^+(F)$.   Let \[ \bpsi\colon \begin{tikzcd}
  \Cl^+(F) \ar[r, two heads] &  G \ar[r] &  (\Z[G]^-)^*
  \end{tikzcd}
  \]
 denote the canonical character.  For a positive integer $k$, let $M_k$ denote the usual
 group of Hilbert modular forms for $F$ of level 1 and of weight $k$ with Fourier coefficients lying in $\Z$.  
 For $k$ odd, define $M_k(\bpsi)$ to be
 the $\Z[G]^-$-submodule of  $M_k \otimes \Q[G]^-$ consisting of those $f$ whose Fourier coefficients lie in $\Z[G]^-$ and such that for each odd character $\psi \in \hat{G}$, the specialization $\psi(f)$ is a classical form of nebentypus $\psi$.
The form $f$ can be viewed as encoding the ``family" of forms $\{\psi(f) \}$.  The fact that $f$ has integral Fourier coefficients implies that the forms $\psi(f)$ in this family satisfy certain congruences.
 
One of the few examples of group ring valued forms that one can write down explicitly are the Eisenstein series $E_k$.  Using these Eisenstein series along with an important auxiliary construction drawn from   \cite{dks}, we prove that for positive integers $k$ sufficiently large and close to 1 in $\hat{\Z}$, 
   there is a {\em cuspidal} group ring valued form $f$ such that:
\begin{equation} \label{e:fec}
 f \equiv E_k \pmod{\Theta_{S_\infty, T}}. \end{equation}
 Let  $\T$ denote the Hecke algebra over $\Z[G]^-$ of the module of weight $k$ cuspidal group ring valued forms. 
The congruence (\ref{e:fec}) implies that there is a surjective $\Z[G]^-$-algebra homomorphism 
\begin{equation} \label{e:feh}
 \begin{tikzcd}
  \varphi\colon \T \ar[r] & \Z[G]^-/(\Theta_{S_\infty, T}) 
  \end{tikzcd}
  \end{equation}
 such that for all primes $\fl \subset \cO_F$, we have \begin{equation} \label{e:phieisenstein}
 \varphi(T_\fl) = 1 + \bpsi(\fl). \end{equation}
  Let $I$ denote the kernel of $\varphi$ (the {\em Eisenstein ideal}). 
 
Let $p$ denote an odd prime, and replace $\T$ and $I$ by their $p$-adic completions.
 The Galois representations associated to cusp forms together with the congruence (\ref{e:fec}) allow for the construction of a faithful $\T$-module $B$
 along with a cohomology class
 \[  \kappa \in H^1(G_F, B/IB) \]   that is unramified at all primes not dividing $p$ or lying in $T$. 
Furthermore the image of $\kappa$ generates $B/IB$, and complex conjugation acts as $-1$ on this space. If $\kappa$ were unramified at all primes dividing $p$ as well, then $\kappa$ would cut out an extension of $H$
 unramified outside the primes of $T$ and tamely ramified at those primes. By 
  class field theory this yields a surjective homomorphism \[ \begin{tikzcd}
   \Cl^T(H)^- \ar[r, two heads] &  B/IB.
   \end{tikzcd} \]
 Since $\T/I \cong \Z_p[G]^-/(\Theta_{S_\infty, T})$ and $B$ is a faithful $\T$-module, general principles regarding Fitting ideals imply
 \[ \Fitt_{\Z_p[G]^-}(\Cl^T(H)^-) \subset \Fitt_{\Z_p[G]^-}(B/IB) \subset (\Theta_{S_\infty, T}). \]
 
This yields the desired inclusion (\ref{e:fct}). Unfortunately, it is simply not true that $\kappa$ is necessarily unramified at the primes above $p$.  Overcoming this obstacle is perhaps the central contribution to the theory of Ribet's method advanced by this paper.
 Previous works have employed the ingenious  method of Wiles \cite{wilesb} to introduce auxiliary primes into the set $S$ and twist by characters with conductor divisible by these primes.  However this technique introduces certain error terms that destroy the delicate results that we need to obtain here and hence give only partial results (see \cite{greithernice}).  Therefore, our new method deals with ramification at $p$ head-on.  We first show that the congruence (\ref{e:fec}) and corresponding homomorphism (\ref{e:feh}) can be strengthened.  There is a certain non-zerodivisor $x \in \Z_p[G]^-$ and a surjective $\Z_p[G]^-$-algebra homomorphism
 \begin{equation} \label{e:phix} 
 \begin{tikzcd}
  \varphi_x \colon \T  \ar[r] &  \Z_p[G]^-/(x \Theta_{S_\infty, T}) 
  \end{tikzcd}
   \end{equation}
that is Eisenstein in the sense that (\ref{e:phieisenstein}) holds.
   The element $x$ can be viewed as encoding the mod $p$ trivial zeroes of the characters $\psi \in \hat{G}$ at the primes $\fp \mid p$, i.e.\ such that $\psi(\fp) \equiv 1$ modulo a prime above $p$ (in which case $L_{S_\infty \cup \{\fp\}, T}(\psi, 0) \equiv 0$).  It is striking that these trivial zeroes play a crucial role even while considering the primitive Stickelberger element $\Theta_{S_\infty, T}$.
 
 Working as before, we let $I$ denote the kernel of $\varphi_x$ and construct a faithful $\T $-module $B$ together with a cohomology class \begin{equation} \label{e:kappadef0}
 \kappa \in H^1(G_F, B/IB) \end{equation} generating $B/IB$.
The class $\kappa$ is unramified at all primes not dividing $p$ or lying in $T$. To produce a class unramified at primes above $p$, we rather bluntly consider the image of the inertia groups at all primes above $p$ under
$\kappa$ and denote the $\T$-module that they generate in $B$ by $B(I_p)$.  We define $\overline{B} = B/(IB, B(I_p))$ and note that the image of $\kappa$ in $H^1(G_F, \overline{B})$ is now tautologically unramifed at the primes above $p$.  Hence we deduce a surjective homomorphism $\Cl^T(H)^- \longtwoheadrightarrow \overline{B}$ which yields
\begin{equation} \label{e:cbb}
 \Fitt_{\Z_p[G]^-}(\Cl^T(H)^-) \subset \Fitt_{\Z_p[G]^-}(\overline{B}). \end{equation}

The Galois representations used in the construction of $\kappa$ are {\em ordinary} at all primes dividing $p$.  Theorems of Hida and Wiles  precisely describe the shape of ordinary representations when restricted to the decomposition groups at these primes.  Using this we are able to relate the module $B(I_p)$ to the element $x \in \Z_p[G]^-$ and prove:
\begin{equation}
\label{e:fittinginclude}
 (x)\Fitt_{\Z_p[G]^-}(\overline{B}) \subset (x \Theta_{S_\infty, T}). 
 \end{equation}

Since $x$ is a non-zerodivisor, it can be canceled from the left and right sides of this inclusion; combining with (\ref{e:cbb}), we obtain the desired result
\[  \Fitt_{\Z_p[G]^-}(\Cl^T(H)^-) \subset (\Theta_{S_\infty, T}).
\]

Our calculation of Fitting ideals leading to the first inclusion in (\ref{e:fittinginclude}) is new (see Theorem~\ref{t:fitbp}), as is the idea to produce ``extra congruences" yielding the second inclusion;  we expect this technique to have further applications toward Bloch--Kato type results proved using Ribet's method.

This concludes our summary of the proof of Theorem~\ref{t:sk} in the case that $H/F$ is unramified at all finite primes.  It is  worth reflecting that the inclusion 
(\ref{e:fct}) deduced from Ribet's method is the {\em reverse} of that required by the Strong Brumer--Stark conjecture.  Combining this inclusion with the analytic argument of (\ref{e:cnf}) enables us to deduce that the inclusion is an equality, and hence to conclude that the desired  inclusion holds.  In the general case of conductor greater than 1, this equality does not hold and hence both sides must be replaced by generalizations.  

\bigskip

The paper is organized as follows. In \S\ref{s:algebraic}, we present some analytic and algebraic preliminaries, including some results on Fitting ideals.
In \S\ref{s:selmer} we recall the $\Z[G]$-module $\Sel_{\Sigma}^{\Sigma'}(H)$ of Burns, Kurihara, and Sano 
that plays the role of $\Cl^T(H)^\vee$ in the discussion above in the more general context
when there exist ramified primes for $H/F$.  Here $\Sigma$ and $\Sigma'$ denote arbitrary finite disjoint sets of places of $F$ such that $\Sigma \supset S_\infty$.  In order to prove the $p$-part of Theorem~\ref{t:sk} for an odd prime $p$ (that is, after tensoring with $\Z_p$), we choose the sets
\begin{equation} \Sigma = S_\infty \cup \{v \in S_{\ram}, v \mid p\}, \qquad \Sigma' = T \cup \{v \in S_{\ram}, v \nmid p\}. \label{e:ssi} \end{equation} 
The keystone result proven over the course of the paper using group ring valued Hilbert modular forms over $\Z_p[G]$, from which all our previously stated theorems are deduced, is the following.
 \begin{theorem} \label{t:main} The $\Z_p[G]^-$-module $\Sel_{\Sigma}^{\Sigma'}(H)_p^- = (\Sel_{\Sigma}^{\Sigma'}(H) \otimes_{\Z} \Z_p)^-$ is quadratically presented and we have
 \begin{equation} \label{e:isub}
  \Fitt_{\Z_p[G]^-}(\Sel_{\Sigma}^{\Sigma'}(H)_p^-) = (\Theta_{\Sigma, \Sigma'}^\#). 
  \end{equation}
   \end{theorem}
The module $\Sel_{\Sigma}^{\Sigma'}(H)_p$ plays an important role in our argument since $\Cl^T(H)^\vee_p$ is in general  not quadratically presented.  Also in \S\ref{s:selmer}, we deduce a partial result towards Kurihara's conjecture for the Fitting ideal of $\Cl^T(H)^{\vee, -}_p$, namely, we compute
$\Fitt_{\Z_p[G]^-} \Sel_{\Sigma}^{T}(H)_p^{-}$ assuming Theorem~\ref{t:main}.  We show that this partial result is strong enough to imply  Strong Brumer--Stark  (Theorem~\ref{t:sbs}). 
The key point here is that $\Cl^T(H)^{\vee, -}_p$ is a quotient of $\Sel_{\Sigma}^{T}(H)_p^{-}$.
 We conclude \S\ref{s:selmer} by deducing the prime-to-2 part of Rubin's conjecture, i.e.\ Theorem~\ref{t:r}.

In \S\ref{s:sad} we make some technical modifications of the smoothing and depletion sets $\Sigma, \Sigma'$
that assist in later arguments.
 In \S\ref{s:die} we prove an analogue of the discussion surrounding (\ref{e:fct})--(\ref{e:cnf}) above to show that  an inclusion in (\ref{e:isub}) for all $H/F$ implies an equality---see Theorem~\ref{t:include} for a precise statement.
 This result is significantly more complicated than the situation in (\ref{e:cnf}).

In \S\ref{s:properties} we describe  a $\Z[G]$-module $\nabla_{\Sigma}^{\Sigma'}(H)$ that was essentially defined previously by Ritter and Weiss \cite{rw}.  Our contribution is the introduction of the smoothing set $\Sigma'$.  In \S\ref{s:properties} we state the salient properties of $\nabla_{\Sigma}^{\Sigma'}(H)$. The actual construction of  $\nabla_{\Sigma}^{\Sigma'}(H)$  and the proof of these properties is postponed to Appendix~\ref{s:rw}.
Under the appropriate assumptions, the $\Z[G]$-module $\nabla_{\Sigma}^{\Sigma'}(H)$ is locally quadratically presented and is a transpose of the module $\Sel_{\Sigma}^{\Sigma'}(H)$ in the sense of Jannsen \cite{jannsen}.  We remark that smoothing at $\Sigma'$ is essential toward deducing the quadratic presentation property.
It is likely that  $\nabla_{\Sigma}^{\Sigma'}(H)$ is isomorphic to the canonical transpose $\Sel_{\Sigma}^{\Sigma'}(H)^{\tr}$ defined by Burns--Kurihara--Sano in \cite{bks}, though we have not tried to prove this.

 Burns--Kurihara--Sano study their Selmer group and its transpose in detail under the assumption  
$\Sigma \supset S_{\ram}$.  For us, it is essential to relax this assumption as in (\ref{e:ssi}).  We also give an interpretation of the minus part $\nabla_{\Sigma}^{\Sigma'}(H)^-$ in terms of Galois cohomology (Lemma~\ref{l:ext}) that does not appear explicitly in  prior works.  However the essential content of this lemma (indeed, our proof of it) can be gleaned from the calculations of Ritter--Weiss.

 The remainder of the paper, which uses Ribet's method applied to group ring valued Hilbert modular forms, proves the inclusion that is the supposition of Theorem~\ref{t:include}. In \S\ref{s:hmf} we set our notations for classical Hilbert modular forms.  In \S\ref{s:hmf}--\ref{s:cusp} we define group ring valued Hilbert modular forms and construct a cusp form congruent to an Eisenstein series in this context.  We use this construction to define a homomorphism on the Hecke algebra generalizing (\ref{e:phix}).
 Our construction of a cusp form is a strong refinement of Wiles' construction of cusp forms in \cite{wiles} in two ways: (i) we work over a group ring rather than character by character, and more importantly (ii) we construct ``extra congruences" beyond those predicted by the Stickelberger element using trivial zeroes as discussed in (\ref{e:phix}) above.  One key difference that allows us to produce these congruences is that we calculate the constant terms of  relevant Eisenstein series at all cusps, rather than focusing exclusively on the cusps above $\infty$.  These calculations are contained in \cite{dka}.
 Furthermore, we apply  important results of Silliman that show the existence of group ring valued modular forms with certain prescribed constant terms \cite{dks}.
 
   We conclude in \S\ref{s:galrep} by exploiting the Galois representations associated to Hilbert modular cusp forms in order to construct the cohomology class $\kappa$ of (\ref{e:kappadef0}), and using this construction to deduce the desired inclusion of Fitting ideals.  Our new calculation of the Fitting ideal in Theorem~\ref{t:fitbp}  should have future applications; it is inspired by the calculation of Gross's regulator in our previous work \cite{dkv}*{\S5}.
 
 As mentioned above, Appendix~\ref{s:rw} contains the construction of the Ritter--Weiss modules 
 $\nabla_{\Sigma}^{\Sigma'}(H)$ and  proofs of their key properties.  Appendix~\ref{s:kurihara} contains the proof of Kurihara's Conjecture (Theorem~\ref{t:main}), bootstrapping from the partial result proved in \S\ref{s:selmer} and mentioned above (i.e.\ the computation of $\Fitt_{\Z_p[G]^-} \Sel_{\Sigma}^{T}(H)_p^{-}$).  This proof is included in an appendix because it requires the full details of the construction of $\nabla_{\Sigma}^{\Sigma'}(H)$, and not just the properties listed in \S\ref{s:properties}.
 
 \subsection{Acknowledgements}
 
 It is a pleasure to recognize the significant influence of the works of David Burns, Cornelius Greither, Cristian Popescu, Masato Kurihara, Kenneth Ribet, J\"urgen Ritter, Takamichi Sano,  Alfred Weiss, and Andrew Wiles on this project.  We would also like to thank David Burns, Henri Darmon, Henri Johnston, Masato Kurihara, Andreas Nickel, Cristian Popescu, Takamichi Sano, Jesse Silliman, and Jiuya Wang for very helpful discussions.  In particular we are indebted to Sano for introducing us to the paper \cite{rw} of Ritter--Weiss.
 
 The first named author was supported by NSF grants DMS-1600943 and DMS-1901939 while working on this project. The second author is supported by DST-SERB grant SB/SJF/2020-21/11, SERB SUPRA grant SPR/2019/000422 and SERB MATRICS grant MTR/2020/000215. 
 
\section{Algebraic and Analytic Preliminaries} \label{s:algebraic}

Throughout this paper we work with a totally real field $F$ and a finite abelian CM extension $H$.  We let $G = \Gal(H/F)$.  In this section we record some basic algebraic and analytic facts that will be used in the sequel.

\subsection{Analytic Class Number Formula}

Let $H^+$ denote the maximal totally real subfield of the CM field $H$, and let $\epsilon$ denote the nontrivial character of $\Gal(H/H^+)$.

\begin{lemma} \label{l:cnf}
We have $L_{S_\infty, T}(H/H^+, \epsilon, 0) \in \Z$ and
\begin{align}
\#\Cl^T(H)^{-} & \doteq L_{S_\infty, T}(H/H^+, \epsilon, 0) \label{e:cll} \\
& = \prod_{\psi \in \hat{G} \! \text{ odd}} L_{S_\infty, T}(H/F, \psi, 0). \label{e:cll2}
\end{align}
where $\doteq$ denotes equality up to a power of $2$.
\end{lemma}

\begin{proof} This result is well-known, but we have not found a precise reference for it; the results (\ref{e:cll})--(\ref{e:cll2}) are proven in \cite{nickel}*{Proposition 2} without the $T$-smoothing.

 Since $\epsilon$ is $\pm 1$-valued, $L_{S_\infty, T}(H/H^+, \epsilon, 0)$ is rational by Klingen \cite{klingen} or Siegel \cite{siegel}.  It is actually an integer by Cassou-Nogu\`es \cite{cn} or Deligne--Ribet \cite{dr}  because of our assumption on the set $T$ made in the introduction, since
 \[ L_{S_\infty, T}(H/H^+, \epsilon, 0) = L_{S_{\ram}(H/H^+), T}(H/H^+, \epsilon, 0). \]

To prove (\ref{e:cll}), we note
\begin{align} 
L_{S_\infty, T}(H/H^+, \epsilon, 0) &= \frac{\zeta_{H, S_\infty, T}^*(0)}{\zeta_{H^+, S_\infty, T}^*(0)} \label{e:zs1} \\
&= \frac{\#\Cl^T(H) R_T(H)}{\#\Cl^T(H^+) R_T(H^+)} \label{e:zs2} \\
&\doteq \#\Cl^T(H)^-. \label{e:zs3}
\end{align} 
In (\ref{e:zs1}), $\zeta_{H, S_\infty, T}^*(0)$ denotes the leading term of the zeta function at $s=0$ (both this zeta function and $\zeta_{H^+, S_\infty, T}$ have order \[ \rank(\cO_H^*) = \rank(\cO_{H^+}^*) = [H^+:\Q] - 1 \] at $s=0$).  Equation (\ref{e:zs2}) is simply the $T$-smoothed Dedekind class number formula expressed at $s=0$; see for instance \cite{pcmi}*{(16)}.  The ``up to 2-power" equality (\ref{e:zs3}) follows from the following:
\begin{itemize}
\item $\Cl^T(H) \otimes \Z[\frac{1}{2}] \cong (\Cl^T(H^+) \otimes \Z[\frac{1}{2}]) \oplus \Cl^T(H)^-.$
\item $R_T(H) = 2^{[H^+:\Q] - 1} R_T(H^+)$ since $\cO_{H, S_\infty, T}^*= \cO_{H^+, S_\infty,T}^*$ by property (\ref{e:drcond}).
\end{itemize}
Finally (\ref{e:cll2}) follows from (\ref{e:cll}) by the Artin formalism for $L$-functions, as 
\[\Ind_{G_{H^+}}^{G_F} \epsilon = \bigoplus_{\psi \in \hat{G} \! \text{ odd}} \psi. \]
\end{proof}

\subsection{Character group rings} \label{s:cgr}

We fix an odd prime $p$ and a finite extension $\cO$ of $\Z_p$ 
that contains all the values of all characters $G \longrightarrow \overline{\Q}_p^*$. 
There is an $\cO$-algebra embedding \[ \cO[G] \hookrightarrow \prod_{\psi \in \hat{G}} \cO_\psi, \qquad x \mapsto (\psi(x))_{\psi \in \hat{G}}. \]
 Here $\cO_\psi$ denotes the ring $\cO$ endowed with the $G$-action in which
$g \in G$ acts by multiplication by $\psi(g)$.
More generally, given any subset of characters $\Psi \subset \hat{G}$, we define $R_\Psi$ to be the image of 
\[ \cO[G] \longrightarrow \prod_{\psi \in \Psi} \cO_{\psi},  \qquad x \mapsto (\psi(x))_{\psi \in \Psi}. \]
The quotients $R_\Psi$ of $\cO[G]$ defined in this way will be referred to as {\em character group rings}.  Each $R_\Psi$ is a finite index subring of a finite product of DVRs.

Write $G = G_p \times G'$, where $G_p$ is the $p$-Sylow subgroup of $G$, and $G'$ is the subgroup of elements with prime-to-$p$ order.  The ring $\cO[G]$ decomposes as a product of local rings
$R_\chi = \cO[G_p]_\chi$ indexed by the characters $\chi \in \hat{G'}$.  Here $\cO[G_p]_\chi$ denotes the $\cO$-algebra $\cO[G_p]$ endowed with the $G$-action in which $g \in G$ acts by $\chi(g) \overline{g}$, where $\overline{g}$ denotes the image of $g$ under the canonical projection $G \longrightarrow G_p$.
Each connected component $R_\chi$ of $\cO[G]$ is an example of  a character group ring, with associated set $\Psi =\{\psi\colon \psi|_{G'} = \chi\}$.  The characters $\psi \in \Psi$ are said to {\em belong to} $\chi$.

\begin{lemma} Let $I \subset G$ be a subgroup and let \[ \N I = \sum_{\sigma \in I} \sigma \in \cO[G]. \] The quotient $\cO[G]/\N I$ is  a character group ring.  More precisely, $\cO[G]/\N I \cong R_{\Psi}$ where
$\Psi = \{\psi \in \hat{G}\colon \psi(I) \neq 1 \}$.
\end{lemma}

\begin{proof} Consider the canonical surjective $\cO$-algebra homomorphism 
\[ \begin{tikzcd}\alpha\colon \cO[G] \ar[r,two heads] & R_{\Psi}. \end{tikzcd} \] It is clear that $\N I$ lies in the kernel of $\alpha$, since $\psi(\N I) = 0$ if $\psi(I) \neq 1$.  
Conversely if $x \in \ker \alpha$, then for all $g \in I$ we see that $\psi(gx) = \psi(x)$ for all $\psi \in \hat{G}$.  This is clear if $\psi(g) = 1$, and follows from $x \in \ker \alpha$ if $\psi(g) \neq 1$.  Hence $gx = x$ for all $g \in I$, which implies that $x \in (\N I)$.
\end{proof}

\begin{corollary} \label{c:rchi} Let $\chi \in \hat{G'}$ and let $I \subset G_p$.  Then $R_\chi / \N I \cong R_\Psi$, where \[ \Psi =\{\psi \in \hat{G} \colon \psi|_{G'} = \chi, \psi(I) \neq 1\}. \]
In particular, $R_\chi / \N I$ can be expressed as a finite index subring of a product of DVRs.
\end{corollary}

\subsection{Fitting ideals} \label{s:fitting}

In this section we collect some results---presumably well-known---about Fitting ideals. Let $R$ be a commutative ring.  An $R$-module $M$ is called {\em quadratically presented} over $R$ if there exists a positive integer $m$ and an exact sequence
\[ \begin{tikzcd}
  R^m \ar[r,"\varphi"] & 
 R^m \ar[r] & N \ar[r] & 0.
 \end{tikzcd} \]
In this case, $\Fitt_R(N)$ is principal and generated by the determinant of the map $\varphi$.

\begin{lemma} \label{l:size}  Let $B$ be a finite index subring of a finite product of PIDs (such as any character group ring $R_\Psi$ associated to a subset $\Psi \subset \hat{G}$).
Let $N$ be a quadratically presented $B$-module such that $\Fitt_B(N) = (x)$ for some non-zerodivisor $x \in B$.  Suppose that $B/(x)$ is finite.  Then $N$ is finite and \[ \#N = \#B/(x). \]
\end{lemma}

\begin{proof}  Let $A$ be an $m \times m$ matrix representing the relations among the generators of $N$, so $N \cong B^m/A \cdot B^m$ and $\Fitt_B(N) = ( \det(A)) = (x)$. We must show that $B^m/A \cdot B^m$ is finite and $\#(B^m/A \cdot B^m) = \#B/\det(A)$.  This result is well-known for PIDs. Indeed, \[ N \cong B^m/A \cdot B^m \cong \bigoplus_{i=1}^m B/x_i B, \] where the Smith normal form of $A$ is  the diagonal matrix with diagonal entries $x_1, \ldots, x_m$. Then $(\det(A)) = (x) = (\prod_{i=1}^m x_i)$.  Each $B/x_iB$ is finite as it is a quotient of $B/xB$, and furthermore
\[ \#N = \prod_{i=1}^m \#B/x_i B = \#B/(x). \]
  From this, we can deduce the result more generally when $B$ is a finite index subring of a product of PIDs.
Indeed, it is clear that if the result holds for two rings $B, B'$, then it holds for $B\times B'$, since both sides of the desired equality factor as a product over the corresponding terms for $B$ and $B'$.  

Furthermore, if the result holds for a ring $B'$, then it holds for a finite index subring $B \subset B'$ as we now show.  We see that
\begin{equation} \label{e:rrp1}
 \frac{\#B'/\det(A)B'}{\#B/\det(A)B} = \frac{\#B'/B}{\#\det(A)B'/\#\det(A)B} = 1 
 \end{equation}
since multiplication by the non-zerodivisor $\det(A)$ is an isomorphism between the space in the numerator and in the denominator.
Similarly, one sees that
\begin{equation} \label{e:rrp2}
 \frac{\# (B')^m/A \cdot (B')^m}{\# B^m/A \cdot B^m} = \frac{\# (B')^m/B^m}{\# A \cdot (B')^m/ A \cdot B^m} = 1
 \end{equation}
since multiplication by $A$ induces an isomorphism between the space in the numerator and in the denominator in the middle of the equation.  
Only injectivity of this map is not obvious; for this, note that if $Av \in AB^m$ for some $v \in (B')^m$, then multiplying by the adjugate of $A$ we obtain
$\det(A)v \in \det(A) B^m$, whence $v \in B^m$ since $\det(A)$ is a non-zerodivisor.

Equations (\ref{e:rrp1}) and (\ref{e:rrp2}) imply that the lemma holds for any finite index subring $B$ of a ring $B'$ for which it holds; this gives the result.
\end{proof}

\begin{lemma} \label{l:size2}
Let $\Psi \subset \hat{G}$ and let $R_\Psi$ denote the associated character group ring over $\cO$. Let $x \in R_\Psi$ be a non-zerodivisor.  Then $\#R_\Psi/(x) =  \#\cO/(\prod_{\psi \in  \Psi}\psi(x))$.
\end{lemma} 

\begin{proof}
We proceed as in the proof of the previous lemma.  There is an injection 
\[ R_\Psi \hookrightarrow \cO_{\Psi}= \prod_{\psi \in \Psi} \cO, \qquad y \mapsto (\psi(y))_{\psi \in \Psi}\] with image of finite index.   
Then
\[ \#(\cO_\Psi / R_\Psi)= \#(x\cO_\Psi / xR_\Psi) \]
since multiplication by $x$ is an isomorphism between the two quotients.
It follows that
\[ \#(R_\Psi/  xR_\Psi) = \#(\cO_\Psi / x\cO_\Psi) = \prod_{\psi \in \Psi} \#(\cO/\psi(x)), \]
where the last equality holds since $\cO_\Psi$ is a product ring.  The result follows.
\end{proof}

We can now describe the ``elementary argument" mentioned in the introduction to show that (\ref{e:cnf}) implies that the inclusion (\ref{e:fct}) is an equality.  We work over $\cO[G]^-$.  In the case that $H/F$ is unramified at all finite primes, one can show that $\Cl^T(H)^{-}_\cO$, defined as $\Cl^T(H)^{-} \otimes {\cO}$, is quadratically presented as a module over $\cO[G]^-$. Therefore the inclusion (\ref{e:fct}) implies that
\[ \Fitt_{\cO[G]^-}(\Cl^T(H)^{-}_{\cO}) = (x \cdot \Theta_{S_\infty, T}) \]
for some $x \in \cO[G]^-$. Hence by Lemmas~\ref{l:size} and \ref{l:size2}, we have
\[ \# \Cl^T(H)^{-}_\cO = \#\cO[G]^-/(x  \Theta_{S_\infty, T}) = \#\cO/\prod_{\psi \text{ odd}} \psi(x  \Theta_{S_\infty, T}). \]
Therefore (\ref{e:cnf}) implies that $\psi(x) \in \cO^*$ for all $\psi$.  This implies that $x \in (\cO[G]^-)^*$, yielding the desired result.

We conclude with two more standard lemmas on Fitting ideals.

\begin{lemma} \label{l:fittmult} Let $R$ be a commutative ring, $C$ a quadratically presented $R$-module, and
\[ \begin{tikzcd}
 0 \ar[r] &  A \ar[r] & B \ar[r] & C \ar[r] & 0 
 \end{tikzcd}
  \]
a short exact sequence of $R$-modules.  Then
\[ \Fitt_R(B) = \Fitt_R(A)\Fitt_R(C). \]
Furthermore, if $A$ and $C$ are both quadratically presented, then $B$ is as well.
\end{lemma}
See~\cite{northcott}*{Theorem 22} for a proof.
\begin{lemma}\label{l:abab} Let $R$ be a commutative ring, and $B, B'$ two quadratically presented $R$-modules fitting into exact sequences\[
\begin{tikzcd}[row sep=tiny]
 0 \ar[r] & A \ar[r] & B \ar[r] & C \ar[r] & 0, \\
 0 \ar[r] & A' \ar[r] & B' \ar[r] & C \ar[r] & 0.
\end{tikzcd} \]
of $R$-modules.  Then \[ \Fitt_R(A) \Fitt_R(B') = \Fitt_R(A') \Fitt_R(B). \]
\end{lemma}

\begin{proof}
Let $M$ denote the fiber product of $B$ and $B'$ over $C$, i.e.\ the $R$-module of ordered pairs $(b,b')$ such that
$b$ and $b'$ have the same image in $C$.  Projection onto the first and second components yields two short exact sequences
\[
\begin{tikzcd}[row sep=tiny]
 0 \ar[r] & A' \ar[r] & M \ar[r] & B \ar[r] & 0, \\
 0 \ar[r] & A \ar[r] & M \ar[r] & B' \ar[r] & 0.
\end{tikzcd}
\]
Computing $\Fitt_R(M)$ in two ways using these exact sequences and Lemma~\ref{l:fittmult} yields the desired result.
\end{proof}

For a more general statement see \cite{kataoka}*{Theorem 2.6}.

\section{Main Results} \label{s:selmer}

\subsection{The Selmer module of Burns--Kurihara--Sano} \label{s:bks}

We recall the definition of the Selmer module defined by Burns--Kurihara--Sano in \cite{bks} and studied further by Burns in~\cite{burns}.  This $G$-module will play a central role in this paper.
For this, we fix finite disjoint sets of places $\Sigma, \Sigma'$ of $F$ such that $\Sigma \supset S_\infty$ and $\Sigma'$ satisfies condition (\ref{e:drcond}) from the introduction.  Let $H_{\Sigma'}^*$ denote the subgroup of $x \in H^*$ such 
that $\ord_w(x - 1) > 0$ for each prime $w \in \Sigma'_H$, where this latter set denotes the set of primes of $H$ lying above those in $\Sigma'$.  Define
\begin{equation} \label{e:selmerdef}
 \Sel_\Sigma^{\Sigma'}(H) = \Hom_{\Z}(H_{\Sigma'}^*, \Z)/\prod_{w \not\in \Sigma_H \cup \Sigma'_H} \Z \end{equation}
where the product ranges over the  primes $w \not \in \Sigma_H \cup \Sigma_H'$, and the implicit map sends a tuple $(x_w)$ to the function $\sum_w x_w \ord_w$.  As usual we give $ \Sel_\Sigma^{\Sigma'}(H)$ the contragredient $G$-action $(g \varphi)(x) = \varphi(g^{-1}x)$.

Let $Y_{H, \Sigma}$ denote the free abelian group on the places of $H$ above $\Sigma$, endowed with its canonical $G$-action. 

\begin{lemma} \label{l:ysc}
There is a canonical short exact sequence of $\Z[G]^-$-modules
\[ \begin{tikzcd}
0 \ar[r] &  Y_{H,\Sigma}^{-} \ar[r] & \Sel_\Sigma^{\Sigma'}(H)^{-} \ar[r] & \Cl^{\Sigma'}(H)^{\vee, -} \ar[r] & 0.
\end{tikzcd}
 \]
\end{lemma}

\begin{proof}  We have a canonical short exact sequence
\[ \begin{tikzcd}
0 \ar[r] & Y_{H,\Sigma \setminus S_\infty} \ar[r] & \Sel_{\Sigma}^{\Sigma'}(H) \ar[r] & \Sel_{S_\infty}^{\Sigma'}(H) \ar[r] & 0,
\end{tikzcd}
 \]
where the first nontrivial arrow is induced by $w \mapsto \ord_w$.  Note that $Y_{H,S_\infty}^- = 0$. To prove the result we must show that
\begin{equation} \label{e:selt}
\Sel^{\Sigma'}_{S_\infty}(H)^- \cong \Cl^{\Sigma'}(H)^{\vee, -}.
\end{equation}
  Yet the sequence (5) in \cite{burns} for $\Sigma = S_\infty$  reads
\[ \begin{tikzcd}
 0 \ar[r] & \Cl^{\Sigma'}(H)^{\vee} \ar[r] & \Sel^{\Sigma'}_{S_\infty}(H) \ar[r] & \Hom_\Z(\cO_{H, S_\infty, \Sigma'}^*, \Z) \ar[r] & 0. 
 \end{tikzcd} \]
Since $H$ is a CM field, $(\cO_{H, S_\infty, \Sigma'}^*)^-$ is trivial, yielding (\ref{e:selt}).  The result follows.
\end{proof}

It is convenient to provide an alternate presentation of $\Sel_\Sigma^{\Sigma'}(H)$ as follows. 
Let $S'$ be any finite set of places of $F$ containing $\Sigma$ and disjoint from $\Sigma'$.  Assume that $S'$ is chosen such that the class group $\Cl_{S'}^{\Sigma'}(H)$ is trivial.
As shown in \cite{bks}*{equation (12)}, there is a canonical isomorphism
\begin{equation} \label{e:sels}
 \Sel_\Sigma^{\Sigma'}(H) \cong \Hom_{\Z}(\cO_{H, S', \Sigma'}^*, \Z) /\prod_{w \in S'_H - \Sigma_H} \Z,
\end{equation}
with the implicit map as in (\ref{e:selmerdef}).

As a final note in this section, we show that the Fitting ideal of  $\Sel_\Sigma^{\Sigma'}(H)$ vanishes on any non-identity component $\psi$ with a trivial zero.  More precisely, let $\psi \in \hat{G}$, $\psi \neq 1$, such that
$\psi(G_v) = 1$ for some $v \in \Sigma$.  Here $G_v \subset G$ denotes the decomposition group at $v$.  Writing \[ \Sel_{\Sigma}^{\Sigma'}(H)_\psi = \Sel_{\Sigma}^{\Sigma'}(H) \otimes_{\Z[G]} \cO_\psi \]  with $\cO_\psi$ as in \S\ref{s:cgr}, we  claim that $\Fitt_{\cO_\psi}(\Sel_\Sigma^{\Sigma'}(H)_\psi) = 0$.
For this,  it suffices to show that the finitely generated $\cO_\psi$-module $\Sel_\Sigma^{\Sigma'}(H)_\psi$ is infinite. Thus it is enough to show that $\Sel_{\Sigma}^{\Sigma'}(H)_{\psi} \otimes_{\cO_{\psi}} K$, or equivalently $(\Sel_{\Sigma}^{\Sigma'}(H) \otimes K)^{\psi}$, is non-zero. By Lemma~\ref{l:ysc} above and as $\psi \neq 1$, we have
\[ (\Sel_{\Sigma}^{\Sigma'}(H) \otimes K)^\psi  \cong (Y_{H,\Sigma} \otimes K)^\psi \supset (Y_{H,\{v\}} \otimes K)^\psi = (\Ind^G_{G_v} K)^\psi \cong K
\]
 by Frobenius reciprocity because $\psi(G_v) = 1$.  The desired result $\Fitt_{\cO_\psi}(\Sel_\Sigma^{\Sigma'}(H)_\psi) = 0$ follows.
 
 \begin{lemma} \label{l:trivial} Let $\Psi \subset \hat{G}$ with $1 \not\in \Psi$ and let $R_\Psi$ denote the associated character group ring.  Suppose that for each $\psi \in \Psi$, there exists $v \in \Sigma$ such that $\psi(G_v) = 1$.  Then
 \[ \Fitt_{R_\Psi}( \Sel_\Sigma^{\Sigma'}(H) \otimes_{\Z[G]} R_\Psi) = 0 = \Theta_{\Sigma, \Sigma'}^{\#} R_\Psi. \]
 \end{lemma}
 \begin{proof} The first equality follows immediately from the fact that $\Fitt_{\cO_\psi}(\Sel_\Sigma^{\Sigma'}(H)_\psi) = 0$ since Fitting ideals are functorial with respect to quotients.
 Similarly the second equality follows since for all $\psi \in \Psi$ we have
 \[ \psi(\Theta_{\Sigma, \Sigma'}^{\#}) = L_{\Sigma, \Sigma'}(\psi, 0) = (1 - \psi(v))L_{\Sigma-\{v\}, \Sigma'}(\psi, 0) = 0. 
 \]
  \end{proof}
 
 \subsection{Keystone Result} \label{s:fmr}
 
Recall that $S_{\ram}$ denotes the set of finite primes of $F$ that are ramified in $H/F$.  As in the introduction, let $T$ denote a finite set of primes of $F$ that are unramified in $H$ and such that $T$ satisfies the condition (\ref{e:drcond}).
Let
\begin{equation}\label{e:sigmadef}
\begin{split}
\Sigma &=  \{v \in S_{\ram}\colon v \mid p\} \cup S_\infty,  \\
\Sigma' &= \{ v \in S_{\ram}\colon v \nmid p\} \cup T.
\end{split}
\end{equation}
 In other words, we  transfer the ramified primes not above $p$ from the depletion set to the smoothing set.
 The theorem whose proof occupies most of the paper, and from which all other results are deduced, is the following.
 \begin{theorem} \label{t:main0} The $\Z_p[G]^-$-module $\Sel_{\Sigma}^{\Sigma'}(H)_p^- = (\Sel_{\Sigma}^{\Sigma'}(H) \otimes_{\Z} \Z_p)^-$ is quadratically presented and we have
 \[ \Fitt_{\Z_p[G]^-}(\Sel_{\Sigma}^{\Sigma'}(H)_p^-) = (\Theta_{\Sigma, \Sigma'}^\#). \]
 \end{theorem}
 Implicit in the statement of Theorem~\ref{t:main0} is that $\Theta_{\Sigma, \Sigma'}^\# \in \Z_p[G]$, which follows from a lemma of Kurihara (see Lemma~\ref{l:sku} and Remark~\ref{r:sku} below).  
 
\subsection{Strong Brumer--Stark and Kurihara's Conjecture} \label{s:sk}

In Theorem~\ref{t:sk} we stated Kurihara's formula for the Fitting ideal of the $\Z[G]^-$module $\Cl^T(H)^{\vee,-}$, which he conjectured in \cite{kuriharaunp}  (see also \cite{greither}).  The following lemma shows that the statement is well-formed.

\begin{lemma}[Kurihara {\cite[Proposition 3.1]{kuriharaunp}}] \label{l:sku} $\KS^T(H/F)$ is contained in $\Z[G]$ and hence is an ideal of this ring. 
\end{lemma}
\begin{proof} The key input for this result is the  integrality statement (\ref{e:dr}) of Deligne--Ribet and Cassou-Nogu\`es.  For $S_\infty \subset J \subset S_\infty \cup S_{\ram}$, we write $\overline{J} = S_{\ram} \setminus J$.
Note that 
\[ \KS^T(H/F) = \left(\prod_{v \in \overline{J}} \N I_v \cdot  (\Theta_{J, T}^{H/F})^{\#} \colon  S_\infty \subset J \subset S_\infty \cup S_{\ram} \right).
 \]

Write  $H^{\overline{J}}$ for the maximal subextension of $H$ unramified at all primes in $\overline{J}$.  Then $H^{\overline{J}}$ is the subfield of $H$ fixed by the subgroup of $G$ generated by $I_v$ for all $v \in \overline{J}$. Multiplication by $\prod_{v \in \overline{J}} \N I_v$ defines a homomorphism 
\[
\Z[\Gal(H^{\overline{J}}/F)] \longrightarrow \Z[G],
\]
and we have \begin{equation} \label{e:nuj}
\prod_{v \in \overline{J}} \N I_v \cdot (\Theta_{J, T}^{H^{\overline{J}}/F})^{\#} = \prod_{v \in \overline{J}} \N I_v \cdot (\Theta_{J, T}^{H/F})^{\#}. \end{equation} 
Therefore
\begin{equation} \label{e:skalt}
 \KS^T(H/F) = \left(\prod_{v \in \overline{J}} \N I_v \cdot  (\Theta_{J, T}^{H^{\overline{J}}/F})^{\#} \colon  S_\infty \subset J \subset S_\infty \cup S_{\ram} \right).
 \end{equation}

By  (\ref{e:dr}), the element 
$(\Theta_{J, T}^{H^{\overline{J}}/F})^{\#}$ belongs to $\Z[\Gal(H^{\overline{J}}/F)]^-$ and hence (\ref{e:nuj}) lies in $\Z[G]$.
The result follows.
\end{proof}

The following is our main result.
\begin{theorem}[Conjecture of Kurihara] \label{t:sk0}
We have 
\[ \Fitt_{\Z[G]^-}(\Cl^T(H)^{\vee, -})  = \KS^T(H/F)^-. \]
\end{theorem}
As noted in (\ref{e:tstd}), Theorem~\ref{t:sk0} implies Strong Brumer--Stark  (Theorem~\ref{t:sbs}).
In this section, we assume Theorem~\ref{t:main0} and prove a partial result toward Theorem~\ref{t:sk0} that still yields Strong Brumer--Stark.  In Appendix~\ref{s:kurihara} we bootstrap from this partial result to complete the proof of Theorem~\ref{t:sk0}.

For an odd prime $p$ we define the $p$-modified Sinnott--Kurihara ideal by
\[ \KS_p^T(H/F) = (\Theta^\#_{\Sigma,T}) \prod_{v \in S_{\ram},\ \! v\nmid p} \!\!\!\! ( \N I_v, 1 - \sigma_v e_v ) \subset \Z_p[G]\]
where $\Sigma$ is as in (\ref{e:sigmadef}).

\begin{remark}  \label{r:sku} The fact that $\KS_p^T(H/F) \subset \Z_p[G]$ follows directly from Lemma~\ref{l:sku}, since 
\[ \Theta^\#_{\Sigma,T} = \Theta^\#_{S_\infty, T} \prod_{v \in S_{\ram},\ \! v \mid p} (1 - \sigma_v e_v). \]
Moreover, for $v \nmid p$ the $p$-Sylow subgroup of $I_v$ is a quotient of 
$(\cO/v)^*$ and hence $\#I_v$ divides $\N v - 1$ in $\Z_p$.  Therefore for $\Sigma, \Sigma'$ as in Theorem~\ref{t:main0}, 
\begin{align*}
 \Theta^\#_{\Sigma,\Sigma'} &= \Theta^\#_{\Sigma, T} \prod_{v \in S_{\ram}, v \nmid p} (1 - \sigma_v e_v \N v)  \\
 &= \Theta^\#_{\Sigma, T} \prod_{v \in S_{\ram}, v \nmid p} \left[(1 - \sigma_v e_v ) + \left(\sigma_v \cdot \N I_v \frac{1- \N v}{\#I_v} \right) 
 \right]\\
 & \in \KS_p^T(H/F) \subset \Z_p[G].
\end{align*}

\end{remark}

The partial result toward Theorem~\ref{t:sk0} that we prove in this section is the following.

\begin{theorem} \label{t:ks} For every odd prime $p$ we have
\[ \Fitt_{\Z_p[G]^-}(\Sel_{\Sigma}^T(H)_p^-) = \KS_p^T(H/F)^-. \]
\end{theorem}

Before discussing the proof of Theorem~\ref{t:ks}, let us note that it is strong enough to imply Strong Brumer--Stark.

\begin{corollary} \label{t:sbsc} 
The Strong Brumer--Stark Conjecture is true:
\[ \Theta_{S,T}^\#(H/F) \in \Fitt_{\Z[G]^-}(\Cl^T(H)^{\vee, -}). \]
\end{corollary}

\begin{proof}[Proof of Corollary~\ref{t:sbsc}]  It suffices to work prime by prime, i.e. to show that
\[
 \Theta_{S,T}^\#(H/F) \in \Fitt_{\Z_p[G]^-}(\Cl^T(H)_p^{\vee, -}) \]
 for each odd prime $p$.  By Lemma~\ref{l:ysc} there is a surjection $\Sel_{\Sigma}^T(H)^{-} \longtwoheadrightarrow \Cl^T(H)^{\vee, -}$ that together with
Theorem~\ref{t:ks} implies 
\begin{equation} \label{e:fitsupsk}
 \Fitt_{\Z_p[G]^-}(\Cl^T(H)_p^{\vee, -}) \supset \KS_p^T(H/F). \end{equation}
Since \[\Theta_{S,T}^\# =  \Theta_{\Sigma, T}^\#  \prod_{v \in S_{\ram}, \ \! v\nmid p} \!\!\!\! (1 - \sigma_v e_v) \in  \KS_p^T(H/F), \]
the result follows.
\end{proof}

We now prove Theorem~\ref{t:ks} assuming our keystone result, Theorem~\ref{t:main0}.

\begin{proof}[Proof of Theorem~\ref{t:ks}]  First note that it suffices to prove the result after extending scalars to $\cO$ and then  projecting to the connected component $R = \cO[G_p]_\chi$ of $\cO[G]^-$ associated to each odd character $\chi$ of $G'$. 
Theorem~\ref{t:main0} yields
\begin{equation} \label{e:pmr2}
 \Fitt_R(\Sel_{\Sigma}^{\Sigma'}(H)_R) = (\Theta_{\Sigma, \Sigma'}^\#),
 \end{equation}
where the right side denotes the principal ideal of $R$ generated by the projection of the element $\Theta_{\Sigma, \Sigma'}^\# \in \cO[G]$ to $R$.

 To prove the theorem, we must demonstrate the effect of removing the primes in \[ S' = \{v \in S_{\ram}, v \nmid p\} \] from the superscript of the Selmer group in (\ref{e:pmr2}).  
For this we first consider  the short exact sequence of $\Z[G]^-$-modules
\begin{equation} \label{e:ctt}
\begin{tikzcd}
 0 \ar[r] &   \Cl^T(H)^{\vee, -} \ar[r] &  \Cl^{\Sigma'}(H)^{\vee, -} \ar[r] & \prod_{w \in S'_H} ((\cO_{H}/w)^*)^{\vee, -} \ar[r] & 0. 
\end{tikzcd}
 \end{equation}

For each $v \in S'$, note that the inertia group $I_v \subset G$ acts trivially on $(\cO_{H}/w)^*$.
Decompose $I_v$ as a product $I_v = I_{v,p} \times I_v'$ of its subgroups of $p$-power order elements and prime-to-$p$ order elements, respectively.  For $\tau \in I_v'$, the element $\tau - 1 \in \cO[G]$ has image $\chi(\tau) -1$ in $R$.  This is a unit if $\chi(\tau) \neq 1$, and is 0 if $\chi(\tau)=1$.  Since $\tau - 1$ kills $(\cO_{H}/w)^*$, it follows that the base extension of $(\cO_H/w)^*$ to $R$ is trivial unless $\chi(I_v') = 1$.  And in this latter case $\tau - 1$ has vanishing image in $R$ for $\tau \in I_v'$.

Next note that by class field theory, $I_{v,p}$ is a quotient of $(\cO_F/v)^*$ since $v \nmid p$.  Hence $I_{v,p}$ is cyclic, and $\N v \equiv 1 \pmod{\# I_{v,p}}$.
Let $\tau_v$ be a generator of $I_{v,p}$.  Fixing a prime $w$ of $H$ above $v$ and a generator $u$ for $(\cO_H/w)^*$ yields an isomorphism
\[ \Z[G_v]/(\tau_v-1, \sigma_v - \N v, \tau - 1: \tau \in I_v') \cong (\cO_H/w)^*, \qquad x \mapsto  u^x, \]
where $\sigma_v$ is any element representing the Frobenius $I_v$-coset in $G_v$.
Inducing from $G_v$ to $G$, taking duals, and projecting to the $R$-component yields:
\begin{equation} \label{e:presento}
  \prod_{w \in S'_H} ((\cO_{H}/w)^*)^{\vee}_R \cong \prod_{v \in S', \chi(I_v') = 1} R/(\tau_v-1, \sigma_v^{-1} - \N v). \end{equation}

Next consider the commutative diagram:
\[ \begin{tikzcd}
0 \ar[r] & Y_{H,\Sigma}^- \ar[r]  \ar[d, "\id"] & \Sel_{\Sigma}^T(H)^- \ar[r] \ar[d] & \Cl^T(H)^{\vee, -} \ar[r] \ar[d] & 0 \\
0 \ar[r] & Y_{H,\Sigma}^- \ar[r] & \Sel_{\Sigma}^{\Sigma'}(H)^- \ar[r] & \Cl^{\Sigma'}(H)^{\vee, -} \ar[r] & 0.
\end{tikzcd}
\]
The snake lemma in conjunction with (\ref{e:ctt}) 
yields
a short exact sequence
\begin{equation} \label{e:sst0}
\begin{tikzcd}
 0 \ar[r] &   \Sel_{\Sigma}^T(H)^-  \ar[r] & \Sel_{\Sigma}^{\Sigma'}(H)^- \ar[r] & \prod_{w \in S'_H} ((\cO_{H}/w)^*)^{\vee, -}  \ar[r] & 0. 
 \end{tikzcd}
\end{equation}
Applying (\ref{e:presento}), this may be written 
\begin{equation} \label{e:sst}
\begin{tikzcd}
 0 \ar[r] &   \Sel_{\Sigma}^T(H)^-  \ar[r] & \Sel_{\Sigma}^{\Sigma'}(H)^- \ar[r] & \displaystyle\prod_
 {\substack{v \in S' \\ \chi(I_v')=1}} R/(\tau_v-1, \sigma_v^{-1} - \N v) \ar[r] & 0. 
 \end{tikzcd}
\end{equation}

Consider for each $v \in S'$ such that $\chi(I_v')=1$ the short exact sequence:
\begin{equation} \label{e:cok}
\begin{tikzcd}
 0 \ar[r] &  R/(\N I_{v,p}, \sigma_v^{-1} - \N v)  \ar[r] & R/(\sigma_v^{-1} - \N v)  \ar[r] &  R/(\tau_v-1, \sigma_v^{-1} - \N v) \ar[r] & 0,
\end{tikzcd}
\end{equation}
where the first non-trivial arrow is multiplication by $\tau_v - 1$ and the next arrow is projection.  Only the injectivity of this multiplication is unclear.  Suppose $x(\tau_v - 1) = y(\sigma_v^{-1} - \N v)$ for $x, y \in R$.
Then $y(\sigma_v^{-1} - \N v)$ vanishes in $R/(\tau_v - 1) \cong \cO[G_p/I_{v,p}]$.  But $\sigma_v^{-1} - \N v$ is a non-zerodivisor in this group ring, and hence the image of $y$ in this ring vanishes, i.e.\ $y = (\tau_v - 1)y'$ for some $y' \in R$.
Then $x - y'(\sigma_v - \N v)$ is annihilated by $\tau_v - 1$ and hence is a multiple of $\N I_{v,p}$.  Thus $x \in (\N I_{v,p}, \sigma_v^{-1} - \N v)$, proving the desired injectivity.

  Applying Lemma~\ref{l:abab} to (\ref{e:sst}) and the product of   (\ref{e:cok}) over the appropriate $v$  yields
\begin{equation} \label{e:fsfs}
 \Fitt_{R}(\Sel_{\Sigma}^{T}(H)_R) \!\!\!\prod_{v \in S', \chi(I_v')=1} (\sigma^{-1}_v - \N v) = \Fitt_{R}(\Sel_{\Sigma}^{\Sigma'}(H)_R) \!\!\!\prod_{v \in S', \chi(I_v')=1}(\N I_{v, p}, \sigma_v^{-1} - \N v). 
\end{equation}
A key point is that the terms $\sigma^{-1}_v - \N v$ are non-zerodivisors and hence can be inverted in $\Frac(R)$. Note also that if $\chi(I_{v}') \neq 1$ then the projection of $I_v$ to $R$ vanishes
and hence $e_v = 0$ in $\Frac(R)$. 
In particular
\begin{align*}
\Theta_{\Sigma, \Sigma'}^\# &= \Theta_{\Sigma, T}^\# \prod_{v \in S'} (1 - \sigma_v e_v \N v)  \\
&= \Theta_{\Sigma, T}^\# \prod_{v \in S', \chi(I_v') = 1} (1 - \sigma_v e_v \N v).
\end{align*}
Furthermore, if $\chi(I_{v}') = 1$ then $\N I_v = (\#I_v')\N I_{v,p}$, and the integer $\#I_v'$ is a $p$-adic unit.  Also in this case $e_v = e_{v,p} e_{v}' = e_{v,p}$ in $\Frac(R)$, 
where $e_{v,p} = \N I_{v,p}/\#I_{v,p}$ and $e_v' = \N I_v' /\#I_v' = 1$.

Therefore, applying (\ref{e:pmr2}) to (\ref{e:fsfs})  yields:
\begin{align*}
 \Fitt_{R}(\Sel_{\Sigma}^T(H)_R) &= (\Theta_{\Sigma, \Sigma'}^\#)  \prod_{v \in S', \chi(I_v')=1} (\N I_{v,p}, \sigma_v^{-1} - \N v)(\sigma^{-1}_v - \N v)^{-1} \\
 &= (\Theta_{\Sigma, T}^\#) \prod_{v \in S', \chi(I_v')=1} (\N I_{v,p}, \sigma_v^{-1} - \N v) (1 - \sigma_v e_{v,p} \N v)(\sigma^{-1}_v - \N v)^{-1} \\
 &= (\Theta_{\Sigma, T}^\#) \prod_{v \in S', \chi(I_v')=1} (\N I_{v,p}, 1- \sigma_v e_{v,p} \N v).
 \end{align*}
Finally we note that for $v \in S', \chi(I_v')=1$, since $\N v \equiv 1 \pmod{\# I_{v,p}}$ we have  \begin{align*}
 (\N I_{v,p}, 1- \sigma_v e_{v,p} \N v) &=  (\N I_{v,p}, 1- \sigma_v e_{v,p}) \\
 &= (\N I_v, 1 - \sigma_v e_v). \end{align*} 
 To conclude the proof, we note that for $v \in S'$ such that $\chi(I_v') \neq 1$, we have
 \[ (\N I_v, 1 - \sigma_v e_v) = (1) \text{ in } R. \]
We have therefore proven that
\[  \Fitt_{R}(\Sel_{\Sigma}^T(H)_R) = (\Theta_{\Sigma, T}^\#) \prod_{v \in S'} (\N I_{v}, 1- \sigma_v e_{v} \N v), \]
which is the projection to $R$ of the desired result.
\end{proof}

\subsection{Rubin's Conjecture} \label{s:rubin}

In this section we prove that  Strong Brumer--Stark  implies Rubin's conjecture away from $2$, namely, we prove Theorem \ref{t:r}.
This result is known by the experts, but since only a dual version of this appears in the literature (see~\cite{popcm}*{Corollary 2.4}), we 
give a proof here.

\begin{lemma} \label{l:coker}
Let $R$ be a commutative ring and let $N \subset M$ be $R$-modules with $N$ finitely generated and $M$ finitely presented.  For each positive integer $r$, the ideal $\Fitt(M/N)$ annihilates the cokernel of the canonical map
\[ \bigwedge\nolimits^r_R N \longrightarrow  \bigwedge\nolimits^r_R M. \]
\end{lemma}

\begin{proof} We first reduce to the case that $M$ and $N$ are both finitely generated free $R$-modules.  By the assumptions on $M$ and $N$, we may fix a surjection $R^m \longrightarrow M$ and a finite presentation 
\[ \begin{tikzcd}
R^n \ar[r]  & R^m  \ar[r]  & M/N   \ar[r]  & 0
\end{tikzcd}
\]
This yields a commutative diagram 
\[
\begin{tikzcd}
 R^n \ar[r] \ar[d, dotted] & R^m \ar[r] \ar[d, two heads] & M/N \ar[d, equal] \\
N \ar[r] & M \ar[r] & M/N.
\end{tikzcd}
\]
The dotted arrow exists because $R^n$ is free. Using the right exactness of the exterior power functor we get a commutative diagram
\[
\begin{tikzcd}
\bigwedge^r_R R^n \ar[r] \ar[d] & \bigwedge^r_R R^m \ar[r, two heads] \ar[d, two heads] &
 C_2 \ar[d, two heads] \\ 
\bigwedge^r_R N \ar[r] & \bigwedge_R^r M \ar[r, two heads] & C_1.  
\end{tikzcd}
\]
Here $C_1$ and $C_2$ are cokernels of the obvious maps. It is also clear that the map $C_2 \longrightarrow C_1$ is surjective. Therefore it is enough to show that $\Fitt(M/N)$ annihilates $C_2$. Hence we may assume that $M \cong R^m$ and $N \cong R^n$ are both free $R$-modules. Without loss of generality we further assume that $n \geq m$. Let the map $R^n \longrightarrow R^m$ be given by an $m \times n$ matrix $A$. We fix an $m \times m$ submatrix, say $A'$ of $A$. We must show that $\det(A')$ annihilates $C_2$.

The map $\bigwedge^r_R R^n \longrightarrow \bigwedge^r_R R^m$ is given by the $r$th compound matrix $C_r(A)$---this is the $\binom{m}{r} \times \binom{n}{r}$ matrix whose entries are the $r \times r$ minors of $A$.       
Let $x \in \bigwedge^r_R R^m$. Denote by $\adj_r(A')$ the $r$th higher adjugate matrix of $A'$, so 
\[
 \adj_r(A') \cdot C_r(A') x = \det(A')x.
\]
Observe that  $C_r(A')$ is an $\binom{m}{r} \times \binom{m}{r}$ submatrix of $C_r(A)$ obtained by deleting $\binom{n}{r} - \binom{m}{r}$ columns. Let $\tilde{x}$ be the element of $\bigwedge^r_R R^n$ obtained from $x$ by inserting 0's in the entries corresponding to these deleted columns. Then $  C_r(A) \tilde{x} = C_r(A') x $, hence
\[
 \adj_r(A') \cdot C_r(A) \tilde{x} = \adj_r(A') \cdot C_r(A') x = \det(A')x.
\] 
This shows that $\det(A')x$ belongs to the image of $\bigwedge^r_R R^n \longrightarrow \bigwedge^r_R R^m$.  Hence $\det(A')$ annihilates $C_2$, as desired.
\end{proof}

For Rubin's conjecture, recall that we are given a set of $r$ prime ideals \[ S' = \{ v_1, \dots, v_r \} \] of $F$ that split completely in $H$.
Let $A \subset \Cl^T(H)^-$ denote the subgroup generated by the classes associated to the primes in $S'$.  By duality we obtain a surjection $\Cl^T(H)^{\vee, -} \longrightarrow A^\vee$.
The strong Brumer--Stark conjecture implies that
\begin{equation} \label{e:tav}
\Theta_{S,T}^\# \in \Fitt_{\Z[G]^-}(\Cl^T(H)^{\vee,-}) \subset \Fitt_{\Z[G]^-}(A^\vee). \end{equation}

The $\Z[G]^-$-module $A$ sits in a short exact sequence
\begin{equation} \label{e:aseq}
\begin{tikzcd}
 0 \ar[r] & U_{S',T}^- \ar[r] & Y_{H,S'}^- \ar[r] & A \ar[r] & 0. 
 \end{tikzcd}
 \end{equation}
Here $U_{S',T}$ is defined in (\ref{e:udef}).  The first nontrivial map in (\ref{e:aseq}) sends 
\[ u \mapsto \sum_{w \in S'_H} \ord_w(u) w = \sum_{i=1}^{r} \left(\sum_{\sigma \in G} \ord_{w_i}(\sigma(u))[\sigma^{-1}]\right)w_i,
\] where the $w_i$ are the chosen primes above the $v_i \in S'$ as in (\ref{e:ordg})--(\ref{e:ordgdef}).
The second nontrivial map in (\ref{e:aseq}) sends $w \in S'_H$ to its class in $A \subset \Cl^T(H)^-$.  Since $A$ is finite, the long exact sequence associated to the functor $\Hom_{\Z[\frac{1}{2}]}( -, \Z[\frac{1}{2}])$ applied to (\ref{e:aseq}) yields
\begin{equation} \label{e:extyua}
 \begin{tikzcd}
 0 \ar[r] & \Hom_{\Z[\frac{1}{2}]}(Y_{H,S'}^-, \Z[\textstyle{\frac{1}{2}}]) \ar[r] & \Hom_{\Z[\frac{1}{2}]}(U_{S',T}^-, \Z[\frac{1}{2}]) \ar[r] &  A^\vee \ar[r] & 0. 
 \end{tikzcd}
 \end{equation}
 To maintain $G$-equivariance of this sequence, all terms are given the contragradient $G$-action.
Note that by Shapiro's Lemma there is a canonical isomorphism of functors 
 \[ \Hom_{\Z[\frac{1}{2}]}(-, \Z[\textstyle{\frac{1}{2}}]) \cong \Hom_{\Z[G]}(-, \Z[G]^-) \]
  on the category of $\Z[G]^-$-modules.  We can therefore write (\ref{e:extyua}) as
  \begin{equation} \label{e:ext2} \begin{tikzcd}
 0 \ar[r] & \Hom_{\Z[G]}(Y_{H,S'}^-, \Z[G]^-)  \ar[r] & \Hom_{\Z[G]}(U_{S',T}^-, \Z[G]^-)  \ar[r] & A^\vee  \ar[r] & 0.  
 \end{tikzcd}
 \end{equation}
Using (\ref{e:tav}), Lemma~\ref{l:coker} implies that $\Theta_{S,T}^\# \in \Fitt_{\Z[G]^-}(A^\vee)$ annihilates the cokernel of the induced map
\begin{equation} \begin{tikzcd}
 \bigwedge\nolimits_{\Z[G]}^r \Hom_{\Z[G]}(Y_{H,S'}^-, \Z[G]^-) \ar[r] &  \bigwedge\nolimits_{\Z[G]}^r  \Hom_{\Z[G]}(U_{S',T}^-, \Z[G]^-). 
 \end{tikzcd}
\label{e:bigmap}
\end{equation}

Suppose now that we are given an element 
\[ 
\varphi \in \bigwedge\nolimits_{\Z[G]}^r  \Hom_{\Z[G]}(U_{S',T}^-, \Z[G]^-).
\]
We must prove that $\varphi(u_{\RBS}) \in \Z[G]^-$. Note that after tensoring with $\Q$ over $\Z[\frac{1}{2}]$, the first nontrivial map in (\ref{e:aseq}) and the map in (\ref{e:bigmap}) become isomorphisms. Consequently, $\varphi$ extends to an element of 
\[ 
\bigwedge\nolimits_{\Z[G]}^r \Hom_{\Z[G]}(Y_{H,S'}^-, \Z[G]^-) \otimes_{\Z\left[\frac{1}{2}\right]} \Q \cong \bigwedge\nolimits_{\Q[G]}^r \Hom_{\Q[G]}(Y_{H,S'}^- \otimes_{\Z\left[\frac{1}{2}\right]} \Q, \Q[G]^-). 
\]
We then note that
\begin{align} \varphi(u_{\RBS}) &= \varphi(\ord_G(u_{\RBS})(w_1 \wedge \cdots \wedge w_r)) \nonumber \\
&= (\ord_G(u_{\RBS})^\# \varphi)(w_1 \wedge \cdots \wedge w_r) \nonumber \\
&= (\Theta_{S,T}^\# \cdot\varphi)(w_1 \wedge \cdots \wedge w_r). \label{e:vw}
\end{align}
Here $\#$ appears because of the contragradient $G$-action.
Since $\Theta_{S,T}^\#$ annihilates the cokernel of (\ref{e:bigmap}), it follows that (\ref{e:vw}) lies in $\Z[G]^-$ as desired.
This concludes the proof that Theorem~\ref{t:sbs} implies Theorem~\ref{t:r}.

\section{On the smoothing and depletion sets}  \label{s:sad}

The goal of the rest of the paper is to prove the \emph{keystone result} Theorem~\ref{t:main0} from subsection \ref{s:fmr}.  After extending to $\cO$ and projecting onto the component $R = R_\chi = \cO[G_p]_\chi$ corresponding to a prime-to-$p$ order character $\chi$, this statement reads
 \begin{equation} \label{e:mainr}
   \Fitt_{R}(\Sel_{\Sigma}^{\Sigma'}(H)_R) = (\Theta_{\Sigma, \Sigma'}^\#). 
   \end{equation}

In this section, we alter some of the parameters in this equation.

\subsection{Removing primes above $p$ from the smoothing set} \label{s:nop}

The set $T$, and hence $\Sigma'$, may contain primes above $p$.  We show that it is safe to remove these primes from $T$ 
without altering the situation.  Note that by definition these primes are necessarily unramified in $H$.

\begin{lemma} Let $\Sigma'' = \Sigma' -  \{v \in T \colon v \mid p\}$.  We have 
\[ \Sel_{\Sigma}^{\Sigma'}(H)_R \cong \Sel_{\Sigma}^{\Sigma''}(H)_R \]
 and 
\[  (\Theta_{\Sigma, \Sigma'}^\#) =  (\Theta_{\Sigma, \Sigma''}^\#).
\]
\end{lemma}

\begin{proof}  As in (\ref{e:sst}), 
we have a  short exact sequence
\[ \begin{tikzcd}
 0 \ar[r] &  \Sel_{\Sigma}^{\Sigma''}(H)_R  \ar[r] &  \Sel_{\Sigma}^{\Sigma'}(H)_R  \ar[r] &  \Big[\displaystyle\prod_{\substack{v \in T \\ v \mid p}}\prod_{w \mid v} (\cO_H/w)^*\Big]_R^\vee  \ar[r] &  0. 
 \end{tikzcd}
  \]
The group on the right in brackets has prime-to-$p$ order, hence its tensor product with $R$ vanishes.  This proves the first result.
On the analytic side, we note that
 the factor $(1 - \sigma_v \N v)$ has  image in $R$  that is a unit when $v \mid p$ and hence the elements
\[  \Theta^\#_{\Sigma, \Sigma'} =  \Theta^\#_{\Sigma, \Sigma''} \prod_{v \in T, \ \! v \mid p} (1 - \sigma_v \N v) \]
and $ \Theta^\#_{\Sigma, \Sigma''} $ generate the same ideal under projection to $R$.
\end{proof}

Hereafter we replace $\Sigma'$ by $\Sigma''$ and therefore assume that $T$ and $\Sigma'$ contain no primes above $p$.
\subsection{Passing to the field cut out by $\chi$}

Next, we show that we can replace $H$ by the fixed field of the kernel of $\chi$ inside $G'$, which we denote $H_\chi$. We note that $p \nmid [H:H_{\chi}]$. 

\begin{lemma}   \label{l:descendchi}
Let $H_\chi \subset H$ denote the subfield of $H$ fixed by the kernel of $\chi$ inside $G'$.  Let $\Sigma \supset S_\infty$ and  $\Sigma'$ be  finite disjoint sets of places of $F$ whose union contains the set $S_{\ram}$ of finite primes  ramified in $H/F$.
There is a canonical isomorphism $\Sel_{\Sigma}^{\Sigma'}(H)_R \cong \Sel_{\Sigma}^{\Sigma'}(H_\chi)_R.$
\end{lemma}

\begin{proof}  The inclusion $H_\chi \subset H$ induces a map $\Sel_{\Sigma}^{\Sigma'}(H) \longrightarrow  \Sel_{\Sigma}^{\Sigma'}(H_\chi)$, which upon passing to the $R$-component induces a map 
\[ \begin{tikzcd}
 \Sel_{\Sigma}^{\Sigma'}(H)_R \ar[r] &  \Sel_{\Sigma}^{\Sigma'}(H_\chi)_R. 
 \end{tikzcd}
 \]

To show that this map is an isomorphism, we use the presentation (\ref{e:sels}) for the Selmer groups.  Note that
\begin{align*}
 (\Hom_\Z(\cO_{H, S', \Sigma'}^*, \Z) \otimes_\Z \cO)_R &= \Hom_{\cO}(\cO_{H, S', \Sigma'}^* \otimes \cO, \cO)_R \\
&= \Hom_{\cO}( (\cO_{H, S', \Sigma'}^* \otimes \cO)^{G' = \chi}, \cO),
\end{align*} 
where the last equality follows since $\cO_{H, S', \Sigma'}^* \otimes \cO$ is a free $\cO$-module of finite rank.  Here the superscript denotes the sub-$\cO$-module on which $g \in G'$ acts by multiplication by $\chi(g)$.
We therefore obtain a commutative diagram
\begin{equation} \label{e:seldiagram}
\begin{gathered}
\begin{tikzcd}
 (Y_{H,S' - \Sigma})_R \ar[r] \ar[d] & \Hom_{\cO}( (\cO_{H, S', \Sigma'}^* \otimes \cO)^{G' = \chi}, \cO) \ar[r] \isoarrow{d} & 
\Sel_{\Sigma}^{\Sigma'}(H)_R \ar[d] \ar[r] & 0 \\
 (Y_{H_\chi,S' - \Sigma})_R \ar[r] & \Hom_{\cO}( (\cO_{H_\chi, S', \Sigma'}^* \otimes \cO)^{G'=\chi}, \cO) \ar[r] & 
\Sel_{\Sigma}^{\Sigma'}(H_\chi)_R \ar[r] & 0.
\end{tikzcd}
\end{gathered}
 \end{equation}
As indicated, the middle vertical arrow is clearly an isomorphism (Galois theory).  Therefore the right vertical arrow is surjective.
 To prove that it is also injective, it suffices to prove that the left vertical arrow is surjective. 
 This follows since the primes in $S' - \Sigma$ are unramified in $H_\chi$. 
\end{proof}

It is clear from the definitions that the images of $\Theta_{\Sigma,\Sigma'}^{H/F}$ and $\Theta_{\Sigma,\Sigma'}^{H_\chi/F}$ in $R$ are equal.  
Lemma~\ref{l:descendchi} therefore shows that it suffices to prove equation (\ref{e:mainr}) with $H$ replaced by $H_\chi$.
Next we show that the primes ramified in $H$ but not ramified in $H_\chi$ can be excluded from the depletion and smoothing sets.
In other words, we let 
\begin{align*}
\Sigma(\chi) &= \{ v \mid p \colon v \text{ is ramified  in } H_\chi \} \cup S_\infty, \\
\Sigma'(\chi) &=  \{ v \nmid p \colon v \text{ is ramified in } H_\chi\} \cup T.
\end{align*}
Note that
\begin{equation} \label{e:thetachi}
 \Theta_{\Sigma, \Sigma'}(H_\chi/F)^\# = \Theta_{\Sigma(\chi), \Sigma'(\chi)}(H_\chi/F)^\# \prod_{v \in \Sigma - \Sigma(\chi)} (1 - \sigma_v) \prod_{v \in \Sigma' - \Sigma'(\chi)} (1 - \sigma_v \N v). \end{equation}
The fact that the Selmer group also behaves nicely with respect to the addition of unramified primes to the depletion and smoothing sets is well known:
\begin{lemma} \label{l:selchi} Suppose that the $R$-module $\Sel_{\Sigma(\chi)}^{\Sigma'(\chi)}(H_\chi)_R$ is quadratically presented.  Then $\Sel_{\Sigma}^{\Sigma'}(H_\chi)_R$ is quadratically presented as well, and we have
\[ \Fitt_R(\Sel_{\Sigma}^{\Sigma'}(H_\chi)_R) = \Fitt_R(\Sel_{\Sigma(\chi)}^{\Sigma'(\chi)}(H_\chi)_R) \prod_{v \in \Sigma - \Sigma(\chi)} (1 - \sigma_v) \prod_{v \in \Sigma' - \Sigma'(\chi)} (1 - \sigma_v \N v). \]
\end{lemma}

\begin{proof} We expand the smoothing and depletion sets one by one, using Lemma~\ref{l:ysc} in both instances.
We have the commutative diagram 
\begin{equation} \label{e:selmermaps}
\begin{tikzcd}
(Y_{H_\chi, \Sigma(\chi)})_R \ar[r,hook] \ar[d,equal] & \Sel_{\Sigma(\chi)}^{\Sigma'(\chi)}(H_{\chi})_R  \ar[r,two heads] \ar[d] & 
\Cl^{\Sigma'(\chi)}(H_\chi)^\vee_R \ar[d]  \\
 (Y_{H_\chi, \Sigma(\chi)})_R \ar[r,hook] & \Sel_{\Sigma(\chi)}^{\Sigma'}(H_\chi)_R \ar[r,two heads] & 
\Cl^{\Sigma'}(H_\chi)^\vee_R.
\end{tikzcd}
\end{equation}
 The right vertical arrow is injective with cokernel isomorphic to \begin{equation} \label{e:owd}
  \prod_{v \in \Sigma' -\Sigma'(\chi)} \prod_{w \mid v} ((\cO_H/w)^*)^\vee_R. \end{equation}
Since the primes in $\Sigma'-\Sigma'(\chi)$ are unramified in $H_\chi$, the module (\ref{e:owd}) is quadratically presented over $R$ and has Fitting ideal generated by $ \prod_{v \in \Sigma' - \Sigma'(\chi)} (1 - \sigma_v \N v)$.
Applying the snake lemma and Lemma~\ref{l:fittmult}, we find that $\Sel_{\Sigma(\chi)}^{\Sigma'}(H_\chi)_R$ is quadratically presented over $R$ with Fitting ideal equal to
\[ \Fitt_R(\Sel_{\Sigma(\chi)}^{\Sigma'(\chi)}(H_\chi)_R) \prod_{v \in \Sigma' - \Sigma'(\chi)} (1 - \sigma_v \N v). \]
Similarly, we expand depletion set using the diagram
\begin{equation} \label{e:selmermaps2}
\begin{tikzcd}
(Y_{H_\chi, \Sigma})_R \ar[r,hook] \ar[d] & \Sel_{\Sigma}^{\Sigma'}(H_{\chi})_R  \ar[r,two heads] \ar[d] & 
\Cl^{\Sigma'}(H_\chi)^\vee_R \ar[d,equal]  \\
 (Y_{H_\chi, \Sigma(\chi)})_R \ar[r,hook] & \Sel_{\Sigma(\chi)}^{\Sigma'}(H_\chi)_R \ar[r,two heads] & 
\Cl^{\Sigma'}(H_\chi)^\vee_R.
\end{tikzcd}
\end{equation}

Here the left vertical arrow is the projection associated to the canonical decomposition $Y_{H_\chi, \Sigma} = Y_{H_\chi, \Sigma(\chi)} \oplus Y_{H_\chi, \Sigma - \Sigma(\chi)}$.  The left and middle vertical arrows of (\ref{e:selmermaps2}) are therefore surjective, with kernel equal to $(Y_{H_\chi, \Sigma - \Sigma(\chi)})_R$.  Since the primes in $\Sigma - \Sigma(\chi)$ are unramified in $H_\chi$, this latter module is quadratically presented over $R$ with Fitting ideal generated by $\prod_{v \in \Sigma - \Sigma(\chi)} (1 - \sigma_v)$.  
The desired result now follows from Lemma~\ref{l:fittmult}.
\end{proof}

In view of (\ref{e:thetachi}) and Lemma~\ref{l:selchi}, in order to prove (\ref{e:mainr}) it suffices to prove that $\Sel_{\Sigma(\chi)}^{\Sigma'(\chi)}(H_\chi)_R$ is quadratically presented over $R$ and that
 \begin{equation} \label{e:mainrchi}
   \Fitt_{R}(\Sel_{\Sigma(\chi)}^{\Sigma'(\chi)}(H_\chi)_R) = (\Theta_{\Sigma(\chi), \Sigma'(\chi)}^\#). 
   \end{equation}

To recapitulate, by the results of \S\ref{s:sad}, it remains to prove that the module $\Sel_{\Sigma}^{\Sigma'}(H)_R$ is quadratically presented over $R$ and that
\begin{equation} \label{e:myfitt}
   \Fitt_{R}(\Sel_{\Sigma}^{\Sigma'}(H)_R) = (\Theta_{\Sigma, \Sigma'}^\#) \end{equation}
when:
\begin{itemize}
\item $H/F$ is such that $\chi$ is a {\em faithful} odd character of the maximal prime-to-$p$ subgroup $G' \subset G$;
\item the sets $\Sigma, \Sigma'$ are defined as in the beginning of \S\ref{s:fmr} for this extension $H/F$;
\item the set $T$ contains no primes above $p$.
\end{itemize}
The results of this section show that (\ref{e:myfitt}) in this setting implies Theorem~\ref{t:main0}.

\section{Divisibility Implies Equality} \label{s:die}

In this section we prove an analogue in the general setting of the ``elementary argument" mentioned in the introduction and described in \S\ref{s:fitting} for the case where $H/F$ is unramified at all finite primes.  First, this argument will replace $\Cl^T(H)^-$ with an appropriate Selmer module since the former is not in general quadratically presented.  Second, the analytic argument will be quite a bit more complicated for two reasons. (i) The Selmer module and Stickelberger element will have ``trivial zeroes" at any  character $\psi$ for which there exists $v \in \Sigma$ such that $\psi(G_v) = 1$, hence any  generalization of (\ref{e:cnf}) must account for trivial zeroes.  (ii) The class number formula relates the size of class groups to $L$-values, and the exact sequences relating these class groups to Selmer modules are in general not split; appropriate quotients must be taken on which the size of  class groups and Selmer modules can be related.

Recall the notation $G = \Gal(H/F) = G_p \times G'$, with $G_p$ of $p$-power order and $G'$ of prime-to-$p$ order.  Let $R = \cO[G_p]_\chi$ be a connected component of $\cO[G]$ corresponding to an odd character $\chi$ of $G'$.  
Let $H_p$ denote the fixed field of $G'$ in $H$, so $\Gal(H_p/F) \cong G_p$.
By  our earlier reductions we can  assume that $\chi$ is a faithful character of $G'$.

We recall the sets $\Sigma, \Sigma'$ defined in \S\ref{s:sad} and introduce the notation $\Sigma_p$.  As usual $S_{\ram}$ denotes the set of finite primes of $F$ ramified in $H$.
\begin{align}
\Sigma & = \{ v  \in S_{\ram}  \colon v \mid p \} \cup S_\infty, \label{e:bigsigmadef} \\
\Sigma_p &= \{ v \in S_{\ram} \colon v \mid p  \text{ and } \chi(G_v') =1\} \subset \Sigma, \label{e:sigmasub} \\
\Sigma' &=  \{ v \in S_{\ram} \colon  v \nmid p \} \cup T. \label{e:sigmapdef}
\end{align}

In (\ref{e:sigmasub}), $G_v' = G' \cap G_v$.  Since $\chi$ is faithful, the condition $\chi(G_v') = 1$ is equivalent to $G_v' = 1$, i.e.\ that $G_v$ is a $p$-group.
 The goal of this section is to prove the following:

\begin{theorem}  \label{t:include} Suppose that in every situation with notation as above, we have that the $R$-module $\Sel_{\Sigma}^{\Sigma'}(H)_R$ is quadratically presented and that  
 \begin{equation} \label{e:include}
\Fitt_R(\Sel_{\Sigma}^{\Sigma'}(H)_R) \subset (\Theta_{\Sigma, \Sigma'}^\#). 
\end{equation}  
Then each such inclusion is an equality.
\end{theorem}
Note that our proof proceeds by replacing $H$ by the subfield $H_{\psi}$, the subfield of $H$ fixed by the kernel of $\psi$, for every character $\psi$. Therefore we do not show directly that a single inclusion as in (\ref{e:include}) is necessarily an equality; we show that if {\em every} such inclusion holds, then they are {\em all} equalities.  For the remainder of this section, we assume that (\ref{e:include}) always holds.

\bigskip

Recall the following exact sequence of $\cO[G]$-modules (Lemma~\ref{l:ysc}):
\begin{equation} \label{e:selmerex}
 \begin{tikzcd}
 0  \ar[r] &  Y_{H,\Sigma}^-  \ar[r] &  \Sel_{\Sigma}^{\Sigma'}(H)^-  \ar[r] &  \Cl^{\Sigma'}(H)^{\vee,-}  \ar[r] &  0. 
 \end{tikzcd}
  \end{equation}
Note that $(Y_{H,\Sigma})_R \cong (Y_{H,\Sigma_p})_R$ since $(Y_{H,\{v\}})_R = 0$ when $\chi(G'_v) \neq 1$.  In particular:
\begin{equation} \text{if }  \Sigma_p = \emptyset, \text{ then } \Sel_{\Sigma}^{\Sigma'}(H)_R \cong (\Cl^{\Sigma'}(H)^\vee)_R. \label{e:selcl}
\end{equation}

\begin{lemma} \label{l:alpha}
Let $\alpha$ be any character of $G'$.  Denote by $\cO_\alpha$ the ring $\cO$ endowed with the $G'$-action in which $G'$ acts via $\alpha$.  Write \[  \Cl^{\Sigma'}(H)^\vee_{\cO_\alpha} =    
\Cl^{\Sigma'}(H)^\vee \otimes_{\Z[G']} {\cO_\alpha}. \]  Let $H_\alpha$ denote the fixed field of the kernel of $\alpha$ in $G'$.
Then
\[ \Cl^{\Sigma'}(H_\alpha)^\vee_{\cO_\alpha} \cong  \Cl^{\Sigma'}(H)^\vee_{\cO_\alpha} \]
\end{lemma}

\begin{proof}
This follows because $[H:H_\alpha]$ is relatively prime to $p$.  The maps 
\[
\fa \mapsto \fa\cO_H \qquad \text{ and } \qquad \fa \mapsto \frac{1}{[H:H_\alpha]} \N_{H/H_\alpha} \fa
\]
are explicit mutually inverse isomorphisms between $\Cl^{\Sigma'}(H_\alpha)_{\cO_\alpha}$ and $\Cl^{\Sigma'}(H)_{\cO_\alpha}$. The isomorphism in the lemma is  the Pontryagin dual of this, with $\alpha$ replaced by $\alpha^{-1}$.
\end{proof}

  The proof of Theorem~\ref{t:include} relies  on the analytic class number formula, which manifests itself in the following lemma.

\begin{lemma}  \label{l:gptriv}
We have
\[ \# (\Cl^{\Sigma'}(H)^\vee)_R = \#\cO/ L, \]
where 
\[ L = L_{S_\infty, \Sigma'}(H/H_p, \chi, 0) = \prod_{\psi|_{G'} = \chi} L_{S_\infty, \Sigma'}(H/F, \psi,0). \]
Here the product runs over the characters $\psi$ of $G = \Gal(H/F)$ that belong to $\chi$.
\end{lemma}

\begin{proof}
For the purposes of the first equality, we can work entirely over $H_p$, i.e.\ we can replace $F$ by $H_p$.  Note that $H_p$ is totally real since its degree over $F$ is odd.  For the extension $H/H_p$, the associated set $\Sigma_p$ is empty, since $G_v = G_v'$ and $\chi$ is faithful, so $\chi(G_v') = 1$ implies that $G_v' = 1$ and hence $v$ is unramified (in fact totally split).
In this setting the ring $R$ is just $\cO_\chi$, i.e.\ the ring $\cO$ in which the group $\Gal(H/H_p)$ acts via $\chi$.
Therefore the running assumption (\ref{e:include}) together with the isomorphism (\ref{e:selcl}) yield
 \[ \Fitt_{\cO_\chi}(\Cl^{\Sigma'}(H)^\vee_{\cO_\chi}) \subset ( L_{S_\infty, \Sigma'}(H/H_p, \chi, 0)), \]
 which says simply
 \[ \#\cO/L_{S_\infty, \Sigma'}(H/H_p, \chi, 0) \mid \#\Cl^{\Sigma'}(H)^\vee_{\cO_\chi}. \]
 We apply the same result to all odd characters $\alpha$ of $H/H_p$, to obtain
 \begin{equation} \label{e:alpha}
  \#\cO/L_{S_\infty, \Sigma'}(H/H_p, \alpha, 0) \mid \#\Cl^{\Sigma'}(H_\alpha)^\vee_{\cO_\alpha} =  \#\Cl^{\Sigma'}(H)^\vee_{\cO_\alpha},  
  \end{equation}
where the last equality uses Lemma~\ref{l:alpha}.  
Taking the product over all $\alpha$ gives
\begin{equation} \label{e:divprod}
  \#\cO/L_{S_\infty, \Sigma'}(H/H^+, \epsilon, 0) \mid  \#\Cl^{\Sigma'}(H)^{\vee, -}_\cO, \end{equation}
where $H^+$ is the maximal totally real subfield of $H$, and $\epsilon$ is the nontrivial character of $\Gal(H/H^+)$.
The left side of (\ref{e:divprod}) uses the Artin formalism for $L$-functions, and the right side uses the fact that $\cO[\Gal(H/H_p)]^-$ is the direct product of the $\cO_\alpha$.  Now, (\ref{e:divprod}) is actually an equality by the analytic class number formula (Lemma~\ref{l:cnf}).  It follows that each divisibility (\ref{e:alpha}) is an equality as well.  This yields the first equality of the lemma, with $\alpha = \chi$.  The second equality follows from the Artin formalism for $L$-functions.
\end{proof}

\subsection{Case: $\Sigma_p$ is empty} \label{s:base}

We first handle the case that $\Sigma_p$ is empty.  Note that in this case, the image of $\Theta_{\Sigma, \Sigma'}$ is a non-zerodivisor in $R$.
The fact that  $\Sel_{\Sigma}^{\Sigma'}(H)_R$ is quadratically presented together with the inclusion (\ref{e:include}) implies that we may write
\[ \Fitt_R(\Sel_{\Sigma}^{\Sigma'}(H)_R) = (x \cdot \Theta_{\Sigma, \Sigma'}^\#) \]
for some $x \in R$.  By (\ref{e:selcl}), which applies since $\Sigma_p = \emptyset$, this reads
\[ \Fitt_R(\Cl^{\Sigma'}(H)^\vee_R) =(x \cdot \Theta_{\Sigma, \Sigma'}^\#). \]
Lemmas~\ref{l:size} and~\ref{l:size2} imply that
\begin{equation} \#\Cl^{\Sigma'}(H)^\vee_R = \#\cO / \prod_{\psi|_{G'} = \chi} \psi(x) L_{\Sigma, \Sigma'}(H/F, \psi, 0). \label{e:clxl}
\end{equation}
Yet  
\begin{equation} \label{e:lnos}
 L_{\Sigma, \Sigma'}(H/F, \psi, 0) = L_{S_\infty, \Sigma'}(H/F, \psi, 0) \prod_{v \in \Sigma - S_\infty} (1 - \psi(v)). \end{equation}
Since $\Sigma_p=\emptyset$, any $v \in \Sigma - S_\infty$ satisfies
\[ \begin{cases} \psi(v) = 0 & \text{if } \psi \text{ is ramified at } v, \\
\psi(v) \equiv \chi(v) \not\equiv 1 \pmod{\pi_E} & \text{if } \psi \text{ is unramified at } v.
\end{cases}
\]
Here $\pi_E \in \cO$ is a uniformizer.  It follows that
the product on the right in (\ref{e:lnos}) is a $p$-adic unit.
Hence
Lemma~\ref{l:gptriv} and (\ref{e:clxl}) imply that
 $\prod \psi(x) \in \cO^*$.  Therefore each $\psi(x) \in \cO^*$, which implies that $x \in R^*$ since the $\cO$-algebra maps $R \longrightarrow \cO$ induced by each character $\psi$ are local homomorphisms of local rings.  
 This is the desired result.
 
\subsection{Case: $\Sigma_p$ is not empty}

Now consider the case of $\Sigma_p$ nonempty.  
As in \S\ref{s:base}, the inclusion (\ref{e:include}) implies that the principal ideal $ \Fitt_R(\Sel_{\Sigma}^{\Sigma'}(H)_R)$ is generated by an element of the form $x \cdot \Theta_{\Sigma, \Sigma'}^\#$ for some $x \in R$.  We must show that $x$ is a unit in $R$.

The equality $ \Fitt_R(\Sel_{\Sigma}^{\Sigma'}(H)_R) = (x \cdot \Theta_{\Sigma, \Sigma'}^\#)$ implies that for all characters $\psi$ of $G$ that belong to $\chi$, we have
\begin{equation} \label{e:fitpsi}
 \Fitt_\cO(\Sel_{\Sigma}^{\Sigma'}(H)_\psi) = (\psi(x) \cdot L_{\Sigma, \Sigma'}(\psi, 0)) \subset 
 (L_{\Sigma, \Sigma'}(\psi, 0)). \end{equation}
Note here that 
\begin{align*}
\Sel_{\Sigma}^{\Sigma'}(H)_\psi  = & \ \Sel_{\Sigma}^{\Sigma'}(H) \otimes_{\Z[G]} \cO_\psi \\
 = & \ (\Sel_{\Sigma}^{\Sigma'}(H) \otimes_{\Z} \cO) / \langle g - \psi(g)\colon g \in G \rangle 
\end{align*}
denotes the $\psi$-coinvariants of $\Sel_{\Sigma}^{\Sigma'}(H)$. 
Suppose we can prove that the inclusion in (\ref{e:fitpsi}) is an equality for some $\psi$ that belongs to $\chi$ satisfying $L_{\Sigma, \Sigma'}(\psi, 0) \neq 0$.  This implies that $\psi(x) \in \cO^*$, which implies $x \in R^*$, giving the desired result
\[ \Fitt_R(\Sel_{\Sigma}^{\Sigma'}(H)_R) = (\Theta_{\Sigma, \Sigma'}^\#). \]
Now if {\em every} character $\psi$ belonging to $\chi$ has a trivial zero (i.e.\ if for each $\psi$ there exists $v \in \Sigma$ with $\psi(G_v)=1$, so $L_{\Sigma, \Sigma'}(\psi, 0) = 0$) then
Lemma~\ref{l:trivial} shows that
\[  \Fitt_R(\Sel_{\Sigma}^{\Sigma'}(H)_R) = 0 =  (\Theta_{\Sigma, \Sigma'}^\#),\]
again giving the desired result.
It therefore suffices to prove that 
\begin{equation} \label{e:fitpsieq}
\Fitt_\cO(\Sel_{\Sigma}^{\Sigma'}(H)_\psi) = (L_{\Sigma, \Sigma'}(\psi, 0)) \end{equation}
for every character $\psi$ belonging to $\chi$. We need the following lemma.

\begin{lemma}   \label{l:descend}
Let $H_\psi \subset H$ denote the subfield of $H$ fixed by the kernel of $\psi$.
There is a canonical isomorphism $\Sel_{\Sigma}^{\Sigma'}(H)_\psi \cong \Sel_{\Sigma}^{\Sigma'}(H_\psi)_\psi.$
\end{lemma}

\begin{proof} The proof is nearly identical to Lemma~\ref{l:descendchi}, replacing $(H_\chi, R)$ with $(H_\psi, \cO_\psi)$.
We omit the details.
\end{proof}

Since it remains only to prove (\ref{e:fitpsieq}), Lemma~\ref{l:descend} implies that we may replace $H$ by $H_\psi$ and hence assume $H=H_\psi$ for the remainder of the proof. Note that in view of equation (\ref{e:thetachi}) and Lemma \ref{l:selchi} we can also replace $\Sigma$ and $\Sigma'$ by $\Sigma(H_{\psi})$ and $\Sigma'(H_{\psi})$ respectively. Note that $\psi$ is ramified at all primes in $\Sigma_p$. By subsection \ref{s:base} we assume that $\Sigma_p$ is not empty. 

The extension $H_\psi/F$ is cyclic.  Each $v \in \Sigma_p$ satisfies $\psi(G_v') = \chi(G_v') = 1$, hence the decomposition group of $v$ in $\Gal(H_\psi/F)$ is a $p$-group.
Therefore there exists a $v \in \Sigma_p$ whose inertia group $I_v$ is minimal in the sense that $I_v \subset I_w$ for all $w \in \Sigma_p$, since the subgroups of a cyclic $p$-group are linearly ordered by inclusion.  We write $I$ for this minimal $I_v$.  The fact that $\psi$ is ramifed at all $v \in \Sigma_p$ and $\Sigma_p$ is nonempty implies that $I \neq 1$. 

\begin{lemma} \label{l:selcl}
With notation as above, we have $\Sel_{\Sigma}^{\Sigma'}(H_\psi)_{R}/\N I \cong (\Cl^{\Sigma'}(H_\psi)^\vee)_{R}/\N I $. \end{lemma}

\begin{proof} Denote by $\Sigma_{H_\psi}$ the set of places of $H_\psi$ above those in $\Sigma$.
 First note that from the short exact sequence
 \[ \begin{tikzcd}
  0  \ar[r] &  Y_{H_\psi, \Sigma - \Sigma_p}  \ar[r] &  \Sel_{\Sigma}^{\Sigma'}(H_\psi)  \ar[r] &  \Sel_{\Sigma_p}^{\Sigma'}(H_\psi)  \ar[r] &  0 
  \end{tikzcd}
   \]
 it follows that 
\[ \Sel_{\Sigma}^{\Sigma'}(H_\psi)_{R} \cong \Sel_{\Sigma_p}^{\Sigma'}(H_\psi)_{R}. \]
Indeed, any place $w \in (\Sigma - \Sigma_p)_{H_\psi}$ has image in $(Y_{H_\psi,\Sigma-\Sigma_p})_R$ that vanishes
 (if $\sigma \in G_v'$ with $\chi(\sigma) \neq 1$, then $1 - \sigma$ acts trivially on the image of $w$ in $Y_{H_\psi,\Sigma - \Sigma_p}$ and has image in $R$ that is a unit).

It therefore suffices to prove the  result with $\Sigma$ replaced by $\Sigma_p$.
For this, we tensor the exact sequence (\ref{e:selmerex}) with $R/\N I$ over $R$.  We need to show that the image of 
\[ \begin{tikzcd}
 (Y_{H_\psi,\Sigma_p})_R/\N I \ar[r] &  \Sel_{\Sigma_p}^{\Sigma'}(H_\psi)_{R}/\N I 
 \end{tikzcd} 
 \]
  vanishes.   
We will show that this already holds on the full minus side over $\Z[1/2]$ (without passing to the $R$-component), i.e. that
\begin{equation} \label{e:xs}
\begin{tikzcd}
 Y_{H_\psi,\Sigma_p}^-/\N I  \ar[r] &  \Sel_{\Sigma_p}^{\Sigma'}(H_\psi)^-/\N I  
 \end{tikzcd}
\end{equation}
vanishes.
 
Define $M = \Cl^{\Sigma'}(H_\psi)^-/\N I$.  The primes $\fP \in (\Sigma_p)_{H_\psi}$ come in pairs $(\fP, \overline{\fP})$ that are associated by complex conjugation, with $\fP \neq \overline{\fP}$ since $\chi(G_v') = 1$ while $\chi$ is odd. We choose a representative $\fP$ for each pair and denote this set of  representatives by $J$.  Let $e = \#I$.
We claim that the images of $\fP/\overline{\fP}$ are ``linearly independent modulo $e$" in $M$, in the following  sense:
 \[ \text{ if } \prod_{\fP \in J} (\fP/\overline{\fP})^{a_\fP} \text{ has trivial image in } M, \text{ then }  e \mid a_{\fP} \text{ for all } \fP. \]
 To see this, suppose that
\begin{equation} \label{e:fP}
  \prod_{\fP \in J} (\fP/\overline{\fP})^{a_\fP}  = (x) \fa^{\N I} \end{equation}
for some $x \in H_{\psi, \Sigma'}^{*, -}$ and some fractional ideal $\fa \in I_{\Sigma'}(H_\psi)^-$.
Then all items in (\ref{e:fP}) are invariant under all $\sigma \in I$ except possibly the fractional ideal $(x)$, which implies that $(x)$ is invariant as well; since the generator in
$H_{\psi, \Sigma'}^{*, -}$ of a principal ideal on the minus side (i.e. in $I_{\Sigma'}(H_\psi)^-$) is unique, this implies that $x \in (H_{\psi}^I)_{\Sigma'}^*$.
But the ideals $\fP$ are totally ramified over $H_{\psi}^I$, and hence the valuations of $x$ at these primes must be multiples of $e$; it follows that the $a_\fP$ are multiples of $e$ as well.

Now fix one of the $\fP \in J$.  We will show that the image of $\ord_\fP - \ord_{\overline{\fP}}$ in $\Sel_{\Sigma_p}^{\Sigma'}(H_\psi)^-$ is a multiple of $\N I$; this is precisely the desired result that (\ref{e:xs}) vanishes.
The claim just proven implies that there is a group homomorphism $\tilde{\phi}\colon M \longrightarrow \Q/\Z[\frac{1}{2}]$ such that
$\tilde{\phi}(\fP - \overline{\fP}) = 1/e$ and $\tilde{\phi}(\fP' -  \overline{\fP}') =0  $ for all $\fP'  \in J, \fP' \neq \fP$. 
Considering $M$ as the quotient:
\[ M = I_{\Sigma'}(H_\psi)^{-}/\langle H_{\psi,\Sigma'}^{*, -}, \N I \cdot I_{\Sigma'}(H_\psi)^{-} \rangle, \]
we can lift $\tilde{\phi}$ to a $\Z[\frac{1}{2}]$-module homomorphism $\phi\colon  I_{\Sigma'}(H_\psi)^{-} \longrightarrow \Q$, since $I_{\Sigma'}(H_\psi)^{-}$ is free as a $\Z[\frac{1}{2}]$-module.  Furthermore we  can choose  this lift to satisfy $\phi(\fP - \overline{\fP}) = 1/e$ and $\phi(\fP'- \overline{\fP}') = 0 $ for all $\fP'  \in J, \fP' \neq \fP$.
The restriction of $\phi$ to $H_{\psi,\Sigma'}^{*, -}$ is $\Z[\frac{1}{2}]$-valued (since this group has trivial image in $M$), and hence yields a class  
$\Phi \in \Sel_{\Sigma_p}^{\Sigma'}(H_\psi)^-$  defined explicitly by  
\[ \Phi = \sum_{w \not\in \Sigma'_{H_\psi}} \phi(w) \ord_w. \]

To conclude the proof, we will show that $\ord_\fP -  \ord_{\overline{\fP}}$ and $\N I  \cdot\Phi$ are equal in $\Sel_{\Sigma}^{\Sigma'}(H_\psi)^-.$ From the  construction of $\phi$, we see that   
\[ \Phi = \frac{1}{e}(\ord_\fP - \ord_{\overline{\fP}}) + \sum_{w\not\in (\Sigma_p \cup \Sigma')_{H_\psi}} \phi(w) \ord_w \]
and hence
\begin{align}
 \N I \cdot \Phi & = (\ord_\fP - \ord_{\overline{\fP}}) + \N I\sum_{w\not\in (\Sigma_p \cup \Sigma')_{H_\psi}} \phi(w) \ord_w \nonumber \\
 & =  (\ord_\fP - \ord_{\overline{\fP}}) + \sum_{w\not\in (\Sigma_p \cup \Sigma')_{H_\psi}} \phi(\N I \cdot w) \ord_w. \label{e:nip}
 \end{align}
 Since $\phi \circ \N I$ is $\Z[\frac{1}{2}]$-valued by the definition of $M$, the sum on the right in (\ref{e:nip}) has trivial image in $\Sel_{\Sigma_p}^{\Sigma'}(H_\psi)^-$, by the definition of this group.  The result follows.
\end{proof}

\begin{lemma} \label{l:clsize}
The size of the group  $(\Cl^{\Sigma'}(H_\psi)_{R})^\vee/\N I $ is $\#\cO/L_I$, where
\[ L_I = \prod_{\substack{\alpha(I) \neq 1 \\ \alpha|_{G'} = \chi}} L_{\Sigma, \Sigma'}(H_\psi/F, \alpha, 0). \] Here the product ranges over all characters $\alpha$ of $\Gal(H_\psi/F)$ that belong to $\chi$ such that $\alpha(I) \neq 1$.
\end{lemma}

\begin{proof} Note  that in the definition of $L_I$, each character $\alpha$ is ramified at every $v \in\Sigma_p$, while the Euler factor $(1 - \alpha(v))$ is a $p$-adic unit
for each finite $v \in \Sigma - \Sigma_p$, hence
\[ L_{\Sigma, \Sigma'}(H_\psi/F, \alpha, 0) = L_{S_\infty, \Sigma'}(H_\psi/F, \alpha, 0). \]

For notational simplicity, write $M= \Cl^{\Sigma'}(H_\psi)_R$.
 By Lemma~\ref{l:gptriv}, we have $\#M^\vee = \cO/L$, where
\[  L = \prod_{\alpha|_{G'} = \chi} L_{S_\infty, \Sigma'}(H_\psi/F, \alpha, 0). \]
We therefore need to prove that 
\begin{equation} \label{e:nimv}
  \# (\N I  \cdot M^\vee) = \# \cO/L_I', \end{equation}
where 
\begin{align}
L_I' &= \prod_{\substack{\alpha(I) = 1 \\ \alpha|_{G'} = \chi}} L_{S_\infty, \Sigma'}(H_\psi/F, \alpha, 0) \\
& = \prod_{\substack{\alpha \in  \Gal(H_\psi^I/F)\hat{\ } \\  \alpha|_{G'} = \chi}} L_{S_\infty, \Sigma'}(H_\psi^I/F, \alpha, 0). \label{e:lip}
\end{align}
First note that we can replace $M^\vee$ by $M$ in (\ref{e:nimv}) since $M$ is finite; indeed, from the exact sequence
\[ \begin{tikzcd}
 0 \ar[r] & M^\vee[\N I] \ar[r] & M^\vee \ar[r,"\N I"] & M^\vee \ar[r] & M^\vee/\N I \ar[r] & 0 
\end{tikzcd}
 \]
we see that
\[ \#M^\vee/\N I = \# M^\vee[\N I]  = \# (M/\N I)^\vee = \#(M/\N I) \]
and hence $\#(\N I \cdot M^\vee) = \#(\N I \cdot M).$ Our goal is therefore to prove that
\begin{equation} \label{e:nim}
  \# (\N I  \cdot M)= \cO/L_I'. \end{equation}

Next note that if $N =  \Cl^{\Sigma'}(H_\psi^I)_R$, then Lemma~\ref{l:gptriv} and (\ref{e:lip}) imply that
\begin{equation} \label{e:nin}
\#N = \# \cO/L_I'.
\end{equation}
In view of (\ref{e:nim}) and (\ref{e:nin}), it suffices to prove that the canonical map $N \longrightarrow M^I$ given by extension of ideals is an injection that identifies $N$ with $\N I \cdot M$.

For the injectivity one applies the snake lemma to the commutative diagram
\[  
\begin{tikzcd}
 0 \ar[r]  & (H_{\psi, \Sigma'}^{I,*})_R \ar[r] \ar[d] & I_{\Sigma'}(H_\psi^I)_R \ar[r] \ar[d,"\fa \mapsto \fa \cO_{H_\psi}"] & N \ar[r] \ar[d]& 0 \\
0 \ar[r] & (H_{\psi, \Sigma'}^{*})^I_R \ar[r] & I_{\Sigma'}(H_\psi)^I_R \ar[r] & M^I \ar[r] & 0
\end{tikzcd}
\]
(Note that the 0 on the bottom right comes from Hilbert's Theorem 90, though it is not necessary here.)   The left vertical arrow is an isomorphism by Galois theory, and the middle vertical arrow is clearly an injection.  It follows that $N \longrightarrow M^I$ is injective.

To conclude we must show that the norm map $M \longrightarrow N$ is surjective. By class field theory the group $M$ and $N$ are Galois groups of extensions of $H_{\psi}$ and $H_{\psi}^I$, respectively, unramified outside $\Sigma'$. The norm map is identified with the natural restriction on Galois groups. This  restriction map is surjective as the extension $H_{\psi}/H_{\psi}^I$ is totally ramified at primes in $\Sigma_p$ by the choice of $I$.    
\end{proof}

We can now apply an analytic argument similar to \S\ref{s:base} to conclude this case.

\begin{lemma} \label{l:ri}  Let $R_I = R/\N I$. We have $\Fitt_{R_I}(\Sel_{\Sigma}^{\Sigma'}(H_\psi)_{R_I}) = (\Theta_{\Sigma, \Sigma'}^\#).$
\end{lemma}

\begin{proof}  Projecting (\ref{e:include}) from $R$ to $R_I$ we obtain an inclusion 
\[ \Fitt_{R_I}(\Sel_{\Sigma}^{\Sigma'}(H_\psi)_{R_I}) \subset (\Theta_{\Sigma, \Sigma'}^\#). \]
Note that by Corollary~\ref{c:rchi}, the ring $R_I$ is a character group ring and hence we may apply Lemmas~\ref{l:size} and \ref{l:size2}.
 If we write 
$\Fitt_{R_I}(\Sel_{\Sigma}^{\Sigma'}(H_\psi)_{R_I})= (x \cdot \Theta_{\Sigma, \Sigma'}^\#)$ for some $x \in R_I$, then
these lemmas imply
that \[ \# \Sel_{\Sigma}^{\Sigma'}(H_\psi)_{R_I} = \# \cO/\prod_{\substack{\alpha(I) \neq 1 \\ \alpha|_{G'} = \chi}} \alpha(x) L_{\Sigma, \Sigma'}(H_\psi/F, \alpha, 0). \]
Combining this equality with Lemmas~\ref{l:selcl} and~\ref{l:clsize} we find that \[ \prod_{\substack{\alpha(I) \neq 1 \\ \alpha|_{G'} = \chi}} \alpha(x)  \in \cO^*,\] and therefore  each $\alpha(x) \in \cO^*$.  This implies  $x \in R_I^*$ as desired.
\end{proof}

Projecting the equality of Lemma~\ref{l:ri} to $\cO_\psi$, we obtain (\ref{e:fitpsieq}).  We have now completed the proof 
of Theorem~\ref{t:include}.

\begin{remark} The remainder of the  paper is  dedicated to proving  the desired inclusion
\[ \Fitt_R(\Sel_{\Sigma}^{\Sigma'}(H)_R) \subset (\Theta_{\Sigma, \Sigma'}^\#). \]
We assume from here on that the image of $\Theta_{\Sigma, \Sigma'}^\#$ in $R$ lies in the maximal ideal $\fm_R$.  Otherwise, it is a unit in $R$ and the desired inclusion 
holds trivially.
\end{remark}

\section{The module $\nabla$ and its key properties} \label{s:properties}

The module that appears in our constructions with Hilbert modular forms is not the Selmer module $\Sel_{\Sigma}^{\Sigma'}(H)$ but a certain canonical transpose  in the sense of Jannsen \cite{jannsen}.  In this section we state the salient properties
of this module, denoted $\nabla_{\Sigma}^{\Sigma'} = \nabla_{\Sigma}^{\Sigma'}(H)$.  The actual construction of $\nabla_{\Sigma}^{\Sigma'}$
and details of the proofs are relegated to the appendix.

\smallskip
In the appendix, we work with general disjoint finite sets $\Sigma$, $\Sigma'$ of places of $F$  such that $\Sigma \supset S_\infty$ and $\Sigma'$ satisfies condition (\ref{e:drcond}) of the introduction.
In this section we specialize to the sets $\Sigma$ and $\Sigma'$ defined in (\ref{e:bigsigmadef}) and (\ref{e:sigmapdef}).  The Ritter--Weiss module  $\nabla_{\Sigma}^{\Sigma'}$ associated to these sets $\Sigma, \Sigma'$ satisfies the following properties.
\begin{enumerate}[label={(P\arabic*)}]
\item \label{i:exact}
There is a short exact sequence of $\Z[G]$-modules
\begin{equation} \label{e:nablaext}
 \begin{tikzcd}
 0 \ar[r] &  \Cl_\Sigma^{\Sigma'}(H) \ar[r] &  \nabla_{\Sigma}^{\Sigma'} \ar[r] &  X_{H,\Sigma} \ar[r] &  0.
 \end{tikzcd}
 \end{equation}
\item \label{i:ext}
 After tensoring with $\Z[\frac{1}{2}]$ and passing to minus parts, the extension class associated to 
\begin{equation} \label{e:nablaextminus}
\begin{tikzcd}
 0 \ar[r] &  \Cl_\Sigma^{\Sigma'}(H)^{-} \ar[r] &  \nabla_{\Sigma}^{\Sigma', -} \ar[r] &  X_{H,\Sigma}^{-} \ar[r] &  0 
 \end{tikzcd}
 \end{equation}
  in
\begin{equation} \label{e:ext0}
\Ext^1_{\Z[G]^-}(X_{H,\Sigma}^-,  \Cl_\Sigma^{\Sigma'}(H)^-) \cong \bigoplus_{v \in \Sigma} H^1(G_v, \Cl_\Sigma^{\Sigma'}(H)^-) \end{equation}
is equal to a certain tuple of canonical Galois cohomology classes $(\lambda_v)_{v \in \Sigma}$ defined below using class field theory (the isomorphism (\ref{e:ext0}) is explained in (\ref{e:ext}) below).
\end{enumerate}

To obtain further desired properties, we must base change to $\Z_p$ and consider \[ (\nabla_{\Sigma}^{\Sigma'})_p = \nabla_{\Sigma}^{\Sigma'} \otimes \Z_p. \]

\begin{enumerate}[label={(P\arabic*)},resume]
\item \label{i:transpose} The $\Z_p[G]$-module   $(\nabla_\Sigma^{\Sigma'})_p$ has  a canonical transpose $(\nabla_\Sigma^{\Sigma'})^{\tr}_p$ that is isomorphic to the Selmer module  $\Sel_{\Sigma}^{\Sigma'}(H)_p$ defined in \S\ref{s:bks}.
\item \label{i:qp}
The $\Z_p[G]$-module $(\nabla_\Sigma^{\Sigma'})_p$ is quadratically presented. 
 \end{enumerate}

While most of the content of our construction is contained in the work of Ritter--Weiss \cite{rw} and Burns--Kurihara--Sano \cite{bks}, the construction of our precise module $\nabla_\Sigma^{\Sigma'}$ satisfying properties \ref{i:exact}--\ref{i:qp} does not seem to be present in the literature.  For instance, Ritter and Weiss do not consider the ``smoothing" set $\Sigma'$.  As a result they obtain a
 presentation $P_1 \longrightarrow P_0 \longrightarrow \nabla_{\Sigma} \longrightarrow 0$
 where $P_1$ is projective, but $P_0$ is only cohomologically trivial.  Furthermore, they do not consider properties \ref{i:ext} and \ref{i:transpose} in the form that we need.  Meanwhile Burns--Kurihara--Sano  define a Selmer module $\Sel_\Sigma^{\Sigma'}(H)^{\tr}$ satisfying properties \ref{i:exact} and \ref{i:transpose}, however \ref{i:qp} is proved only in the case  $\Sigma \supset S_{\ram}$, and property \ref{i:ext} is not considered.  
 
For these reasons, we describe the construction of $\nabla_{\Sigma}^{\Sigma'}$ and the proof of properties \ref{i:exact}--\ref{i:qp} in detail.  This construction, which draws heavily from \cite{rw}, is described in the appendix and may be of independent interest beyond our applications in this paper. 
Our construction is closely related to that of Nickel in \cite{nickel}*{\S2.3}.
 In the remainder of this section we elaborate on the statement of properties \ref{i:ext} and \ref{i:transpose}.

\subsection{Transpose}

In this section we describe property \ref{i:transpose}.  
For any $\Z[G]$-module $M$, we endow the dual $M^* := \Hom_{\Z[G]}(M, \Z[G])$  with the contragradient action
\begin{equation} \label{e:rpd}
(r \cdot \varphi)(x) =   \varphi(r^{\#} \cdot x), \qquad \text{ for } r \in \Z[G], \varphi \in M^*, x \in M. \end{equation}
Suppose that $M$ has a presentation by projective $\Z[G]$-modules of finite rank
\begin{equation} \label{e:pr} \begin{tikzcd}
P_0 \ar[r] &  P_1 \ar[r] &  M \ar[r] &  0. 
\end{tikzcd}
 \end{equation}
Then each of the modules $P_i^*$ is also projective, and following Jannsen \cite{jannsen} we call the cokernel of the induced map $P_1^* \longrightarrow P_0^*$ a {\em transpose} of the module $M$.  Transpose is only well-defined up to homotopy: if $M'$ and $M''$ are transposes of $M$ arising from different presentations, then there exist projective modules $P$ and $Q$ such that $M' \oplus P \cong M'' \oplus Q$.

Let $R = R_\Psi$ be a character group ring associated to a set $\Psi \subset \hat{G}$.  We define $R^\# = R_{\Psi^\#},$ where $\Psi^\# = \{\psi^{-1}\colon \psi \in \Psi\}$.
The involution $\#$ on $\cO[G]$ induces mutually inverse $\cO$-algebra maps 
$\#\colon  R \longrightarrow R^\#, R^\# \longrightarrow R$.
If $M$ is an $R$-module, it is then natural to view $M^*$ as an $R^\#$-module via the rule (\ref{e:rpd}).  The transpose of $M$ with respect to a projective presentation (\ref{e:pr}) also naturally has the structure of an $R^\#$-module.

\begin{lemma} \label{l:trfitt}  Let $R$ be a character group ring and suppose that $M$ is a quadratically presented $R$-module.  Let $M^{\tr}$ be the transpose of $M$ associated with any quadratic presentation of $M$.  Then $M^{\tr}$ is quadratically presented and $\Fitt_{R^\#}(M^{\tr}) = \Fitt_R(M)^\#$.
\end{lemma}

\begin{proof}  If $(a_{ij})$ is the square matrix representing a quadratic presentation of $M$ over $R$, then the matrix representing the corresponding quadratic presentation of $M^{\tr}$ over $R^\#$ is $(a_{ji}^\#)$.  The result follows.
\end{proof}

In view of \ref{i:transpose} and \ref{i:qp}, if $R$ is any character group ring quotient of $\cO[G]$, we have:

\begin{corollary} \label{c:selprin}
The $R$-module $\Sel_{\Sigma}^{\Sigma'}(H)_R$ has a quadratic presentation.   Its Fitting ideal over $R$ is principal and satisfies
\[ \Fitt_R(\Sel_{\Sigma}^{\Sigma'}(H)_R) = \Fitt_{R^\#}(\nabla_{\Sigma}^{\Sigma'}(H)_{R^\#})^\#.  \]
\end{corollary}

Note that Corollary~\ref{c:selprin} was proved in \cite{bks}*{Lemma 2.8} in the case that $\Sigma \supset S_{\ram}$.

\subsection{Extension class via Galois cohomology} \label{s:ext}

In this section we describe property \ref{i:ext}. This is a description of the  module $\nabla_{\Sigma}^{\Sigma'}(H)$, when projected to the minus side, in terms of a certain canonical Galois cohomology class arising from class field theory.  For the remainder of this section we therefore work over $\Z[\frac{1}{2}]$.  Let $M = \Cl_{\Sigma}^{\Sigma'}(H)^{-}$, and let $L/H$ denote the abelian extension corresponding via class field theory to the group $M$.  This is the maximal abelian extension of $H$ of odd degree that is  unramified outside places in $\Sigma'_H$ and at most tamely ramified at $\Sigma'_H$, such that the primes in $\Sigma_H$ split completely, and such that the conjugation action of complex conjugation on $\Gal(L/H)$ is inversion.
The extension $L/F$ is Galois, as can be seen from this description since the action of any $\sigma \in G_F$ sends $L$ to another field with these properties.  The lemma below shows that the short exact sequence of groups
\[ \begin{tikzcd}
 1 \ar[r] &  M \ar[r] &  \Gal(L/F) \ar[r] &  G \ar[r] &  1 
 \end{tikzcd} \]
splits (i.e.\ is a semi-direct product).  For this, it is crucial that we are working on the minus side.

\begin{lemma} \label{l:splitting}   Let $N$ be any $\Z[G]^-$-module, e.g.\ the module $M$ above.
 The restriction map \[ \begin{tikzcd}  \res_{G_H}^{G_F}\colon H^1(G_F, N) \ar[r] &  H^1(G_H, N)^G 
 \end{tikzcd}  \] is an isomorphism. 
\end{lemma}

\begin{proof}   The terms preceding and following the map $\res_{G_H}^{G_F}$ in the inflation-restriction sequence are  $H^i(G, N)$ for $i=1,2$. 
Yet $H^i(G, N) = 0$ for all $i$.  To see this  vanishing, note that the action of any $g \in G$ gives a $G$-module map $N \longrightarrow N$
that induces the identity on cohomology (see \cite{cf}*{Proposition 3, pg. 99}); but complex conjugation acts on $N$ as multiplication by $-1$. Since 2 has been inverted, this implies that  $H^i(G, N) = 0$ as claimed.
\end{proof}

Let \[ \begin{tikzcd}
 \rec_{L/H}\colon M \ar[r,"\sim"] &  \Gal(L/H) 
\end{tikzcd}
 \] denote the Artin reciprocity isomorphism.
Lemma~\ref{l:splitting} implies that there is a unique cohomology class \[ \lambda \in H^1(G_F, M) \] whose restriction to $H^1(G_H,M) = \Hom_{\cont}(G_H, M)$ is equal to 
 the canonical homomorphism 
 \[ \begin{tikzcd}
 \varpi \colon G_H \ar[r] & \Gal(L/H) \ar[r,"\rec_{L/H}^{-1}"] & M.
 \end{tikzcd}\]  An explicit formula for a cocycle representing $\lambda$ is given in \S\ref{s:extgal}.

Let $v \in \Sigma$ and denote by $G_{F, v} \subset G_F$ the decomposition group  of $v$ associated to some embedding $\overline{F} \subset \overline{F}_v.$  The restriction of $\lambda$ to $G_{H,v} = G_{F,v} \cap G_H$ is the restriction of $\varpi$ to a decomposition group of a prime of $H$ above $v$, and hence trivial by the definition of $M$.  It follows from the inflation-restriction sequence that $\res_{G_{F,v}}^{G_F} \lambda$ is the inflation of a unique class \begin{equation} \label{e:lvdef}
 \lambda_v \in H^1(G_v, M). \end{equation}

Next we note that \[ X_{H,\Sigma}^{-} \cong Y_{H,\Sigma}^{-} =  \bigoplus_{v \in \Sigma} (\Ind_{G_v}^G \Z)^-. \]
Therefore \begin{align}
\Ext_{\Z[G]^{-}}^{1}( X_{H,\Sigma}^{-} , M) & \cong  \bigoplus_{v \in \Sigma} \Ext_{\Z[G]^{-}}^{1}( (\Ind_{G_v}^G \Z)^-, M)  \nonumber \\
& \cong   \bigoplus_{v \in \Sigma} \Ext_{\Z[{1}/{2}][G_v]}^1(\Z[\tfrac{1}{2}], M) \nonumber \\
& \cong   \bigoplus_{v \in \Sigma} H^1(G_v, M). \label{e:ext}
\end{align}
Let us make explicit how one associates a class in $H^1(G_v, M)$ to the extension $\nabla_{\Sigma}^{\Sigma', -}$ using the chain of isomorphisms (\ref{e:ext}).  Let $w \in \Sigma_H$ lie over the place $v \in \Sigma$, and consider the element $\frac{1}{2}(w - \overline{w}) \in X_{H,\Sigma}^-$, where $\overline{w}$ denotes the image of $w$ under  complex conjugation.  Let $x$ denote a lift of this element to $\nabla_{\Sigma}^{\Sigma', -}$ under the surjection given by (\ref{e:nablaextminus}).  For any $g \in G_v$ we define 
\begin{equation}
\gamma_v(g) = gx - x \in M.
\label{e:alphadef}
\end{equation}
  This defines a cocycle representing a class in $H^1(G_v, M)$ that does not depend on the choice of $x$.  The tuple $(\gamma_v)_{v \in \Sigma}$ is  associated to $\nabla_{\Sigma}^{\Sigma', -}$ under (\ref{e:ext}).  

\bigskip
In \S\ref{s:extgal} we prove the following characterization of the Selmer module $\nabla_{\Sigma}^{\Sigma',^-}$.

\begin{lemma} \label{l:ext}  Under the isomorphism (\ref{e:ext}), the extension class in $\Ext_{\Z[G]^{-}}^{1}( X_{H,\Sigma}^{-} , M)$ determined by $\nabla_{\Sigma}^{\Sigma',-}$ corresponding to the minus part of the exact sequence (\ref{e:nablaext}) is equal to the tuple of canonical classes $(\lambda_v)_{v \in \Sigma}$ defined in (\ref{e:lvdef}).
\end{lemma}

\section{Group ring valued Hilbert Modular Forms}
 \label{s:hmf}

In the remainder of the paper, we will use Ribet's method in the context of group ring valued Hilbert modular forms to prove the inclusion
\[ \Fitt_R(\Sel_{\Sigma}^{\Sigma'}(H)_R) \subset (\Theta^\#), \qquad \Theta = \Theta_{\Sigma, \Sigma'}, \]
in Theorem~\ref{t:include} from which all of our main theorems were deduced.  Here $R = \cO[G_p]_\chi$ is the component of $\cO[G]$ corresponding to the totally odd character $\chi$. 

\subsection{Replacing $R$ by its trivial zero free quotient} \label{s:replacer}

 In our constructions it will be convenient if $\Theta^\#$ is a non-zerodivisor in $R$.  In the present context, this may not be the case.  Indeed, if there is a character $\psi$ of $G$ belonging to $\chi$ and an element $v \in \Sigma$ such that $\psi(v)=1$, then the associated $L$-function has a trivial zero: $L_{\Sigma, \Sigma'}(\psi, 0) = 0$.  To deal with this, we will replace the component $\cO[G_p]_\chi$ with its quotient $R_\Psi$, the character group ring corresponding to characters $\psi$ without a trivial zero:
\[ \Psi = \{ \psi \in \hat{G} \colon \psi|_{G'} = \chi, \psi(v) \neq 1 \text{ for all } v \in \Sigma\}. \]
We show it suffices to consider this quotient.

\begin{lemma}  Let $R = \cO[G_p]_\chi$, and let $R_\Psi$ be the character group ring quotient of $R$ associated to the set $\Psi$ above.
Suppose that 
\begin{equation} \label{e:rpsi}
 \Fitt_{R_\Psi}(\Sel_{\Sigma}^{\Sigma'}(H)_{R_\Psi}) \subset (\Theta^\#). 
 \end{equation}
Then \begin{equation} \label{e:rpsipsi}
 \Fitt_R(\Sel_{\Sigma}^{\Sigma'}(H)_R) \subset (\Theta^\#). \end{equation}
\end{lemma}
\begin{proof}  Let $R_{\Psi'}$ be the character group ring quotient of $R$ associated to the set of characters with trivial zeroes:
\[ \Psi' = \{ \psi \in \hat{G} \colon \psi|_{G'} = \chi, \psi(v) = 1 \text{ for some } v \in \Sigma\}. \]
There is a canonical injection \[ \iota \colon R \longrightarrow R_\Psi \times R_{\Psi'}, \qquad \text{ denoted } \iota(x) = (\iota_1(x), \iota_2(x)).\]
By Corollary~\ref{c:selprin}, we can write $\Fitt_R(\Sel_{\Sigma}^{\Sigma'}(H)_R) = (x)$ for some $x \in R$. By Lemma~\ref{l:trivial}, we have
\[ \iota_2(x) = 0 = \iota_2(\Theta^\#). \]

The given inclusion (\ref{e:rpsi}) implies that there exists $y \in R_{\Psi}$ such that $\iota_1(x) = y \cdot \iota_1(\Theta^\#).$  Let $\tilde{y}$ be any lift of $y$ to $R$.  We then have that $x - \tilde{y} \cdot\Theta^\#$ vanishes under both $\iota_1$ and $\iota_2$.  It follows that $x = \tilde{y} \cdot\Theta^\#$, giving the desired result (\ref{e:rpsipsi}).
\end{proof}

For the rest of the paper, we will work with the ``trivial zero free character group ring quotient" $R_\Psi$ of the component $\cO[G_p]_\chi$.  For notational simplicity, we will simply write $R$ for this ring $R_\Psi$.  The image of $\Theta^\#$ is a non-zerodivisor in $R$.

\subsection{Definitions and notations on Hilbert modular forms} 

The rest of this section sets up required notation of Hilbert modular forms. The reader familiar with it from \cite{dka} or \cite{dkv} may skip this and move to section \ref{s:cusp}. We follow the definitions of Shimura \cite{shim} for the space of classical Hilbert modular forms over the totally real field $F$ (see also \cite{ddp}*{\S2.1}). Here we recall certain aspects of this definition and set up notation.

\subsubsection{Hilbert modular forms} Let $\cH$ denote the complex upper half plane endowed with the usual action of $\GL_2^+(\R)$ via linear fractional transformations, where $\GL_2^+$ denotes the group of matrices with positive determinant.  We fix an ordering of the $n$ embeddings $F \hookrightarrow \R$, which yields an embedding $\GL_2^+(F) \hookrightarrow \GL_2^+(\R)^n$ and hence an action of $\GL_2^+(F)$ on $\cH^n$.  Here $\GL_2^+(F)$ denotes the group of matrices with totally positive determinant.

For each class $\lambda$ in the narrow class group $\Cl^+(F)$, we choose a representative fractional ideal $\ft_\lambda$.  Let $\fn \subset \cO_F$ be an ideal. Define 
\begin{align*}
 \Gamma_{\lambda}(\fn) 
= & \left\{  \mat{a }{ b}{ c }{ d } \in \GL_2^+(F): a,d \in \cO_F, c \in \ft_{\lambda} \fd \fn,  \right.  \\
& \left. \ \ \ \ \ \ \   b \in (\ft_{\lambda}\fd)^{-1}, ad-bc \in \cO_F^*, \ d \equiv 1 \!\!\!\! \pmod{\fn} \right\}.
 \end{align*}
Here $\fd$ denotes the different of $F$. 

Let $k$ be a positive integer.  We denote by $M_k(\fn)$ the space of Hilbert modular forms for $F$ of level $\fn$ and weight $k$.  Each element $f \in M_k(\fn)$ is a tuple $f = (f_\lambda)_{\lambda \in \Cl^+(F)}$ of holomorphic functions $f_\lambda\colon \cH^n \longrightarrow \C$ 
such that $f_\lambda|_{\alpha,k } = f_\lambda$ for all $\lambda \in \Cl^+(F)$ and $\alpha \in \Gamma_{\lambda}(\fn)$.  Here the weight $k$ slash action is defined in the usual way:
\[ f_\lambda|_{\alpha,k }(z_1, \dotsc, z_n) = \N(\det(\alpha))^{k/2} \prod_{i=1}^n (c_i z_i + d_i)^{-k}  f_\lambda\left( \frac{a_1 z_1 + b_1}{c_1 z_1 + d_1}, \dotsc, \frac{a_n z_n + b_n}{c_n z_n + d_n}\right), \]
where $a_i$ denotes the image of $a$ under the $i$th real embedding of $F$ and similarly for $b_i, c_i, d_i$.

\subsubsection{Hecke Operators}

The space $M_k(\fn)$ is endowed with the action of a Hecke algebra generated by the following operators:
\begin{itemize}
\item $T_\fq$ for primes $\fq \nmid \fn$.
\item $U_\fq$ for primes $\fq \mid \fn$.
\item The ``diamond operators" $S(\fm)$ for  each class $\fm \in G_\fn^+ = $ narrow ray class group of $F$ of conductor $\fn$.
\end{itemize}
We refer to \cite{shim}*{\S2} for the definition of these Hecke operators.

\subsubsection{Cusps, $q$-expansions, and cusp forms} The set of cusps of $\Gamma_{\lambda}(\fn)$
 is by definition the finite set
\begin{equation} \label{e:cuspsdef}
\cusps(\Gamma_{\lambda}(\fn)) = \Gamma_{\lambda}(\fn) \backslash \GL_2^+(F) / \left\{ \mat{a}{b}{0}{d} \in \GL_2^+(F) \right\} \leftrightarrow \Gamma_{\lambda}(\fn) \backslash \PP^1(F). 
\end{equation}
The bijection in (\ref{e:cuspsdef}) is $\begin{psmallmatrix}a & b\\ c & d \end{psmallmatrix} \mapsto (a:c)$.
We define 
\begin{equation} \label{e:disjoint} \cusps(\fn) = \bigsqcup_{\lambda} \cusps(\Gamma_{\lambda}(\fn)). 
\end{equation}
A pair $\cA = (A, \lambda)$ with  $A \in \GL_2^+(F)$ and $\lambda \in \Cl^+(F)$ therefore gives rise to a cusp that we denote
$[\cA] \in \cusps(\fn)$, corresponding to the image of the matrix $A$ in $\cusps(\Gamma_{\lambda}(\fn))$ in the $\lambda$-component of the disjoint union (\ref{e:disjoint}).

Given $f = (f_\lambda) \in M_k(\fn)$ and a pair $\cA = (A, \lambda)$, the function $f_\lambda|_{A, k}$ has a Fourier expansion
\begin{equation} \label{e:fourier}
 f_\lambda|_{A, k}(z) = a_{\cA}(0) + \sum_{\stack{b \in \fa}{b \gg 0}} a_{\cA}(b) e_F(bz), \end{equation}
where $\fa$ is a certain lattice in $F$ depending on $\cA$, the notation $b \gg 0$ means that $b$ is totally positive, and
\[ e_F(bz) = \exp(2 \pi i (b_1 z_1 + \cdots +  b_n z_n)). \]
Here $b_i$ is the image in $\R$ of $b$ under the $i$th real embedding of $F$.

We normalize these Fourier coefficients as follows.  Write $A =\begin{psmallmatrix}a & *\\ c & * \end{psmallmatrix}$ and define the fractional ideal
\[ \fb_\cA = a \cO_F + c(\ft_\lambda \fd)^{-1}. \]  Define

\[
 c_{\cA}(b, f) = a_{\cA}(b) (\N \ft_\lambda)^{-k/2}(\N \fb_\cA)^{-k} (\N (\det A))^{k/2}. 
 \]

The subspace of {\em cusp forms} $S_k(\fn) \subset M_k(\fn)$ is the space of $f = (f_\lambda) \in M_k(\fn)$ such that 
$c_{\cA}(0, f) = 0$ for all pairs $\cA = (A, \lambda)$.
Note that the definition of this subspace does not depend on the choice of ideal class representatives $\ft_\lambda$.

When $k$ is even, the normalized constant term $c_{\cA}(0, f)$ depends only on the cusp $[\cA] \in \cusps(\fn)$ determined by $\cA$ (this motivates our normalizations).  When $k$ is odd, this is {\em almost} true---it holds up to sign.  In this case $c_\cA(0, f)$ is still invariant if $A$ is multiplied on the left by an element of $\Gamma_\lambda(\fn)$.  But if $A' = \begin{psmallmatrix}a & b\\ 0 & d \end{psmallmatrix} \in \GL_2^+(F)$ then
\[ c_{(AA', \lambda)}(0, f) = \sgn(\Norm_{F/\Q}(a)) \cdot c_{(A, \lambda)}(0, f). \]

\subsubsection{$q$-expansions}

When $A = 1$ we write simply
\begin{equation} \label{e:constantdef}
 c_\lambda(0, f) = a_{(1,\lambda)}(0) \cdot  (\N \ft_\lambda)^{-k/2}. \end{equation} 
 Furthermore in this case, the lattice $\fa$ appearing in (\ref{e:fourier}) is the ideal $\ft_\lambda$.
Any nonzero integral ideal $\fm$ may be written $\fm = (b)\ft_\lambda^{-1}$ with $b  \in \ft_\lambda$ totally positive for a unique 
$\lambda \in \Cl^+(F)$. We define the normalized Fourier coefficient
\begin{equation} \label{e:positivedef}
 c(\fm, f) = a_{(1,\lambda)}(b) (\N\ft_\lambda)^{-k/2}. \end{equation}
The collection of normalized Fourier coefficients $\{c_\lambda(0, f), c(\fm, f)\}$ is called the {\em $q$-expansion} of $f$ and determines the form $f$.

\subsubsection{Cusps above infinity and zero}

We recall some notation from \cite{dka}.  Given a pair $\cA = (A, \lambda)$ with $A = \begin{psmallmatrix}a & b\\ c & d \end{psmallmatrix}$, we define the integral ideal 
\[ \fc_\cA = (c)(\ft_\lambda \fd \fb_\cA)^{-1} \subset \cO_F.  \]
The ideal $\fc_\cA$ depends only on the cusp $[\cA]$ associated to $\cA$.
As intuition for this definition, consider the case $F=\Q$.  If $\cA$ represents the cusp $a/c \in \PP^1(\Q)$ where $a$ and $c$ are relatively prime integers, then $\fc_\cA \subset \Z$ is the ideal generated by $c$.

We denote by $C_\infty(\fn) \subset \cusps(\fn)$  the set of cusps  $[\cA]$ such that $\fn \mid \fc_\cA$ and more generally for $\fb \mid \fn$ we define 
\[ C_\infty(\fb, \fn) = \{ [\cA] \in \cusps(\fn) \colon \fb \mid \fc_\cA \}. \]

Similarly, we let $C_0(\fn)$ denote the set of cusps $[\cA] \in \cusps(\fn)$ such that $\gcd(\fc_\cA, \fn) = 1$ and more generally for $\fb \mid \fn$ we define
\[  C_0(\fb, \fn) = \{ [\cA] \in \cusps(\fn) \colon \gcd(\fb, \fc_\cA)=1 \}. \] 
 The sets $C_\infty(\fb, \fn)$ and $C_0(\fb, \fn)$ are stable under the action of the diamond operators $S(\fm)$.  These sets are enumerated in \cite{dka}.
 
\subsubsection{Forms with Nebentypus}

Recall that $G_\fn^+$ denotes the narrow ray class group of $F$ attached to the conductor $\fn$.  Write $h_\fn^+ = \# G_\fn^+$.
Let $\psi$ denote a character $ G_\fn^+ \longrightarrow \C^*$ whose associated sign is congruent to $(k, k, \dotsc, k)$ in $(\Z/2\Z)^n$, i.e.\ such that if $\alpha \in \cO_F$ with $\alpha \equiv 1 \pmod{\fn}$, we have \[ \psi((\alpha)) = \sgn(\Norm_{F/\Q}(\alpha))^k. \]  A form $f \in M_k(\fn)$ is said to have nebentypus $\psi$ if
\[  f|_{S(\fa)} = \psi(\fa) f \] 
for all $\fa \in G_\fn^+$.  The space of forms with nebentypus $\psi$ is denoted $M_k(\fn, \psi)$, and we let $S_k(\fn, \psi) = M_k(\fn, \psi) \cap S_k(\fn)$. We have decompositions
\[  
M_k(\fn) = \bigoplus_{\psi} M_k(\fn, \psi), \qquad S_k(\fn) = \bigoplus_\psi S_k(\fn, \psi). 
\]

\subsubsection{Raising the level} \label{sss:level}

For a Hilbert modular form $f \in M_k(\fn)$ and an integral ideal $\fq$ of $F$, there is a form \[ f|\fq \in M_k(\fn \fq) \] characterized by the fact that for nonzero integral ideals $\fa$ we have
\[ c(\fa, f|\fq) = \begin{cases}
c(\fa/\fq, f) & \text{if } \fq \mid \fa \\
0 & \text{if } \fq \nmid \fa
\end{cases} \]
and 
\begin{equation} \label{eq:levelconstant} 
c_\lambda(0, f|\fq) = c_{\lambda \fq}(0, f)
\end{equation}
for all $\lambda \in \Cl^+(F)$.  For the construction of $f|\fq$ see \cite{shim}*{Prop 2.3}. 

\subsubsection{Group ring valued Hilbert modular forms} 

Define $M_k(\fn, \Z) \subset M_k(\fn)$ to be the subgroup of forms $f$ such that
\[ c(\fm, f) \in \Z \text{ for all nonzero } \fm \subset \cO_F, \qquad c_\lambda(0, f) \in \Z \text{ for all } \lambda \in \Cl^+(F). \]
For any abelian group $A$, define
\[ M_k(\fn, A) = M_k(\fn, \Z) \otimes A. \]
Now suppose that $A$ is a ring and that $\psi \colon G_\fn^+ \longrightarrow A^*$ is a character.  We define the forms of nebentypus $\psi$ by
\[ M_k(\fn, A, \psi) = \{ f \in M_k(\fn, A) \colon f|_{S(\fa)} = \psi(\fa) f \text{ for all } \fa \in G_\fn^+ \}. \]

These definitions generalize in the obvious way to yield $S_k(\fn, A)$ and $S_k(\fn, A, \psi).$
We are particularly interested in the case where $A$ is the ring $R = R_\Psi$ as in \S\ref{s:replacer}.
If the extension $H/F$ has conductor dividing $\fn$, then $G = \Gal(H/F)$ is canonically a quotient of the narrow ray class group $G_\fn^+$. We define
\[ \begin{tikzcd}
\bpsi \colon G_\fn^+ \ar[r] & G \ar[r] & R^* 
\end{tikzcd}
\]
to be the canonical character. 
  The space of ``group ring valued Hilbert modular forms" $M_k(\fn, R, \bpsi)$ 
was first considered by Wiles \cite{wiles}.   In practice, we will define such forms by specifying their Fourier coefficients, as described by the following lemma.

 \begin{lemma}  Let $c(\fm) \in R$ for  $\fm \subset \cO_F, \fm \neq 0$ and $c_\lambda(0) \in R$ for  $\lambda \in \Cl^+(F)$ be a collection of elements of $R$ such that for all $\psi \in \Psi$, there exists a form $f_\psi \in M_k(\fn, \cO, \psi)$ with
 \[ c(\fm, f_\psi) = \psi(c(\fm)), \qquad c_\lambda(0, f_\psi) = \psi(c_\lambda(0)). \]
 Then there exists a unique $f \in M_k(\fn, R, \bpsi)$ such that $\psi(f) = f_\psi$ for all $\psi \in \Psi$.
 \end{lemma}
 
\begin{proof} Recall that there is an embedding
\begin{equation} \label{e:rembed}
\begin{tikzcd}
R \ar[r,hook] & \displaystyle\prod_{\psi \in \Psi} \cO,
\end{tikzcd}
 \qquad x \mapsto (\psi(x))_{\psi \in \Psi}. \end{equation}
The lemma follows from an important result of Silliman \cite{dks}*{Corollary 7.28 and Remark 7.29}, which implies that
\begin{equation} \label{e:jinject}
 M_k(\fn, R) = \{ f \in M_k(\fn, \prod_{\psi \in \Psi} \cO) \colon c(\fm, f), c_\lambda(0, f) \in R \text{ for all } \fm, \lambda\}. 
 \end{equation}
Now
\[ M_k(\fn, \prod_{\psi \in \Psi} \cO) = \prod_{\psi \in \Psi} M_k(\fn, \cO), \]
and we can define $f \in M_k(\fn, \prod_{\psi \in \Psi} \cO)$ to be the form corresponding to the tuple $(f_\psi)$ under this identification.  Then:
\begin{align}
c(\fm, f) &=  (c(\fm, f_\psi))_{\psi} = (\psi( c(\fm)))_\psi  \label{e:cmdef} \\
c_\lambda(0, f) &= (c_\lambda(0, f_\psi))_{\psi} = (\psi(c_\lambda(0)))_\psi. \label{e:cldef}
\end{align}
The elements on the right side of (\ref{e:cmdef}) and (\ref{e:cldef}) are the images of $c(\fm)$ and $c_\lambda(0)$ under the embedding (\ref{e:rembed}), respectively. By (\ref{e:jinject}), it follows that $f \in M_k(\fn, R)$.  The fact that $f \in M_k(\fn, R, \bpsi)$ now follows since $\psi(f) = f_\psi \in M_k(\fn, \cO, \psi)$.
\end{proof}

\begin{remark}  As this proof shows, a group ring valued modular form $f$ over  $R = R_\Psi$ can be viewed as encoding the {\em family} of modular forms $\{\psi(f)\}$ indexed by the characters $\psi \in \Psi$.   The fact that the Fourier coefficients of $f$ lie in $R$, rather than just $\prod_{\psi \in \Psi} \cO$, implies that the forms $\psi(f)$ satisfy certain $p$-adic congruences.
\end{remark}

The Hecke operators $T_\fq$ for prime $\fq \nmid \fn$, $U_\fq$ for prime $\fq \mid \fn$, and $S(\fm)$ for $(\fm, \fn) =1 $ preserve the space $M_k(\fn, R, \bpsi)$.  To see this, note first that $S(\fm)$ acts by $\bpsi(\fm) \in R^*$.  For prime $\fq \nmid \fn$, we have the formulas:
\begin{align}
c(\fm, f|_{T_\fq}) &= \sum_{ \fa \mid (\fm, \fq)} \bpsi(\fa) \N\fa^{k-1} c(\fm \fq/\fa^2, f), \label{e:heckec} \\
c_\lambda(0, f|_{T_\fq}) &= c_{\lambda \fq^{-1}}(0, f) + \bpsi(\fq) \N\fq^{k-1} c_{\lambda \fq}(0, f), \nonumber
\end{align}
which show that $T_\fq$ preserves $M_k(\fn, R, \bpsi)$.  In fact the same formulas hold for $U_\fq$ when $\fq \mid \fn$ with the convention that $\bpsi(\fq) = 0$, implying that $U_\fq$ preserves $M_k(\fn, R, \bpsi)$ as well.

\subsubsection{Ordinary forms}

The ring $R = R_\Psi$ is a complete local $\Z_p$-algebra. Let $\fP = \gcd(p^\infty, \fn)$ denote the $p$-part of $\fn$.
Let $\fp \mid \fP$.  Following Hida, we define the ordinary operators
\[ e_\fp^{\ord} = \lim_{n \rightarrow \infty} U_\fp^{n!}, \qquad e_\fP^{\ord} = \prod_{\fp \mid \fP} e_{\fp}, \qquad e_p^{\ord} = \prod_{\fp \mid p} e_{\fp}. \]
For any character $\psi\colon G_\fn^+ \longrightarrow R^*$ we have the spaces of $p$-ordinary forms:
\[ M_k(\fn, R, \bpsi)^{p\text{-ord}} = e_{p}^{\ord}  M_k(\fn, R, \bpsi), \qquad S_k(\fn, R, \bpsi)^{p\text{-ord}} = e_p^{\ord}  S_k(\fn, R, \bpsi). \]
By construction, the operator $U_\fp$ acts invertibly on the space of $p$-ordinary modular forms for each $\fp \mid p$.

\subsection{Eisenstein series} \label{s:eisenstein}

Let $k \ge 1$ be an odd integer and let $\psi\colon G \longrightarrow \cO^*$ be a totally odd character. Let $S$ be a  finite set of places of   $F$.  We denote by $\psi_S$ the character $\psi$ viewed as having modulus divisible by all finite primes in $S$, i.e.\ $\psi(\fa) = 0$ if $\fa$ is divisible by a prime in $S$.  If $\fn$ is the product of $\cond(\psi)$ and the  primes in $S$ not dividing $\cond(\psi)$, then there is an ``$S$-stabilized" Eisenstein series  $E_k(\psi_S, 1) \in M_k(\fn, \Frac\cO, \psi)$
with Fourier coefficients given by 
\[ c(\fm, E_k(\psi_S, 1)) = \sum_{\fr \mid \fm} \psi_S\left(\frac{\fm}{\fr}\right) \N\fr^{k-1}. \]
If $k > 1$ and $\fn \neq 1$, we have $c_\lambda(0, E_k(\psi_S, 1) ) = 0$.
If $k > 1$ and $\fn = 1$, we have
\[ c_\lambda(0, E_k(\psi_S, 1) ) = \frac{1}{2^n}\psi^{-1}(\lambda) L(\psi^{-1}, 1-k). \]
If $k=1$, then
\begin{equation} 
c_{\lambda}(0, E_1(\psi_S, 1)) = 2^{-n} \cdot \begin{cases}  L(\psi_S, 0) & \text{ if } \fn \neq 1, \\
L(\psi, 0) + \psi^{-1}(\lambda) L( \psi^{-1}, 0) & \text{ if } \fn = 1.
\end{cases}
\end{equation}
The Eisenstein series $E_k(\psi_S,1)$ is an eigenvector for the Hecke operators with eigenvalues given by the corresponding Fourier coefficients, i.e.
\begin{itemize}
\item $T_\fl$ acts as $\psi(\fl)+ \N\fl^{k-1}$ for $\fl \nmid \fn$ 
\item $U_\fl$ acts as $\N\fl^{k-1}$ for $\fl \mid \fn$.
\end{itemize}

These Eisenstein series nearly fit into group ring families: the non-constant coefficients belong to the group ring but the constant terms only lie in the fraction field. Let $R$ denote a character group ring associated to $G$, and let $\bpsi\colon G_\fn^+ \longrightarrow G \longrightarrow R$ denote the canonical character.  Let $S$ denote the set of primes dividing $\fn$.
There is an Eisenstein series $E_k(\bpsi, 1)$ whose specialization at a character $\psi$  is $E_k(\psi_S, 1)$. 
The group ring form $E_k(\bpsi, 1)$ has $q$-expansion coefficients
\begin{equation} \label{e:ekgp}
c(\fm, E_k(\bpsi, 1)) = \sum_{\substack{\fr \mid \fm \\ (\fm/\fr, \fn)=1}} \bpsi\left( \frac{\fm}{\fr}\right) \N\fr^{k-1} \in R.
\end{equation}
Let $H_{\fn}$ denote the narrow ray class field of conductor $\fn$, so $\Gal(H_\fn/F) \cong G_\fn^+$. 
Let $\Theta_S$ denote the image in $\Frac(R)$ of $\Theta^{H_{\fn}/F}_{S} \in \Q[G_\fn^+]$, the $S$-depleted Stickelberger element for the extension $H_{\fn}/F$.
The constant terms of $E_k(\bpsi, 1)$ lie in $\Frac(R)$ and are given by
\[ c_\lambda(0, E_k(\bpsi, 1)) = 
2^{-n} \cdot \begin{cases} 
0 & \text{if } k > 1 \text{ and } \fn \neq 1 \\
\bpsi^{-1}(\lambda) \Theta_S(1-k) & \text{if } k > 1 \text{ and } \fn = 1 \\
 \Theta_S^\#(0) & \text{if } k=1 \text{ and  } \fn \neq 1, \\
\Theta_S^\#(0) + \bpsi^{-1}(\lambda) \Theta_S(0) & \text{if } k = 1 \text{ and } \fn = 1.
\end{cases}
\]

\section{Construction of  cusp forms} \label{s:cusp}

In this section we  apply certain results appearing in the papers \cite{dka}, \cite{dks} to construct a group ring valued cusp form congruent to an Eisenstein series.  First we note the following elementary lemma.
Following the notation of \S\ref{s:hmf}, for the rest of the paper an unadorned $\Theta$ denotes $\Theta_{\Sigma, \Sigma'}$.

\begin{lemma} \label{l:largem}
 For sufficiently large positive integers $m$, the Stickelberger element $\Theta^\#$ divides $p^m$ in $R$.
\end{lemma}

\begin{proof}  Recall that in \S\ref{s:replacer} we replaced $R$ by a character group ring quotient in which $\Theta^\#$ is not a zerodivisor.
Therefore we can consider $(\Theta^\#)^{-1} \in \Frac(R).$  For sufficiently large $m$, we have $z = p^m (\Theta^\#)^{-1} \in R$, since $\Frac(R) = R \otimes_{\Z_p} \Q_p$.  Therefore $\Theta^\# \cdot z = p^m$ with $z \in R$ as desired.
\end{proof}

For the remainder of the paper, we work modulo $p^m$ where $m$ satisfies Lemma~\ref{l:largem}.  We will later require $m$ to be sufficiently large to satisfy other conditions that will arise.
We also choose a positive integer $k$ such that $k \equiv 1\pmod{(p-1)p^N}$ for a sufficiently large integer $N > m$.  These notions of ``sufficiently large" will become apparent as we use them in our proofs.

\subsection{Construction of modified Eisenstein series} \label{s:modified}

We introduce some notation.
Let $\psi$ be a totally odd character of $G_F$ and $k \ge 1$ an odd integer. Write 
 $\fc_0 = \cond(\psi)$. Let $T$ be a finite nonempty set of primes not dividing $\fc_0$. Write \[  \ft = \prod_{\fl \in T} \fl. \]
  Let $\fP$ be an integral ideal coprime to $\ft$. Put \[\fc = \lcm(\fc_0, \fP), \qquad \fn = \fc \ft. \]

The following construction is of central importance in this paper.  We define a certain linear combination $W_k(\psi_\fP, 1)$ of Eisenstein series that satisfies the following:
\begin{itemize}
\item The $T$-smoothed $L$-function $L_{S_\fP, T}(\psi, 0)$ appears in the constant terms at infinity $c_\lambda(0)$. 
\item  The constant terms at all $p$-unramified cusps vary nicely with respect to the weight.  More precisely, there is a single constant $c_0 = L(\psi^{-1}, 1-k)/L(\psi^{-1}, 0)$, independent of cusp, such that the ratio of the normalized constant terms at these cusps for $W_k$ and $W_1$ is $p$-adically very close to $c_0$. 
\item The forms $W_k$ interpolate into a group ring family.
\end{itemize}

In a fixed level $\fn$, the Fourier coefficients (and constant terms at non-infinite cusps) of Eisenstein series behave differently for characters of different conductor dividing $\fn$.  One miracle regarding the forms $W_k(\psi_\fP, 1)$ is that there is a single group ring form that interpolates all of these forms regardless of the conductor of $\psi$.  For example, this is not the case for the unmodifed Eisenstein series $E_k(\psi_\fP, 1)$---note that the group ring form defined in (\ref{e:ekgp}) interpolates the $S$-depleted forms $E_k(\psi_S, 1)$ rather than the primitive forms $E_k(\psi, 1)$.  Our construction is only robust enough to handle primes in $S$ not dividing $p$, which is why we still deplete at $\fP$.

\begin{definition}
With notation as above, let
\[ 
{W}_k(\psi_{\fP}, 1) = \sum_{\fm \mid \ft}  \mu(\fm)\psi(\fm) \N\fm^k E_k(\psi_{\fP},1)|_{\fm}  \in M_k(\fn, \psi).
\]
\end{definition}

The goal of the remainder of this section is to compute the constant terms of ${W}_k(\psi_{\fP}, 1)$ at all cusps for  odd $k \ge 1$.
 Let $\cA = (A, \lambda)$ with $A =\begin{psmallmatrix}a & b\\ c & d \end{psmallmatrix}  \in \GL_2^+(F)$ and $\lambda \in \Cl^+(F)$.
 
 \begin{definition}
   If  $[\cA] \in C_0(\fc_0, \fn)$ and $\fm\mid \ft \fP$, we put $J_{\fm}$ (respectively, $J_{\fm}^c$) for the set of  prime divisors $\fl \mid \fm$ such that $[\cA] \in C_0(\fl, \fn)$ (respectively $[\cA] \in C_\infty(\fl, \fn)$). \end{definition}

\medskip

The following result is proved in \cite{dka}*{Theorem 4.7}.  Here $\tau(\psi)$ denotes the Gauss sum defined in \cite{dka}*{Definition 4.1}.

\begin{prop} \label{p:enop} Let $\fm$ be a divisor of $\ft$. The normalized constant terms $c_{\cA}(0)$ of  $E_k(\psi_{\fP}, 1)|_{\fm}$ as an element of $M_k(\fn, \psi)$ are as follows.
\begin{itemize}
\item Suppose that $k > 1$.
\begin{itemize}
\item The constant term at $\cA$ is zero if $[\cA] \not\in C_0(\fc, \fn)$.
\item  If    $[\cA] \in C_0(\fc, \fn)$, the normalized constant term at $\cA$  is 
\begin{equation} \label{e:ctepsi}
\frac{\tau(\psi)}{{\N\fc_0}^{k}} \sgn(\N(-c)) \psi(\fc_\cA) \frac{L(\psi^{-1}, 1-k)}{2^n}  \prod_{\fp \mid \fP} \left(1- \frac{\psi(\fp)}{\N\fp^k} \right)  \prod_{\fl \in J_\fm} \N\fl^{-k} \prod_{\fl \in J_{\fm}^c} \psi^{-1}(\fl). 
\end{equation}
\end{itemize}
\item Suppose $k=1$. 
\begin{itemize}
\item The constant term at $\cA$ is zero if $[\cA] \notin C_0(\fc, \fn) \cup C_{\infty}(\fc_0, \fn)$.
\item If $[\cA] \in C_0(\fc, \fn) \cap C_{\infty}(\fc_0, \fn)$ (note that this can happen only when $\fc_0 =1$), then the constant term at $\cA$ is 
\begin{align*}
& \tau(\psi) \psi^{-1}(\fd \ft_{\lambda} \fb_\cA) \frac{L(\psi^{-1}, 0)}{2^n}\prod_{\fp \mid \fP} \left( 1 - \frac{\psi(\fp)}{\N\fp}\right)\prod_{\fl \in J_{\fm}} \N \fl^{-1} \prod_{\fl \in J_{\fm}^c} \psi^{-1}(\fl) \\
+ & \psi(\fb_{\cA}) \frac{L(\psi, 0)}{2^n} \prod_{\fp \mid \fP} (1 - \N\fp^{-1}) \prod_{\fl \in J_{\fm}}(\psi(\fl) \N\fl)^{-1}.
\end{align*}
\item If $[\cA] \in C_\infty(\fc_0, \fn) \setminus C_0(\fc, \fn)$, the normalized constant term at $\cA$ is
\[
 \psi(\fb_{\cA}) \frac{L(\psi, 0)}{2^n} \prod_{\fp \mid \fP} (1 - \N\fp^{-1}) \prod_{\fl \in J_{\fm}}(\psi(\fl) \N\fl)^{-1}.
\]
\item If $[\cA] \in C_0(\fc, \fn) \setminus C_{\infty}(\fc_0, \fn)$, then the normalized constant term at $\cA$ is 
\[
\frac{\tau(\psi)}{\N \fc_0} \sgn(N(-c)) \psi(c_{\cA}) \frac{L(\psi^{-1}, 0)}{2^n} \prod_{\fp \mid \fP} \left(1 - \frac{\psi(\fp)}{\N\fp} \right) \prod_{\fl \in J_{\fm}} \N\fl ^{-1} \prod_{\fl \in J_{\fm}^c} \psi^{-1}(\fl). 
\]

\end{itemize}
\end{itemize}
\end{prop}

\begin{remark}  When considering the expression $\psi(\fc_\cA)$, note that $[\cA] \in C_0(\fc, \fn)$ implies that $\gcd(\fc_\cA, \fc) = 1$.
Note also that one can only have $c = 0$ with $[\cA] \in C_0(\fc, \fn)$ if $\fc = 1$.  In this case, by convention the expression
$\sgn(\N c) \psi(\fc_\cA)$ in (\ref{e:ctepsi}) denotes $\psi^{-1}(\ft_\lambda \fd \fb_\cA) = \psi^{-1}(\ft_\lambda \fd (a))$, which is the value obtained if one replaces $A$ by a left $\Gamma_{1,\lambda}(\fn)$-equivalent matrix for which $c \neq 0$.  This convention will remain in force in the sequel.  More generally, if $\psi$ is a totally odd character of conductor 1, any expression $\sgn(\N x) \psi(x \fm)$ should be interpreted as $\psi(\fm)$ even if $x=0$.
\end{remark}

\begin{prop} \label{p:nop} Suppose that $k > 1$ is odd. The modular form ${W}_k(\psi_{\fP}, 1)$
has constant terms 0 outside the cusps in $C_0(\fc, \fn)$. For a cusp $[\cA] \in C_0(\fc, \fn)$,  the normalized constant term $c_{\cA}(0, {W}_k(\psi_\fP, 1))$ 
equals \[ \ft
\frac{\tau(\psi)}{{\N\fc_0}^{k}} \sgn(\N(-c)) \psi(\fc_\cA)  \frac{L(\psi^{-1}, 1-k)}{2^n}  \prod_{\fp \mid \fP}\left(1 - \frac{\psi(\fp)}{\N\fp^k} \right) \prod_{\fl \in J_{\ft}} (1- \psi(\fl)) \prod_{\fl \in J_{\ft}^c} (1 - \N\fl^k).
\]
\end{prop}
\begin{proof} This is an application of Proposition~\ref{p:enop}.  It is clear that the constant terms of $W_k(\psi_{\fP}, 1)$ are 0 outside $C_0(\fc, \fn)$. Consider $[\cA] \in C_0(\fc, \fn)$. The normalized constant term of $W_k(\psi_{\fP}, 1)$ at $\cA$ is 
\[
\frac{\tau(\psi)}{{\N\fc_0}^{k}} \sgn(\N(-c)) \psi(\fc_\cA) \frac{L(\psi^{-1}, 1-k)}{2^n} \prod_{\fp \mid \fP} \left( 1 - \frac{\psi(\fp)}{\N\fp^k}\right)  \sum_{\fm \mid \ft} \mu(\fm) \prod_{\fl \in J_{\fm}^c} \N\fl^k \prod_{\fl \in J_{\fm}} \psi(\fl).
\]
The result follows from the observation
\[
\sum_{\fm \mid \ft} \mu(\fm) \prod_{\fl \in J_{\fm}^c} \N\fl^k \prod_{\fl \in J_{\fm}} \psi(\fl) = \prod_{\fl \in J_{\ft}} (1- \psi(\fl)) \prod_{\fl \in J_{\ft}^c} (1 - \N\fl^k).
\]
\end{proof}

For $k=1$, the results of Propositions~\ref{p:enop} and~\ref{p:nop} must be slightly modified.
Even though we will only require the constant terms of $W_1(\psi_{\fP}, 1)$ at $C_\infty(\fP, \fn)$, for completeness we calculate its constant terms at all cusps.  The proof of the following proposition is another direct application of \cite{dka}*{Theorem 4.7}, similar to that of Proposition~\ref{p:nop}.

\begin{prop} \label{p:wwt1}
The normalized constant terms of 
$ {W}_1(\psi_{\fP}, 1) \in M_1(\fn, \psi)$ are as follows.
\begin{itemize}
\item Assume $\fc_0 = 1$.  
\begin{itemize}
\item If $[\cA] \in C_0(\fP, \fn) \cap C_{\infty}(\ft, \fn)$,  the normalized constant term at $\cA$ is
\begin{align*}
& \psi(\fb_\cA)\frac{L_{S_{\infty}, T}(\psi, 0)}{2^n}  \prod_{\fp \mid \fP}(1- \N\fp^{-1})  \\
&+ \tau(\psi) \psi^{-1}(\fd \ft_{\lambda} \fb_\cA) \frac{L(\psi^{-1}, 0)}{2^n}  \prod_{\fp \mid \fP} \left(1 - \frac{\psi(\fp)}{\N\fp} \right) \prod_{\fl \in T} (1- \N\fl).
\end{align*}
\item If $[\cA] \in C_0(\fP, \fn)$ but  $[\cA] \not\in C_{\infty}(\ft, \fn)$,  the normalized constant term at $\cA$ is 
\[
\tau(\psi) \psi^{-1}(\fd \ft_{\lambda} \fb_\cA) \frac{L(\psi^{-1}, 0)}{2^n}   \prod_{\fp \mid \fP} \left(1 - \frac{\psi(\fp)}{\N\fp} \right)\prod_{\fl \in J_{\ft}} (1- \psi(\fl)) \prod_{\fl \in J_{\ft}^c} (1- \N\fl).
\]
\item If $[\cA] \not\in C_0(\fP, \fn)$ and $[\cA] \in C_{\infty}(\ft, \fn)$,  the normalized constant term at $\cA$ is 
\[
\psi(\fb_\cA)\frac{L_{S_{\infty}, T}(\psi, 0)}{2^n}  \prod_{\fp \in J_{\fP}}(1- \N\fp^{-1}) \prod_{\fp \in J_{\fP}^c} (1- \psi(\fp)).
\]
\item If $[\cA] \not \in C_0(\fP, \fn)$ and $[\cA] \not\in C_\infty(\ft, \fn)$,  the normalized constant term at $\cA$ is 0.
\end{itemize}
\item Assume $\fc_0 \neq 1$. 
\begin{itemize}
\item The  constant terms are 0 outside the cusps in $C_{\infty}(\fc_0 \ft, \fn) \cup C_0(\fc, \fn)$.
\item If $[\cA] \in C_{\infty}(\fc_0\ft, \fn)$, the normalized constant term at $\cA$  is 
\[
\sgn(\N a)\psi^{-1}(a \fb_\cA^{-1})\frac{L_{S_{\infty},T}(\psi, 0)}{2^n}  \prod_{\fp \in J_{\fP}}(1- \N\fp^{-1}) \prod_{\fp \in J_{\fP}^c} (1 - \psi(\fp)).
\]
\item If $[\cA] \in C_0(\fc, \fn)$,  the normalized constant term at $\cA$ is 
\[
\frac{\tau(\psi)}{\N\fc_0} \sgn(\N(-c)) \psi(\fc_\cA)  \frac{L(\psi^{-1}, 0)}{2^n} \prod_{\fp \mid \fP} \left(1 - \frac{\psi(\fp)}{\N\fp}\right) \prod_{\fl \in J_{\ft}} (1- \psi(\fl))\prod_{\fl \in J_{\ft}^c} (1- \N\fl).
\]
\end{itemize}
\end{itemize}
\end{prop}

\subsection{Linear combinations cuspidal modulo high powers of $p$}

A result of Hida (see \cite{wilesrep}*{Lemma 1.4.2}) states the existence of a Hilbert modular form congruent to 1 modulo $p$.  In 
\cite{dks}*{Theorem 10.7} Silliman proves the following slightly refined version of this result.

\begin{theorem} \label{t:hida}  For positive integers $k \equiv 0 \pmod{(p-1)p^N}$ with $N$ sufficiently large, 
there is a modular form $V_{k} \in M_k(1, \Z_p, 1)$ such that 
$V_{k} \equiv 1 \pmod{p^m}$, and such that the normalized constant term  $c_{\cA}(0,V_k)$
 for each cusp $[\cA] \in \cusps(1)$ is congruent to $1 \pmod{p^m}$.
\end{theorem}

The congruence $V_k \equiv 1 \pmod{p^m}$ means that $c(\fm, V_{k}) \equiv 0 \pmod{p^m}$ for all nonzero ideals $\fm$ and $c_\lambda(0, V_k) \equiv 1 \pmod{p^m}$ for all $\lambda \in \Cl^+(F)$.

\bigskip

Let $\fP$ denote the $p$-part of the ideal $\fn$, i.e. $\fP = \gcd(p^\infty, \fn)$. By establishing the surjectivity of the total constant term map, Silliman proved the following result in 
 \cite{dks}*{Theorem 10.9}.

\begin{theorem}  \label{t:auxforms} 
 For sufficiently large odd positive integers $k$, there exists a group ring valued form $G_{k}(\bpsi) \in M_k(\fn, R, \bpsi)$ with normalized constant term at $\cA$  for $[\cA] \in C_{\infty}(\fn)$ equal to $\sgn(\N a)\bpsi^{-1}(a \fb_A^{-1})$, and constant term at cusps  $[\cA] \in C_\infty(\fP, \fn) \setminus C_\infty(\fn)$ equal to $0$. 
\end{theorem}

\begin{remark}  We repeat our convention that if $\fn = 1$ the expression $\sgn(\N a)\bpsi^{-1}(a \fb_A^{-1})$ is understood to equal $\bpsi(\fb_\cA)$ even if $a = 0$.
\end{remark}

\subsubsection{Case 1: $\cond(H/F)$ not divisible by primes above $p$}

For the remainder of the article, we impose the condition $k >1$. 

\begin{prop} \label{p:wnop} 
 Suppose that $\fn$ is not divisible by any primes above $p$.
  Fix a positive integer $m$.  For positive $k \equiv 1 \pmod{(p-1)p^N}$ with $N$ sufficiently large, and each character $\psi$ of $R$, the form
 \[ f_k(\psi) = {W}_1(\psi,1)V_{k-1} - \frac{L(\psi^{-1}, 0)}{L(\psi^{-1}, 1-k)}{W}_k(\psi, 1) - \frac{L_{S_\infty, T}(\psi, 0)}{2^n}  G_k(\psi)
\]
has normalized constant terms at all cusps divisible by $p^{m}$.  Here $G_k(\psi)$ denotes the specialization of the group ring form $G_k(\bpsi)$ in Theorem~\ref{t:auxforms} at the character $\psi$.
\end{prop}

\begin{proof}
The constant terms of $W_k(\psi, 1), W_1(\psi, 1), V_{k-1},$ and $G_k(\psi)$ are given explicitly by Proposition~\ref{p:nop}, Proposition~\ref{p:wwt1}, Theorem~\ref{t:hida}, and Theorem~\ref{t:auxforms}, respectively.

Since $\fP = 1$, the form $G_k(\psi)$ has constant term equal to  $\sgn(\N a)\psi^{-1}(a \fb_\cA^{-1})$  for $[\cA] \in C_{\infty}(\fn)$ and equal to 0 if $[\cA] \not \in C_{\infty}(\fn)$.

 With $\fc = \fc_0 = \cond(\psi)$, it is clear that the constant terms of $f_k(\psi)$ are 0 outside $C_0(\fc, \fn) \cup C_{\infty}(\fc, \fn)$, since the same is true for
${W}_1(\psi, 1)$,  ${W}_k(\psi, 1)$, and $ G_k(\psi)$. 

To evaluate the constant terms at other cusps, we first assume that $\fc \neq 1$. If $[\cA] \in C_{\infty}(\fc, \fn) \setminus C_\infty(\fn)$, then  ${W}_1 (\psi, 1)$,  ${W}_k(\psi, 1)$, and $ G_k(\psi)$ all have constant term 0 at $\cA$, hence $f_k(\psi)$ does as well.  If $[\cA] \in C_\infty(\fn)$, the constant term of $W_k(\psi, 1)$ at $\cA$ is 0.  Meanwhile, the constant term of $W_1(\psi, 1)$ at $\cA$ is \[
\sgn(\N a)\psi^{-1}(a \fb_\cA^{-1})\frac{L_{S_{\infty},T}(\psi, 0)}{2^n} 
\]
 and the constant term of $V_{k-1}$ is 1  modulo $p^{m}$ for positive $k \equiv 1 \pmod{(p-1)p^N}$ with $N$ sufficiently large. The constant term of $G_k(\psi)$ at $\cA$ is $\sgn(\N a)\psi^{-1}(a \fb_\cA^{-1})$.  It follows that the constant term of $f_k(\psi)$ at $\cA$ is 0 mod $p^{m}$. Therefore the constant term of  $f_k(\psi)$ is 0 mod $p^{m}$ at all cusps in $C_{\infty}(\fc, \fn)$. 

Next we consider the case  $[\cA] \in C_0(\fc, \fn)$, still maintaining the assumption $\fc \neq 1$. 
As in \S\ref{s:modified} let $J \subset T$ denote the set of primes $\fl \in T$ such that $[\cA]$ belongs to $C_0(\fl, \fn)$. The normalized constant term of  $f_k(\psi)$ at $\cA$ is 
\begin{equation} \label{e:jprod}
\begin{aligned}
\tau(\psi)\sgn(\N(-c))\psi(\fc_\cA) & \frac{L(\psi^{-1}, 0)}{2^n}\prod_{\fl \in J} (1- \psi(\fl)) \times \frac{\prod_{\fl \notin J} (1 - \N \fl)}{\N \fc_0}\\
& \left[c_{\cA}(0, V_{k-1}) - \N\fc_0^{1-k} \prod_{\fl \not \in J} \frac{1 - \N\fl^k}{1 - \N\fl}\right].
\end{aligned}
\end{equation}
The expression in brackets in (\ref{e:jprod}) $p$-adically approaches $0$: indeed for positive $k \equiv 1 \pmod{(p-1)p^N}$ with $N$ increasing the terms $c_{\cA}(0, V_{k-1})$ and $\N\fc_0^{1-k} \prod_{\fl \not \in J} \frac{1 - \N\fl^k}{1 - \N\fl}$ both approach 1. It follows that for $N$ sufficiently large,  (\ref{e:jprod}) is divisible by $p^{m}$.

Next we consider the case $\fc = 1$, so $\fn = \ft$.  
If $[\cA] \not\in C_\infty(\fn)$, the constant term of $G_k(\fn)$ at $\cA$ is 0, and 
 the normalized constant term of $f_k(\psi)$  at $\cA$ is again given by (\ref{e:jprod}). 
 If $[\cA] \in C_\infty(\fn)$, the normalized constant term of $f_k(\psi)$  at $\cA$ is (\ref{e:jprod}) plus the expression
 \[ \psi(\fb_A)\frac{L_{S_\infty, T}(\psi, 0)}{2^n} \left[ c_{\cA}(0, V_{k-1}) - 1 \right]. \]
  The result follows.
\end{proof}

\subsubsection{Case 2: $\cond(H/F)$ is divisible by some primes above $p$}

Let $\psi$ be a character of conductor $\fc_0$, with $\fc_0$ possibly divisible by some primes above $p$. Let $\fl_1, \ldots, \fl_d$ be distinct primes not dividing $\fc_0 p$. Let $\fn = \fc_0 \fl_1 \cdots \fl_d \fP'$,  with $\fP'$ the product of powers of some (but not necessarily all) primes dividing $p$. Let $\fP$ denote the $p$-part of $\fn$ i.e. $\fP = \gcd(\fn, p^{\infty})$ and put $\fc = \lcm(\fc_0, \fP)$.  We assume in this section that $\fP \neq 1$. Let $S_\fP$ denote the union of $S_\infty$ and the set of primes dividing $\fP$.

\begin{prop}  \label{p:cuspw}  For positive integers $k \equiv 1 \pmod{(p-1)p^N}$ with $N$ sufficiently large,
the form
\[  
{W}_1(\psi_\fP, 1)V_{k-1} - \frac{L_{S_\fP, T}(\psi, 0)}{2^n}  G_k(\psi)
\]
has normalized constant terms at all cusps in $C_\infty(\fP, \fn)$ divisible by $p^m$. \end{prop}

\begin{proof} By Proposition \ref{p:wwt1}, the constant terms of $W_1(\psi_{\fP}, 1)$ are supported on $C_{\infty}(\fc_0\ft, \fn) \cup C_0(\fc, \fn)$. As $\fP \neq 1$, we have 
\[
(C_{\infty}(\fc_0\ft, \fn) \cup C_0(\fc, \fn) ) \cap C_\infty(\fP, \fn) = C_{\infty}(\fn).
\]
By definition, the constant terms of $G_k(\psi)$ are also 0 at cusps in $C_\infty(\fP, \fn) \setminus C_{\infty}(\fn)$.

Suppose now that $[\cA] \in C_{\infty}(\fn)$.  By definition, the normalized constant term of $G_k(\psi)$ at $\cA$ is $\sgn(\N a) \psi^{-1}(a \fb_{\cA}^{-1})$.
The normalized constant term of $V_{k-1}$ is congruent to 1 modulo $p^m$ for $N$ sufficiently large.  By Proposition \ref{p:wwt1}, the normalized constant term of $W_1(\psi_{\fP}, 1)$ at $\cA$ is
\[
\sgn(\N a) \psi^{-1}(a \fb_{\cA}^{-1}) \frac{L_{S_{\fP}, T}(\psi, 0)}{2^n}.
\]
The result follows.
\end{proof}

\subsection{Group ring valued forms}

We now interpolate the construction of the previous section into a group ring family.  Recall our ring $R$ defined in \S\ref{s:replacer}, a quotient of a connected component $\cO[G_p]_\chi$ associated to a totally odd faithful character $\chi$ of $G'$.  
The level of our forms will be
\[ \fn = \cond(H/F) \prod_{\fl \in T} \fl \]
and as above we let \[ \fP = p\text{-part of } \fn = \gcd(p^\infty, \fn). \]

\begin{lemma} \label{l:conductor}
Let $\psi$ be a character of $R$, and let $\fc_0 = \cond(\psi)$.  Then we can write \[ \fn = \fc_0 \fl_1 \cdots \fl_d \fP' \] where
the $\fl_i$ are distinct primes not dividing $\fc_0 p$ and $\fP'$ is divisible only by primes above $p$. 
\end{lemma}
\begin{proof}
We must show that if $\fl$ is a prime not above $p$ such that $\fl^n \mid \cond(H/F)$ with $n \ge 2$ then $\fl^n \mid \cond(\psi)$.  
Let $H_p \subset H$ denote the fixed field of $G'$ and $H'$ the fixed field of $G_p$, so \[ 
\Gal(H_p/F) = G_p \text{ and  } \Gal(H'/F) = G'. 
\]  
The field $H$ is the compositum of $H_p$ and $H'$.
Since $G_p$ is a $p$-group and $\fl \nmid p$, the prime $\fl$ is at most tamely ramified in $H_p$.  Therefore if $\fl^n \mid \cond(H/F)$ with $n \ge 2$ then  $H'/F$ must have conductor divisible by $\fl^n$.  Since $\chi$ is a faithful character of $G'$, it follows that
$\fl^n \mid \cond(\chi)$.
Any character $\psi$ of $R$ can be written $\psi = \psi_p \chi$ where $\psi_p$ is a character of $G_p$.  As already noted, the $\fl$-part of the conductor of $\psi_p$ is at most $\fl$.  Therefore $\fl^n \mid \cond(\psi)$ as desired.
\end{proof}

\begin{prop} For all odd $k \geq 1$, the unique form $W_k(\bpsi, 1)  \in M_k(\fn, \Frac(R), \bpsi)$ that specializes to $W_k(\psi_{\fP}, 1)$ for all characters $\psi$ of $R$ has non-constant term $q$-expansion coefficients $c(\fm, W_k(\bpsi, 1))$ lying in $R$.
\end{prop}

\begin{proof} Let $\ft$ denote the product of all primes dividing $\fn/\fP$.  For each $\fm \mid \ft$, let 
$I_{\fm}$ be the subgroup generated by $I_v$ for $v \mid \fm$. Note that $\#I_{\fm} \mid \prod_{v \mid \fm}(1- \N v^k)$ in $\Z_p$. 

Write \[ \begin{tikzcd}
 \bpsi^{\fm} \colon G^+_{\fn/\fm} \ar[r] & G/I_{\fm} \ar[r] &  \cO[G/I_{\fm}]^* 
 \end{tikzcd}
 \]   for the canonical character corresponding to the maximal subextension of $H/F$ in which the primes dividing $\fm$ are unramified.    Let $E_k(\bpsi^{\fm}, 1) \in M_k(\fn/\fm, \Frac(\cO[G/I_\fm]), \bpsi^\fm)$ be the group ring form defined in (\ref{e:ekgp}) associated to the character $\bpsi^{\fm}$.  If $\psi$ is a character of $G$ unramified at all primes dividing $\fm$, then the specialization of $E_k(\bpsi^{\fm}, 1)$ at $\psi$ is the form $E_k(\psi_{\fn/\fm}, 1)$.

Next note that there is a canonical $\cO[G]$-module map 
  \[ \begin{tikzcd}
   \cO[G/I_\fm] \ar[r] & \cO[G] \ar[r] & R
   \end{tikzcd}
    \]
 given by $x \mapsto \N I_{\fm} \cdot \tilde{x}$, where $\tilde{x}$ is an arbitrary lift of $x$.  This map does not depend on the choice of lift.
 The image of $E_k(\bpsi^{\fm}, 1)$ under this map is a form 
 \[ \N I_{\fm} \cdot \tilde{E}_k(\bpsi^{\fm}, 1) \in M_k(\fn/\fm, \Frac(R)),\]
and all of the non-constant term $q$-expansion coefficients lie in $R$.
We define
\begin{equation} \label{e:wkdef}
W_k(\bpsi, 1) = \sum_{\fm \mid \ft} \N I_{\fm} \cdot \tilde{E}_k(\bpsi^{\fm}, 1)|_{\fm} \bpsi^{\fm}(\fm) \frac{1}{\# I_{\fm}} \prod_{v \mid \fm}(1- \N v^k).
\end{equation}
It is clear from our construction that the non-constant term $q$-expansion coefficients of $W_k(\bpsi, 1)$ lie in $R$.   To conclude the proof we must show that the specialization of $W_k(\bpsi, 1)$ at a character $\psi$ of $R$ is equal to $W_k(\psi_\fP, 1)$.

Given $\psi$, let $\fc_0 = \cond(\psi)$ and let $\ft'$ be the product of the primes dividing $\ft$ that do not divide $\fc_0$. 
By Lemma~\ref{l:conductor}, we can write $\fn = \fc_0 \ft' \fP'$ where  $\fP'$ is divisible only by primes above $p$.

Note that if $\psi$ is nontrivial on $I_{\fm}$, then $\psi(\N I_{\fm}) = 0$. Applying $\psi$ to the  sum in (\ref{e:wkdef}) gives 
\begin{align*}
& \sum_{\fm \mid \ft'}   \psi(\fm) E_k(\psi_{\frac{\ft'}{\fm} \fP}, 1)|_{\fm}\sum_{\fm' \mid \fm} \mu(\fm') \N \fm'^k \\
= & \sum_{\fm' \mid \ft'}  \mu(\fm') \N \fm'^k \psi(\fm') \sum_{\fm \mid \frac{\ft'}{\fm'}} \psi(\fm)E_k(\psi_{\frac{\ft'}{\fm'\fm}  \fP}, 1)|_{\fm'\fm} \\
= & \sum_{\fm' \mid \ft'}  \mu(\fm') \N \fm'^k \psi(\fm') \Bigg(\sum_{\fm \mid \frac{\ft'}{\fm'}} \psi(\fm)E_k(\psi_{\frac{\ft'}{\fm'\fm} \fP}, 1)|_{\fm}\Bigg)|_{\fm'} 
\end{align*}
To finish the proof that this equals $W(\psi_\fP, 1)$, we must show that  
\[
\sum_{\fm \mid \frac{\ft'}{\fm'}} \psi(\fm)E_k(\psi_{\frac{\ft'}{\fm'\fm} \fP}, 1)|_{\fm} = E_k(\psi_\fP, 1).
\]
We do this by induction on the number of prime factors of $\ft'/\fm'$. The statement is clear with $\ft' = \fm'$. Let $\ft'/\fm' = \ft_1 \cdots \ft_r$ for some $r \geq 1$. Write the  sum above as 
\begin{align*}
& \ \ \ \ \sum_{\fm \mid \fl_1 \cdots \fl_r} \psi(\fm) E_k( \psi_{\frac{\fl_1 \cdots \fl_r}{\fm}  \fP}, 1)|_{\fm} \\ 
 & =  \sum_{\fm \mid \fl_1 \cdots \fl_{r-1}} \psi(\fm) (E_k(\psi_{\frac{\fl_1 \cdots \fl_r}{\fm}  \fP}, 1)|_{\fm} +\psi(\fl_r) E_k(\psi_{\frac{\fl_1 \cdots \fl_{r-1}}{\fm}  \fP}, 1)|_{\fm\fl_r}) \\
 & =  \sum_{\fm \mid \fl_1 \cdots \fl_{r-1}} \psi(\fm) E_k(\psi_{\frac{\fl_1 \cdots \fl_{r-1}}{\fm}  \fP}, 1)|_{\fm}  \\
 & =   E_k(\psi_\fP, 1),
\end{align*}
where the last equality holds by the induction hypothesis.  
\end{proof}

The following result is proved in  \cite{dks}*{Theorem 10.10}.

\begin{theorem} \label{p:minuspm} Fix a positive integer $m \ge \ord_p(\#G_p)$.  The following holds for all sufficiently large odd integers $k$.
Let $f_\psi \in M_k(\fn, E, \psi)$ be a collection of modular forms for characters $\psi$ of $G$ belonging to $\chi$
with the property that the normalized constant terms of each $f_\psi$ at representatives for each cusp $A \in C_\infty(\fP, \fn)$ are divisible by $p^{m}$.   There exists a group ring family \[ h(\bpsi) \in M_k(\fn, \bpsi, R) \] such that each specialization $h(\psi)$ satisfies the property that \[ \tilde{f}_\psi = f_\psi - (p^{m}/\#G_p) h(\psi) \] has constant term $0$ at all cusps $A \in C_\infty(\fP, \fn)$.  If $\fP = 1$, so $C_\infty(\fP, \fn)= \cusps(\fn)$, then $\tilde{f}_{\psi}$ is cuspidal.  If $\fP \neq 1$, then $e_\fP^{\ord}(\tilde{f}_\psi)$ is cuspidal.
\end{theorem}

\begin{lemma} \label{l:xdef} Suppose we are in case 1, i.e. $\gcd(\fn, p)=1$.
For  $N$ sufficiently large and positive $k \equiv 1 \pmod{(p-1)p^N}$, the element
\[ x =  \frac{\Theta_{S_\infty}(1-k)}{ \Theta_{S_\infty}(0) } \in \Frac(R) \]
lies in $R$ and  is a non-zerodivisor satisfying
\begin{equation} x \equiv \prod_{\fp \mid p} (1 - \chi(\fp)^{-1}) \pmod{\fm_R}. \label{e:xcong}
 \end{equation}
\end{lemma}

\begin{proof}  First note that the specializations $L(\psi^{-1},0)$ of the denominator of $x$ are nonzero, so $x$ is a well-defined element of $\Frac(R)$.  The same is true of the numerator, so if we can show that $x \in R$, it will follow immediately that it is a non-zerodivisor.

Let $S_p$ denote the union of $S_\infty$ with the set of primes above $p$ in $F$.
Note that \[ \Theta_{S_p}(1 - k) = \prod_{\fp \mid p} (1 - \sigma_\fp^{-1}\N\fp^{k-1}) \Theta_{S_\infty}(1-k), 
\]
where $\sigma_\fp \in G$ denotes the Frobenius at $\fp$ (we are in case 1, where each $\fp$ above $p$ is unramified in $H/F$).
Consider the element $y(k) = \Theta_{S_p}(1 - k) - \Theta_{S_p}(0) \in \Frac(R)$.   By the theory of $p$-adic $L$-functions, this element $p$-adically approaches $0$ for positive $k \equiv 1 \pmod{(p-1)p^N}$ as $N \longrightarrow \infty$.   In particular, for positive $k \equiv 1 \pmod{(p-1)p^N}$ and $N$ sufficiently large we have that $y(k) \in R$.  Furthermore, for any positive integer $m$ we can take $N$ larger still to ensure that $y(k)$ is divisible by $p^{m}$ in $R$.  Suppose that $m$ has been chosen large enough that $p^{m}/\Theta_{S_\infty}(0) \in R$. 
Then $y(k) / \Theta_{S_\infty}(0) \in R$.
But
\begin{equation} \label{e:buty}
 \frac{y(k)}{\Theta_{S_\infty}(0)} = x  \prod_{\fp \mid p} (1 - \sigma_\fp^{-1}\N\fp^{k-1}) - \prod_{\fp \mid p} (1 - \sigma_\fp^{-1}) \in R. 
 \end{equation}
Since the Euler factors $1 - \sigma_\fp^{-1}\N\fp^{k-1}$ are units in $R$ for $k > 1$, it follows that $x \in R$ as desired.
To conclude we note that after increasing $m$ by 1 if necessary, we have that ${y(k)}/{\Theta_{S_\infty}(0)} \in \fm_R$.  The desired congruence for $x$ then follows from (\ref{e:buty}).

\end{proof}

\begin{theorem}  \label{t:fk1}
In case 1 $(\gcd(\fn, p) = 1)$, for positive $k \equiv 1 \pmod{(p-1)p^N}$ and $N$ sufficiently large, there exists a group ring form $H_k(\bpsi) \in M_k(\fn, R, \bpsi)$ such that
\[ 
\tilde{F}_k(\bpsi) = xW_1(\bpsi, 1)V_{k-1} - W_k(\bpsi, 1) -x \Theta^\#H_k(\bpsi) 
\]
lies in $S_k(\fn, R, \bpsi)$, where $x = \Theta_{S_\infty}(1-k)/ \Theta_{S_\infty}(0) \in R$ is as in Lemma~\ref{l:xdef}.
\end{theorem}

\begin{proof} 
Define
\[ f_k(\bpsi) = W_1(\bpsi, 1)V_{k-1} - \frac{1}{x}W_k(\bpsi, 1) - \frac{ \Theta^\#}{2^n} G_k(\bpsi) \in M_k(\fn, \Frac(R), \bpsi).  \]
By definition, this is a group ring form whose specialization at  a character $\psi$ of $R$ is the 
 form $f_k(\psi)$ defined in Proposition~\ref{p:wnop}.
This proposition states that for any positive integer $m$, for positive $k \equiv 1 \pmod{(p-1)p^N}$ and $N$ sufficiently large
 the constant terms of $f_k(\psi)$  are  divisible by $p^{m}$.
 Therefore by Theorem~\ref{p:minuspm} there exists a group ring form $h_k(\bpsi) \in M_k(\fn, R, \bpsi)$ such that 
\[
\tilde{f}_k(\bpsi) = f_k(\bpsi) - \frac{p^{m}}{\#G_p} h_k(\bpsi)
\]
is a cusp form. Choose $m$ large enough that $\#G_p \cdot \Theta^\#$ divides $p^{m}$ in $R$ and define
  \[ H_k(\bpsi) = \frac{G_k(\bpsi)}{2^n} + \frac{p^{m}}{\#G_p \cdot \Theta^{\#}}h_k(\bpsi) \in M_k(\fn, R, \bpsi). \] 
 The form $\tilde{F}_k(\bpsi)  = x \cdot \tilde{f}_k(\bpsi)$ is cuspidal and can be written explicitly as
\[ 
\tilde{F}_k(\bpsi) = xW_1(\bpsi, 1)V_{k-1} - W_k(\bpsi, 1) -x \Theta^\#H_k(\bpsi) \in S_k(\fn, R, \bpsi).
\]
Note that there is a small subtlety in verifying that the $q$-expansion coefficients of $\tilde{F}_k(\bpsi)$ lie in $R$.  The constant terms of $W_1(\bpsi, 1)$ only lie in $\Frac(R)$.  But the non-constant $q$-expansion coefficients of $V_{k-1}$ are highly divisible by $p$, so the contribution to the non-constant $q$-expansion coefficients of the product  $xW_1(\bpsi, 1)V_{k-1}$ will be integral for $k\equiv 1 \pmod{(p-1)p^N}$ and $N$ sufficiently large.  For the constant terms, there is nothing to check since $\tilde{F}_k(\bpsi)$ is cuspidal.
\end{proof}

In case 2, when there exist primes above $p$ dividing $\fn$, we get the following theorem.   It is proven exactly as above, using Theorem~\ref{p:minuspm} and building off of Proposition~\ref{p:cuspw} in place of Proposition~\ref{p:wnop}.

\begin{theorem}  Suppose we are in case 2, i.e.\ $\gcd(\fn, p) \neq 1$. For positive integers $k \equiv 1 \pmod{(p-1)p^N}$ and $N$ sufficiently large, there exists a group ring form $H_k(\bpsi) \in M_k(\fn, R, \bpsi)$ such that
\[ \tilde{F}_k(\bpsi) = e_\fP^{\ord}\left(W_1(\bpsi, 1)V_{k-1} - \Theta^\#H_k(\bpsi) \right) \]
lies in $S_k(\fn, R, \bpsi)$.
\end{theorem}

\subsection{Applying the ordinary operator}

For clarity we recall the definition of certain ideals.
\begin{align*}
\fn &= \cond(H/F) \prod_{\fl \in T} \fl \\
\fP &= \gcd(p^\infty, \fn) \\
\fP' &= \prod_{\fp \mid p, \ \! \fp \nmid \fP} \fp \\
\fP'' &= \prod_{\fp \mid \fP', \chi(\fp) \neq 1} \fp. 
\end{align*}
We consider a  trichotomy of cases.
\begin{align*}
\text{Case 1a} &: \fP = 1, \fP' \neq \fP''. \\
\text{Case 1b} &: \fP = 1, \fP' = \fP''. \\
\text{Case 2} &: \fP \neq 1. 
\end{align*}

In our applications it will be convenient to apply the ordinary operator at all primes above $p$.  In  addition, in order to ensure that we can work over Hecke algebras that are local rings, we would like to project onto components where the $U_\fp$-operator  for $\fp$ dividing $p$ acts via certain eigenvalues.  In case 1, this latter projection will only be relevant if {\em all} the primes above  $p$ satisfy $\chi(\fp) \neq 1$.  By (\ref{e:xcong}), this is precisely the case that $x$ is a unit in $R$.  We then apply  $U_\fp - \bpsi(\fp)$ for each $\fp  \mid p$  to our family  $\tilde{F}_k(\bpsi)$. Doing so, we obtain the  corollary below.

\begin{corollary} \label{c:fcong} Suppose we are in  case 1.  Then $\fP'$ denotes the product of the primes above $p$. For positive $k \equiv 1 \pmod{(p-1)p^N}$ and $N$ sufficiently large, there exists a cuspidal group ring family $F_k(\bpsi) \in S_k(\fn \fP', R, \bpsi)^{p\text{-ord}}$ such that
\begin{equation} \label{e:case1f}
  F_k(\bpsi) \equiv \left\{ \!\! \begin{array}{lll}
 xW_1(\bpsi, 1) - W_k(\bpsi, 1_p) & \!\! \pmod{x \Theta^\#}  & \  \chi(\fp) = 1 \text{ for some } \fp \mid p \\
 W_1(\bpsi_p, 1)  & \!\! \pmod{\Theta^\#}  & \ \chi(\fp) \neq 1 \text{ for all } \fp \mid p.
 \end{array}\right.
  \end{equation}
\end{corollary}

Here the forms $W_k(\bpsi, 1_p)$ and $ W_1(\bpsi_p, 1) $ are defined like  $W_k(\bpsi, 1)$ and $ W_1(\bpsi, 1)$ but with the characters $1, \bpsi$ replaced by $1_p, \bpsi_p$ in the two cases, respectively.

\begin{remark} \label{r:cong} The congruence (\ref{e:case1f})  should be interpreted as a congruence of Fourier coefficients: 
\begin{equation} \label{e:case1cong}
 c(\fm, F_k(\bpsi)) \equiv x \cdot c(\fm, W_1(\bpsi, 1)) - c(\fm, W_k(\bpsi, 1_p))  \pmod{x \Theta^\#} \end{equation}
for all ideals $\fm$ in the first case, and similarly for the second case.
\end{remark}

\begin{proof} Consider the first case in (\ref{e:case1f}) i.e.~case 1a. For $\fp \mid p$, let 
\[ z = \prod_{\fp \mid p} \frac{\bpsi(\fp)}{\bpsi(\fp) - \N\fp^{k-1}}\equiv 1 \pmod{p^m}. \]
Note that \[ e_p^{\ord}  W_k(\bpsi, 1) = z \cdot W_k(\bpsi, 1_p). \]
Since $z \equiv 1 \pmod{p^m},$ we have $z^{-1} x \equiv x \pmod{x \Theta^\#}.$  The desired result then holds by defining $F_k(\bpsi) = z^{-1} e_p^{\ord}(\tilde{F}_k(\bpsi))$.

In the second case in (\ref{e:case1f}) i.e.~case 1b, we have
 \[ e_p^{\ord}\prod_{\fp \mid p} (U_\fp - \bpsi(\fp))(W_k(\bpsi, 1)) = 0 \]
 whereas  \[ e_p^{\ord}\prod_{\fp \mid p} (U_\fp - \bpsi(\fp))(W_1(\bpsi, 1)) =  W_1(\bpsi_p, 1). \]
Noting that $x \in R^*$ in this case, the result follows by letting
\[ F_k(\bpsi) = x^{-1} e_{p}^{\ord} \prod_{\fp \mid p} (U_\fp - \bpsi(\fp))(\tilde{F}_k(\bpsi)). \]
\end{proof}

In case 2 (i.e.~there exists a prime above $p$ dividing $\fn$), we must apply (in addition to the ordinary operator at  each $\fp \mid p$) the operator $U_\fp - \bpsi(\fp)$ for each $\fp \mid \fP''$. More precisely, we put
\[
F_k(\bpsi) = e_p^{\ord} \prod_{\fp \mid \fP''} (U_{\fp} - \bpsi(\fp))(\tilde{F}_k(\bpsi)).
\]
Note that
\[ e_p^{\ord}\prod_{\fp \mid p} (U_\fp - \bpsi(\fp))(W_1(\bpsi_{\fP}, 1)) =  W_1(\bpsi_{\fP \fP''}, 1). \]
Hence we obtain:

\begin{corollary} \label{c:fcong2} Suppose we are in  case 2.  Then $\fP'$ denotes the product of the primes above $p$ that do not divide $\fn$ and $\fP''$ denote the product of primes $\fp$ dividing $\fP'$ such that $\chi(\fp) \neq 1$.
For positive $k \equiv 1 \pmod{(p-1)p^N}$ and $N$ sufficiently large, there exists a cuspidal group ring family $F_k(\bpsi) \in S_k(\fn \fP', R, \bpsi)^{p\text{-ord}}$ such that
\begin{equation} \label{e:case2f}
  F_k(\bpsi) \equiv W_1(\bpsi_{\fP \fP''}, 1) \pmod{\Theta^\#}.
  \end{equation}
\end{corollary}

\subsection{Homomorphism on the Hecke Algebra}

Let \[ \tilde{\T} \subset \End_{R}(S_k(\fn \fP', R, \bpsi)^{p\text{-ord}}) \] denote the Hecke algebra of the space of  $p$-ordinary group ring valued cusp forms
generated over $R$ by the operators $T_\fl$ for $\fl \nmid \fn \fP'$,  $U_\fp$ for $\fp \mid p$, and the diamond  operators $S(\fm)$.
Note that the operators $S(\fm)$ simply act by $\bpsi(\fm) \in R^*$.  Let $\T \subset \tilde{\T}$ denote the sub-$R$-algebra generated 
by $T_\fl$ for $\fl \nmid \fn \fP'$,  $U_\fp$ for $\fp \mid \fP$, and the $S(\fm)$. In other words, the operators $U_\fp$ for  $\fp \mid \fP'$ are excluded in the definition of $\T$. 

Since our Hecke algebras include only the operators $T_\fl$ for $\fl$ not dividing the level and the operators $U_\fp$ for primes $\fp$ at which our forms are ordinary, the rings $\T$ and $\tilde{\T}$ are reduced.  Let us be more explicit about this fact.
Denote by $M$ the set of $p$-ordinary cuspidal newforms of weight $k$, level dividing $\fn \fP'$, and nebentypus $\psi$ for all  characters $\psi \in \Psi$ (where $R=R_\Psi$).  For each $f \in M$, we denote by $f_p$ the ordinary stabilization of $f$ with respect to all primes $\fp \mid p$.  Suppose that the field $E$ with ring of integers $\cO$ has been chosen large enough so that all the normalized Fourier coefficients $c(\fa, f_p)$ lie in $\cO$. 
Then there are $\cO$-algebra injections with finite cokernels:
\[ \begin{tikzcd}
\T \ar[r] &  \tilde{\T} \ar[r] &  \prod_{M} \cO  
\end{tikzcd} \]
that send $T_\fl \mapsto (c(\fl, f_p))_{f \in M}$, $U_\fp \mapsto (c(\fp, f_p))_{f \in M}$; more succinctly we can write \[ t \mapsto (c(1, (f_p)|_t))_{f \in M}. \]  The injectivity of this map follows from the fact that any $p$-ordinary form $g$ of level $\fn \fP'$ can be written as a linear combination
\[ g = \sum_{f \in M} \sum_{\fb} c_{f, \fb}  (f_p)|_{\fb} \]
as $\fb$ ranges over the divisors of $\fn$ that are relatively prime to $\fp$ and such that $(f_p)|_{\fb}$ has level dividing $\fn \fP'$.  Any element of $\T$ or $\tilde{\T}$ that annihilates every $f_p$ therefore annihilates every $g$.  Finally, the fact that $\T \longrightarrow  \prod_{M} \cO$ has finite cokernel follows from multiplicity 1; for any distinct $f, f' \in M$, there exists $\fl \nmid \fn \fP'$  such that 
$c(\fl, f_p) \neq c(\fl, f_p')$.

Using  
the group ring valued cusp form $F_k(\bpsi)$ constructed in \S\ref{s:cusp}, we now define a certain maximal ideal  $\fm \subset \T$, the maximal Eisenstein ideal.
Note that $F_k(\bpsi)$ is an eigenvector for the action of $\T$ modulo $x \Theta^\#$ or $\Theta^\#$, in cases 1 or 2, respectively.  More precisely, for
$\fl \nmid \fn \fP'$ we have
\begin{equation} \label{e:congs}
 F_k(\bpsi)|_{T_\fl} \equiv \begin{cases}
(\bpsi(\fl) + \epsilon_{\cyc}^{k-1}(\fl)) F_k(\bpsi)    & \pmod{x \Theta^\#}  \text{ in case 1} \\
(\bpsi(\fl) + 1) F_k(\bpsi)  & \pmod{\Theta^\#}  \ \ \text{ in case 2}.
\end{cases} \end{equation}
Here $\epsilon_{\cyc}$ is the $p$-adic cyclotomic character satisfying 
\[ \epsilon_{\cyc}(\fl) = \langle \N\fl \rangle \in \Z_p^*, \qquad \fl \nmid p. \]
 We also have for all $\fp \mid \fP\fP''$:
\begin{equation} \label{e:cong2}
F_k(\bpsi)|_{U_\fp} \equiv F_k(\bpsi) \begin{cases}
(\text{mod }x\Theta^\#) & \text{ in case 1b} \\
(\text{mod }\Theta^\#) &  \text{ in case 2}.
\end{cases}
 \end{equation}
Note that the congruences (\ref{e:congs}) and (\ref{e:cong2})  are to be interpreted as in Remark~\ref{r:cong}. 

\begin{lemma} \label{t:modp} Let $k_E$ denote the residue field of the $p$-adic local ring $\cO = \cO_E$.  There is an $\cO$-algebra homomorphism $\overline{\varphi} \colon {\T} \longrightarrow k_E$ given by
\begin{itemize}
\item $\overline{\varphi}(T_\fl) = 1 + \chi(\fl)$ for $\fl \nmid \fn p$.
\item $\overline{\varphi}(U_\fp) = 1$ for $\fp \mid \fP$.
\item $\overline{\varphi}(S(\fm)) = \chi(\fm)$.
\end{itemize}
\end{lemma}

\begin{proof} The  form $F_k(\bpsi)$ is an eigenform for  the Hecke operators indicated modulo the maximal ideal $\fm_R$ of $R$.  Note that $\fm_R$ is generated by the uniformizer $\pi_E$ of $\cO$ along with the image of the elements $[g] - \chi(g)$ for $g \in G$, and $R/\fm_R \cong k_E$.  The homomorphism $\overline{\varphi}$ is defined by sending each operator to its mod $\fm_R$ eigenvalue.  \end{proof}

We denote by ${\fM} \subset \T$ the kernel of $\overline{\varphi}$.  We denote by $\T_\fM$ and $\tilde{\T}_{\fM} = \tilde{\T} \otimes_\T \T_\fM$ the $\fM$-adic completions of $\T$ and $\tilde{\T}$, respectively. We would also like to identify the $\fM$-adic completion of $\prod_{f \in M} \cO$.   Let $\overline{M} \subset M$ denote the set of $f \in M$ such that $c(1, (f_p)|_t) \equiv \overline{\varphi}(t) \pmod{\pi_E}$ for all $t \in \T$.  We then have 
\[ \Big(\prod_{f \in M} \cO\Big)_{\fM} = \prod_{f \in \overline{M}} \cO. \]
The Artin-Rees Lemma (\cite{am}*{Proposition 10.12})  yields injections with finite cokernel
\[ \begin{tikzcd}
 \T_\fM \ar[r] &  \tilde{\T}_{\fM} \ar[r] &  \prod_{f \in \overline{M}} \cO. 
 \end{tikzcd} \]

 In the statement of the following theorem, $x$ is as in Lemma~\ref{l:xdef} in case 1a, and $x=1$ in cases 1b and 2.

\begin{theorem} \label{t:yxt}
In both cases 1 and 2, there exists a non-zerodivisor $x \in R$, an $R/x\Theta^\#$-algebra $W$, and a surjective $R$-algebra homomorphism $\varphi\colon \tilde{\T}_{\fM} \longrightarrow W$ satisfying the following properties:
\begin{itemize}
\item The structure map $R/x\Theta^\# \longrightarrow W$ is an injection.
\item The restriction of $\varphi$ to $\T_\fM$ takes values in $R/x\Theta^\# \subset W$.  More precisely, \begin{align*}
\varphi(S(\fm)) &= \bpsi(\fm) \text{ for } \fm \in G_\fn^+, \\
\varphi(U_\fp) &= 1 \text{ for } \fp \mid \fP, \text{ and } \\
\varphi(T_\fl) &= \epsilon_{\cyc}^{k-1}(\fl) + \bpsi(\fl) \text{ for } \fl \nmid \fn p.
\end{align*} 
\item Let \[ {U} = \prod_{\fp \mid \fP'} (U_{\fp} - \bpsi(\fp)) \in \tilde{\T}_{\fM}. \]
  If $y \in R$ and $\varphi(U)y = 0$ in $W$, then $y \in (\Theta^\#)$. 
\end{itemize}
\end{theorem}

\begin{proof} We consider case 1a, with $x$ as in Lemma~\ref{l:xdef}, as the other cases are similar (and in fact easier).
Let $\cC = \prod_{\fa \subset \cO_F} R/x\Theta^\#$ be the product of copies of 
 $R/x\Theta^\#$, indexed by the set of nonzero ideals $\fa \subset \cO_F$.
There is an $R$-module homomorphism $c \colon S_k(\fn,R, \bpsi) \longrightarrow \cC$ that associates to each cusp form its collection of Fourier coefficients $c(\fa, f)$.
There is an action of the Hecke operators on $\cC$ given by the formula (\ref{e:heckec}), and the map $c$ is Hecke equivariant.

Let $\cF$ denote the image of the $\tilde{\T}$-span of the cusp form $F_k(\bpsi)$  given in  Theorem~\ref{t:fk1}  under the map $c$.  This is a finite-type $R/x\Theta^\#$-module.  We define $W$ to be the image of the canonical $R$-algebra homomorphism $\tilde{\T} \longrightarrow \End_{R/x\Theta^\#}(\cF)$.  This construction yields a canonical surjective $R$-algebra map $\varphi\colon \tilde{\T} \longrightarrow W$ that sends a Hecke operator to its action on the Hecke span of $F_k(\bpsi)$ under the map $c$.
In view of (\ref{e:congs}), we obtain 
 \[ \varphi(T_\fl) = \epsilon_{\cyc}^{k-1}(\fl) + \bpsi(\fl) \] for $\fl \nmid \fn p$. From (\ref{e:cong2}) we obtain
$ \varphi(U_{\fp}) = 1$ for all $\fp \mid \fP$. Further, both $W_1(\bpsi, 1)$ and $W_k(\bpsi, 1_p)$ are eigenvectors for $S(\fm)$ with eigenvalue $\bpsi(\fm)$. Hence $\varphi(S(\fm)) = \bpsi(\fm)$.   
These observations imply that the algebra $W,$ viewed as a $\T$-module through the homomorphism $\varphi$, is $\fM$-adically complete and we obtain an induced map 
 \[ \begin{tikzcd}
  \varphi \colon \tilde{\T}_{\fM} \ar[r] &  W. 
  \end{tikzcd} \]

Let us verify the necessary properties.  If $\alpha \in R$ has vanishing image in $W,$ then $\alpha F_k(\bpsi) \equiv  0 \pmod{x\Theta^\#}$.
Analyzing the congruence (\ref{e:case1cong}) for $\fm = 1$ and $\fm = \fp$, for any $\fp \mid p$, yields:
\begin{align*}
(x - 1) \alpha & \equiv 0 \pmod{x\Theta^\#} \\
((1+ \bpsi(\fp))x - \bpsi(\fp))  \alpha & \equiv 0 \pmod{x\Theta^\#}.
\end{align*}
Multiplying the first congruence by $-(1 + \bpsi(\fp))$ and adding the second yields $\alpha \equiv 0 \pmod{x\Theta^\#}$.  This establishes the injectivity of $R/x\Theta^\# \longrightarrow W$.

For the last item we note that (\ref{e:case1f}) yields
\[ F_k(\bpsi)|_{U} \equiv x W_1(1, \psi_p) \pmod{x \Theta^\#}. \]
Therefore if $y F_k(\bpsi)|_{U} \equiv 0 \pmod{x \Theta^\#}$ for $y \in R$, then by considering the Fourier coefficient of $\fm = 1$ we see that $xy \in (x \Theta^\#)$ and hence $y \in (\Theta^\#)$ since $x$ is a non-zerodivisor in $R$.

The result in cases 1b and 2 can be proved analogously with $x =1$.  Note that in these cases,  $\epsilon^{k-1}_{\cyc}(\fl) = 1$ in $W$ since $\Theta^\#$ divides $\epsilon^{k-1}_{\cyc}(\fl) - 1$  for $k-1$ divisible by $(p-1)p^{m}$ and $m$ sufficiently large.  For this, it is essential that we are working on the trivial zero free quotient $R$, so $\Theta^\#$ is a non-zerodivisor.
\end{proof}

\section{Galois representation and cohomology class} \label{s:galrep}

\subsection{Galois representation associated to each eigenform} \label{s:galrepf}

Let $f \in \overline{M}$, as defined before Theorem~\ref{t:yxt},  and let $\psi$ denote the nebentypus of $f$. The work of Hida and Wiles \cite{wilesrep}*{Theorems 1 and 2} establishes a continuous Galois representation \[ 
\begin{tikzcd}
 \rho_f \colon G_F \ar[r] &  \GL_2(E)  
 \end{tikzcd} \]
satisfying the following properties:
\begin{itemize}
\item[(1)] $\rho_f$ is unramified outside $\fn p$. 
\item[(2)] For all primes $\fl \nmid \fn p$, the characteristic polynomial of $\rho(\Frob_{\fl})$ is given 
\[
\chr(\rho_f(\Frob_{\fl}))(x) = x^2 - c(\fl, f) x + \psi(\fl) \epsilon_{\cyc}^{k-1}(\fl),
\] 
where $\epsilon_{\cyc}$ is the cyclotomic character.
\item[(3)] For all $\fp \mid p$, we have
\begin{equation} \label{e:locbasis}
\rho_f|_{G_{\fp}} \sim \mat{ \psi \eta_{\fp}^{-1} \epsilon_{\cyc}^{k-1} }{ * }{ 0 }{ \eta_{\fp}},
\end{equation}
 where $\eta_{\fp} \colon G_{\fp} \longrightarrow E^*$ is an unramified character given by $\eta_{\fp}(\rec(\varpi^{-1})) = c(\fp, f_p)$. Here $\varpi \in F_{\fp}$ is a uniformizer and $\rec\colon F_\fp^* \longrightarrow G_\fp^{\ab}$ is the local Artin reciprocity map. We adopt Serre's conventions \cite{serre} for the reciprocity map. Therefore $\rec(\varpi^{-1})$ is a lifting to $G_{\fp}^{\ab}$ of the Frobenius element on the maximal unramified extension of $F_{\fp}$ if $\varpi \in F_{\fp}^*$ is a uniformizer. We denote by $V_{\fp, f}$ the eigenspace of $\rho_f|_{G_\fp}$, i.e. the span of the vector $\binom{1}{0}$ in the basis for which (\ref{e:locbasis}) holds. 
\end{itemize}

By \v{C}ebotarev and property (2) of $\rho_f$, we see that $\chr(\rho_f(\sigma)) \in \cO[x]$ for all $\sigma \in G_F$, and furthermore (since $f \in \overline{M}$) that
 \[ \chr(\rho_f(\sigma)) \equiv (x - 1)(x - \chi(\sigma)) \pmod{\pi_E}. \]
 For this, recall that $\psi \equiv \chi \pmod{\pi_E}$. 
 
  Suppose that $\tau \in G_F$ such that $\chi(\tau) \neq 1$.  For example, we may choose $\tau$ to be an element whose restriction to $H$ is the complex conjugation, so that $\chi(\tau) = -1$.  Since $\chi$ is a prime-to-$p$ order character, $\chi(\tau) \neq 1$ implies $\chi(\tau) \not\equiv 1 \pmod{\pi_E}$, so Hensel's Lemma implies that $\rho_f(\tau)$ has two distinct eigenvalues \[ \lambda_{1, f} \equiv 1 \pmod{\pi_E}, \qquad \lambda_{2,f} \equiv \chi(\tau) \pmod{\pi_E}. \]

Ribet's method involves comparing the ``global" basis for $\rho_f$ given by the eigenvectors of $\rho_f(\tau)$ to the ``local" basis indicated in (\ref{e:locbasis}). This argument, which Mazur \cite{m} has called ``Ribet's Wrench," does not succeed in our application if the global basis and local basis
are the same. We must show, therefore, that $\tau$ can be chosen so that neither of the  eigenspaces of $\rho_f(\tau)$ is equal to the eigenspace $V_{\fp, f}$ appearing in property (3) of $\rho_f$, for any $\fp \mid p.$ Furthermore, we must do this simultaneously for all the finitely many $f \in \overline{M}$.

For this, we distinguish two cases.  We say that $f$ is a CM form if $\rho_f = \Ind_{G_L}^{G_F} \alpha$ where $L$ is a quadratic CM extension of $F$ and $\alpha$ is a $p$-adic Hecke character of $L$.   The following lemma of Ribet, proved using a group theoretic study of $\GL_2$, is essential for our analysis:

\begin{lemma} \label{l:noncm} Let $f$ be a cuspidal eigenform of weight $k > 1$.  Suppose that $f$ is not a CM form.  Then the restriction of $\rho_f$ to any finite index subgroup of $G_F$ is irreducible. 
\end{lemma}

\begin{proof} Suppose that the restriction of $\rho_f$ to a finite index subgroup of $G_F$ is reducible. Then  
\cite{ribetrep}*{Theorem~2.3} implies that $\rho_f$ is induced from an index 2 subgroup of $G_F$. Therefore the image of $\rho_f$ is projectively dihedral. Hence the fixed field of this index two subgroup is a CM field by \cite{bgv}*{Page~2, Remark~(ii)}.
\end{proof}

\begin{lemma} \label{l:cm} Let $f \in \overline{M}$ be a CM form associated to a quadratic CM extension $L/F$, and let $\fp \mid p$.  The subspace $V_{\fp, f}$ is not stable under $\rho_f(\tau)$ for any $\tau$ that restricts to the complex conjugation of $L$.
\end{lemma}

\begin{proof} Since $f$ is $p$-ordinary, the prime $\fp$ splits in the quadratic extension $L/F$ (see \cite{ho}*{Proposition A.3}).  It follows that $G_{\fp} \subset G_L$.  Yet $\rho_f = \Ind_{G_L}^{G_F} \alpha$ has two subspaces that are stable under all of $G_L$, hence $V_{\fp, f}$ must be one of these subspaces (note that the characters of the semisimplification of $\rho|_{G_{\fp}}$ are distinct since one is ramified and the other is not, so $\rho|_{G_{\fp}}$ cannot be a scalar representation).  If this subspace were invariant under any $\tau$ restricting to the complex conjugation of $L$, it would then be invariant under all of $G_F$, contradicting the irreducibility of $\rho_f$. The result follows.
\end{proof}

The following is a modification of Lemma 4.3 in \cite{dkv}.

\begin{prop} \label{p:tau}
There exists $\tau \in G_F$ such that $\tau$ restricts to the complex conjugation of $G$, and such that for all $f \in \overline{M}$ and $\fp \mid p$, the subspace $V_{\fp, f}$ is not stable under $\rho_f(\tau)$.
\end{prop}

\begin{proof}   Let $H_0$ denote the compositum of $H$ with the CM fields $L$ associated to each CM form $f \in \overline{M}$.  The field $H_0$ is a finite CM abelian extension of $F$.
 Let $\tau_0 \in \Gal(H_0/F)$ be the complex conjugation.  Lemma~\ref{l:cm} implies that any $\tau$ restricting to $\tau_0$ on $H_0$ satisifies the desired property for the CM forms $f \in \overline{M}$ and all $\fp \mid p$.
 
 Now label the $V_{\fp, f}$ for $f \in \overline{M}$ that are not CM forms and $\fp \mid p$ by $V_1, \dotsc, V_n$.  We will define $\tau$ inductively starting from the $(\tau_0, H_0)$ defined above as the base case.   Let $1 \le i \le n$.  Denote by $G_i \subset G_F$ the stabilizer of $V_i$ under $\rho_f$ (where $V_i = V_{\fp, f}$ for some $\fp$).  By Lemma~\ref{l:noncm}, $G_i$ 
 has infinite index in $G_F$.  We can therefore select an element $\alpha_i \in \overline{F}$ that is fixed by $G_i$ and that does not lie in $H_{i-1}$.  Let $H_i$ be the Galois closure of $H_{i-1}(\alpha_i)$ over $F$ and let $\tau_i \in \Gal(H_i/F)$ be any element that restricts to $\tau_{i-1}$ on $H_{i-1}$ and such that $\tau_i(\alpha_i) \neq \alpha_i$.  Note that any $\tau \in G_F$ restricting to $\tau_i$ on $H_i$ moves $\alpha_i$ and hence does not lie in $G_i$, i.e.\ does not stabilize $V_i$ under $\rho_f$.  It therefore suffices to let $\tau$ be any element that restricts to $\tau_n$ on $H_n$, and the proposition follows.
\end{proof}

We once and for all fix a $\tau$ as in Proposition~\ref{p:tau} and choose the basis for each $\rho_{f}$ so that 
\[ \rho_f(\tau) = \mat{\lambda_{1,f}}{0}{0}{\lambda_{2,f}}, \]
where $\lambda_{1,f} \equiv 1 \pmod{\pi_E}$ and $\lambda_{2,f} \equiv \chi(\tau) \equiv -1 \pmod{\pi_E}$ as above.  We write
\[ \rho_f(\sigma) = \mat{a_f(\sigma)}{b_f(\sigma)}{c_f(\sigma)}{d_f(\sigma)}. \]
For each $\fp \mid p$, we let 
\[ M_{f, \fp} = \mat{A_{f, \fp}}{B_{f, \fp}}{C_{f, \fp}}{D_{f, \fp}} \in \GL_2(E) \]
denote a change of basis matrix relating this basis to the one giving the local form (\ref{e:locbasis}), i.e. such that
\[
\mat{ a_f(\sigma)}{ b_f(\sigma)}{c_f(\sigma)}{d_f(\sigma) }M_{f, \fp} =
M_{f, \fp}
\mat{ \psi\eta_{\fp}^{-1} \epsilon_{\cyc}^{k-1}(\sigma) }{ *}{0 }{ \eta_{\fp}(\sigma)}
\]
for all $\sigma \in G_{\fp}$.  The key point of Proposition~\ref{p:tau} is the following:
\begin{equation} \label{e:keypoint}
 \text{ for every } f \in \overline{M} \text{ and every } \fp \mid p, \text{ we have } A_{f, \fp} \neq 0
\text{ and } C_{f, \fp} \neq 0. 
\end{equation}

\subsection{Galois representation associated to $\T_\fM$} \label{s:galrepphi}

Let \begin{equation} \label{e:kdef}
K = \Frac(\T_\fM) = \Frac\Big(\prod_{f \in \overline{M}} \cO\Big) = \prod_{f \in \overline{M}} E. 
\end{equation}
Consider the Galois representation \[
\begin{tikzcd}
\rho = \displaystyle\prod_{f \in \overline{M}} \rho_f \colon G_F \ar[r] & \GL_2(K).
\end{tikzcd}
\]
Note that $\rho$ is  continuous with respect to the $p$-adic topology on $K$ (since each factor $\rho_f$ is continuous) and hence continuous with respect to the $\fM$-adic toplogy on $K$, as every ideal $\fM^n$ is finitely generated over $\cO$.
The representation $\rho$ satisfies:
\begin{itemize}
\item[(1)] $\rho$ is unramified outside $\fn p $. 
\item[(2)] For all primes $\fl \nmid \fn p$, the characteristic polynomial of $\rho(\Frob_{\fl})$ is given 
\begin{equation} \label{e:char}
\chr(\rho(\Frob_{\fl}))(x) = x^2 - T_{\fl} x + \bpsi(\fl) \epsilon_{\cyc}^{k-1}(\fl),
\end{equation}
where $\epsilon_{\cyc}$ is the cyclotomic character.
\item[(3)] For all $\fp \mid p$, we have
\begin{equation} \label{e:localrho}
\rho|_{G_{\fp}} \sim \mat{ \bpsi \eta_{\fp}^{-1} \epsilon_{\cyc}^{k-1} }{ * }{ 0 }{ \eta_{\fp}},
\end{equation}
 where $\eta_{\fp} \colon G_{\fp} \longrightarrow \tilde{\T}^*$ is the unramified character given by $\eta_{\fp}(\rec(\varpi^{-1})) = {U}_{\fp}$. 
\end{itemize}

By \v{C}ebotarev and (\ref{e:char}), it follows that $\chr\rho(\sigma)(x) \in \T_\fM[x]$ for all $\sigma \in G_F$, and furthermore, that
\[ \chr\rho(\sigma)(x) \equiv (x - 1)(x - \chi(\sigma)) \pmod{\fM}. \]
Recall the $\tau \in G_F$ fixed in the previous section, for which $\chi(\tau) = -1$.  The polynomial 
$\chr(\rho(\tau))$ has two distinct roots modulo $\fM$ and hence by Hensel's lemma has two distinct roots  $\lambda_1, \lambda_2 \in \T_\fM^*$, with $\lambda_1 \equiv 1 \pmod{\fM}$ and $\lambda_2 \equiv -1 \pmod{\fM}$.

As in \S\ref{s:galrepf}, we choose the basis for $\rho$ in which $\rho(\tau) = \mat{\lambda_1}{0}{0}{\lambda_2}$ and for a general $\sigma \in G_F$ we write
\[ \rho(\sigma) = \mat{a(\sigma)}{b(\sigma)}{c(\sigma)}{d(\sigma)}. \]
For each $\fp \mid p$ there is a change of basis matrix $M_{\fp} = \mat{A_\fp}{B_{\fp}}{C_{\fp}}{D_{\fp}} \in \GL_2(K)$ such that 
\begin{equation} \label{e:plocal}
\mat{ a(\sigma) }{ b(\sigma)  }{c(\sigma) }{ d(\sigma) } M_{\fp} =
M_\fp
\mat{ \bpsi\eta_{\fp}^{-1} \epsilon_{\cyc}^{k-1}(\sigma) }{ * }{ 0 }{ \eta_{\fp}(\sigma)}
\end{equation}
for all $\sigma \in G_{\fp}$.  Here $A_\fp = (A_{f, \fp})_{f \in \overline{M}}$ under the identification (\ref{e:kdef}), and similarly for $B_{\fp}, C_{\fp}, D_{\fp}$.  Therefore, (\ref{e:keypoint}) implies that the elements $A_\fp$ and $C_\fp$ are invertible in $K$.
Comparing the top left corner elements in (\ref{e:plocal}) gives 
\begin{equation} \label{e:beq}
b(\sigma) = \frac{A_{\fp}}{C_{\fp}}(\bpsi\eta_{\fp}^{-1} \epsilon_{\cyc}^{k-1}(\sigma)-a(\sigma))
\end{equation}
for all $\sigma \in G_{\fp}$.

\subsection{Cohomology Class and Ramification away from $p$}

In this section we construct a Galois cohomology class associated to the homomorphism  $\varphi$ constructed in Theorem~\ref{t:yxt}.
Let $\tilde{I} \subset \tilde{\T}_\fM$ denote the kernel of $\varphi$ and let $I = \tilde{I} \cap \T_{\fM}$ denote the kernel of $\varphi|_{\T_\fM}$.

We begin by employing some standard techniques in the theory of pseudo-representations.  As noted above, we have $\Tr \rho(\sigma) \in \T_{\fM}$ for all $\sigma \in G_F$.  Furthermore, in view of  (\ref{e:char}) and the property $\varphi(T_\fl) = \epsilon_{\cyc}^{k-1}(\fl) + \bpsi(\fl)$, we find from \v{C}ebotarev that
\begin{equation} \label{e:trmodi}
 \Tr \rho(\sigma) = a(\sigma) + d(\sigma) \equiv  \epsilon_{\cyc}^{k-1}(\sigma) + \bpsi(\sigma)  \pmod{I} \qquad \text{ for all } \sigma \in G_F. \end{equation}
 In particular,  for the fixed element $\tau$  introduced in \S\ref{s:galrepf}--\ref{s:galrepphi}, we obtain \[ \lambda_1 + \lambda_2 \equiv \epsilon_{\cyc}^{k-1}(\tau) + \bpsi(\tau)  \pmod{I} \] and hence $\lambda_1, \lambda_2$ are roots of the polynomial
 \begin{equation} \label{e:charmodi}
  (x - \epsilon_{\cyc}^{k-1}(\tau))(x- \bpsi(\tau)) \pmod{I}. \end{equation}
  Since $\lambda_1 \equiv 1 \equiv \epsilon_{\cyc}^{k-1}(\tau) \pmod{\fM}$ and $\lambda_2 \equiv \chi(\tau) \equiv\bpsi(\tau) \pmod{\fM}$, with $\lambda_1 \not\equiv \lambda_2 \pmod{\fM}$, it follows from (\ref{e:charmodi}) that
  \[ \lambda_1 \equiv  \epsilon_{\cyc}^{k-1}(\tau) \pmod{I}, \qquad  \lambda_2 \equiv \bpsi(\tau) \pmod{I}.\]
  
The congruence (\ref{e:trmodi}) with $\sigma$ replaced by $\sigma \tau$ yields
\begin{align}
 a(\sigma)\lambda_1 + d(\sigma) \lambda_2 &\equiv  \epsilon_{\cyc}^{k-1}(\sigma \tau) + \bpsi(\sigma \tau) \nonumber \\
 & \equiv \epsilon_{\cyc}^{k-1}(\sigma) \lambda_1 + \bpsi(\sigma) \lambda_2. \label{e:trtau}
 \end{align}
  
  The two congruences (\ref{e:trmodi}) and (\ref{e:trtau}) may be solved, again using $\lambda_1 \not\equiv \lambda_2 \pmod{\fM}$, to yield
  \begin{equation} \label{e:adcong}
   a(\sigma) \equiv  \epsilon_{\cyc}^{k-1}(\sigma) \pmod{I}, \qquad  d(\sigma) \equiv \bpsi(\sigma) \pmod{I} \end{equation}
  for all $\sigma \in G_F$ (in particular $a(\sigma), d(\sigma) \in \T_{\fM}$).

Let $B_0$ be the $\T_\fM$-submodule of $K$ generated by  $\{ b(\sigma) \colon \sigma \in G_F \}$, and let $B$ be any $\T_\fM$-submodule of $K$ containing $B_0$.
Let $B' \subset B$ be any $\T_\fM$-submodule containing $IB_0$.  Put $\overline{B} = B/B'$. 
Since $\rho$ is a representation, we have
\[ b(\sigma \sigma') = a(\sigma) b(\sigma') + b(\sigma)d(\sigma'), \qquad \sigma, \sigma' \in G_F. \]
The congruences (\ref{e:adcong}) imply that
\[ \overline{b}(\sigma \sigma') \equiv \epsilon^{k-1}_{\cyc}(\sigma) \overline{b}(\sigma') + \bpsi(\sigma') \overline{b}(\sigma) \quad \text{in } \overline{B}, \]
where $\overline{b}(\sigma)$ denotes the image of $b(\sigma)$ in $\overline{B}$.
It follows that the function  \begin{equation} \label{e:kappadef}
\sigma \mapsto \overline{b}(\sigma) \bpsi(\sigma)^{-1}\end{equation} is a 1-cocycle yielding  a class 
$\kappa \in H^1(G_F, \overline{B}(\bpsi^{-1}\epsilon_{\cyc}^{k-1})).$
If furthermore $p^m B \subset B'$ and $k \equiv 1 \pmod{(p-1)p^N}$ with $N \ge m$ (so that $\epsilon_{\cyc}^{k-1} \equiv 1 \pmod{p^m}$), then multiplication by $\epsilon_{\cyc}^{k-1}$ acts trivially on $\overline{B}$, and $\kappa$ may be viewed as a class  
\begin{equation} \label{e:kappaspace}
 \kappa \in H^1(G_F, \overline{B}(\bpsi^{-1})). \end{equation}

\begin{prop} \label{p:easyur}
Let $\overline{B}$ be as above.  The class $\kappa \in H^1(G_F, \overline{B}(\bpsi^{-1}))$ defined by (\ref{e:kappadef}) is
unramified away from $\fn p$, i.e. its restriction to \[ H^1(I_v, \overline{B}(\bpsi^{-1})) \] vanishes for places $v \nmid \fn p$ of $F$.
Furthermore, for $v \mid \fn, v \nmid p$, the class $\kappa$ is at most tamely ramified, i.e.\ its restriction to the wild inertia subgroup $I_{v, 1} \subset I_v$ vanishes.
\end{prop}

\begin{proof}
The first property is trivial, since $\rho$ is unramified outside $\fn p$, so \[ b(\sigma) = 0 \text{  for } \sigma \in I_v, v \nmid \fn p. \] Now for any $v$, the wild inertia group 
$I_{v, 1}$ is a pro-$v$ group while the module $\overline{B}$ is a pro-$p$-group.  It follows from continuity of  cocycles that for $v \nmid p$ the entire space
\[ H^1(I_{v,1}, \overline{B}(\bpsi^{-1})) \] is trivial.
\end{proof}

\subsection{Surjection from $\nabla$}

We recall the sets $\Sigma, \Sigma'$ from \S\ref{s:sad}:
\begin{align*}
\Sigma & = \{ v  \in S_{\ram}  \colon v \mid p \} \cup S_\infty, \\
\Sigma' &=  \{  v \in S_{\ram} : v\nmid p \} \cup T.
\end{align*}
Recall from \S\ref{s:nop} that we may assume  $T$ contains no primes above $p$.

As above let $B_0$ denote the $\T_\fM$-submodule of $K$ generated by
the $b(\sigma)$ for all $\sigma \in G_F$.
Define $B$ to be the $\T_\fM$-submodule of $K$ generated by
$B_0$  and by the elements $A_\fp/C_{\fp}$ appearing in (\ref{e:plocal})--(\ref{e:beq}) for all finite $\fp \in \Sigma$:
\[ B = (B_0, A_\fp/ C_\fp \colon \fp \in \Sigma - S_\infty) . \]
In case 1, when $\Sigma = S_\infty$, we have $B_0 =B$.

Let $B'$ be the $\T_\fM$-submodule of $B$ generated by $b(\sigma)$ for $\sigma \in I_\fp$, as $\fp$ ranges over all primes dividing $\fP'$; these are the primes above $p$ that do not ramify in $H/F$.
Define\[   \overline{B}_p = 
B/(IB, B', p^mB).
 \]
Let $\overline{B}_0 \subset \overline{B}_p$ be the image of $B_0$ in $\overline{B}_p$. 

\begin{prop} \label{p:kappaunram}The cohomology class $\kappa \in H^1(G_F, \overline{B}_p(\bpsi^{-1}))$ defined in (\ref{e:kappaspace})  is unramified away from $\Sigma'$, tamely ramified at $\Sigma'$, and locally trivial at $\Sigma$.
\end{prop}

\begin{proof} We  saw in Proposition~\ref{p:easyur} that $\kappa$ is unramified away from $\Sigma \cup \Sigma' \cup \{\fp \mid p\}$, and that it is tamely ramified at $\Sigma'$.  For the primes $\fp \mid p$ not in $\Sigma$, the class $\kappa$ is unramified because in the definition of $\overline{B}_p$ we have taken the quotient by the image of inertia under $b$.  

The local triviality of $\kappa$ at infinite places is automatic, since $p$ is odd.  It remains to show that $\kappa$ is locally trivial at all finite $\fp \in \Sigma$.  For this, we use equation (\ref{e:beq}).  By definition, $A_\fp / C_\fp \in B$ and hence $(A_\fp/C_\fp) I \equiv 0$ in $\overline{B}$.  Furthermore $\eta_\fp \equiv 1 \pmod{I}$ since $\varphi(U_\fp) = 1$ for $\fp \in \Sigma$ (recall that $\varphi$ is defined in Theorem \ref{t:yxt}).  This part of the argument is relevant only when $\Sigma \neq S_\infty$, i.e.~case 2. As noted at the end of the proof of Theorem~\ref{t:yxt}, in this case we have   $\epsilon_{\cyc}^{k-1} \equiv 1 \pmod{I}$.  Therefore $a(\sigma) \equiv 1 \pmod{I}$ as well.  

Combining these observations, we see that
\begin{equation} \label{e:bpac}
 b(\sigma) \bpsi^{-1}(\sigma) \equiv (1 - \bpsi^{-1}(\sigma))\frac{A_\fp}{C_\fp} \quad \text{ in } \overline{B}_p  \text{ for } \sigma \in G_\fp.
 \end{equation}
Therefore $\kappa|_{G_\fp}$ is a coboundary as desired.
\end{proof}

As in \S\ref{s:ext}, let $L/H$ be the finite abelian extension of $H$ associated to $\Cl_{\Sigma}^{\Sigma'}(H)^-$ by class field theory, i.e.~such that the Artin reciprocity map yields an isomorphism 
\[ \begin{tikzcd}
 \rec_{L/H}\colon  \Cl_{\Sigma}^{\Sigma'}(H)^- \ar[r] &  \Gal(L/H).
 \end{tikzcd} \]
Explicitly, $L$ is the maximal abelian extension of $H$ of odd degree unramified outside of $\Sigma'_H$, tamely ramified at $\Sigma'_H$, split completely at $\Sigma_H$, and such that the conjugation action of complex conjugation is equal to inversion on $\Gal(L/H)$.

It is natural to  consider $\overline{B}_0(\bpsi^{-1})$ as an $R^\#$-module. This is the space 
$\overline{B}_0$ in which the action of $R^\#$ is given by
 \[ (r, b) \mapsto r^\# \cdot b, \]
with the action on the right the usual action of $R$ on $\overline{B}_0$.   With this notation, the $G$-module action on $\overline{B}_0(\bpsi^{-1})$ is consistent with the $R^\#$-module action via the projection $\cO[G] \longrightarrow R^\#$.

\begin{corollary} \label{c:cb} There is a canonical  $R^\#$-module surjection 
\[ \begin{tikzcd} \alpha \colon \Cl_{\Sigma}^{\Sigma'}(H)_{R^\#} \ar[r] &  \overline{B}_0(\bpsi^{-1}) 
\end{tikzcd}
\] induced by $\alpha(\fa) = \overline{b}(\sigma)$ for $\fa \in  \Cl_{\Sigma}^{\Sigma'}(H)$, where $\sigma$ denotes any lift of $\rec_{L/H}(\fa)$ to $G_H \subset G_F$.  
\end{corollary}

\begin{proof}  The character $\bpsi$ acts through $G = \Gal(H/F)$, so its restriction to $G_H$ is trivial.  We consider the restriction
\[ \kappa_H = \res_{G_H}^{G_F} \kappa \in H^1(G_H, \overline{B}_0)^{G = \bpsi^{-1}} = \Hom_{\cont}(G_H^{\ab}, \overline{B}_0)^{G = \bpsi^{-1}}. \]
The superscript indicates that we consider the space of continuous homomorphisms that are $G$-equivariant, where $G$ acts on $G_H^{\ab}$ via conjugation by a lift to $G_H$ and on  $\overline{B}_p$ via the character $\bpsi^{-1}$.  

By Proposition~\ref{p:kappaunram}, the fixed field of the kernel of the homomorphism $\kappa_H$, which we denote $L'$, is unramified outside of $\Sigma'_H$, tamely ramified at $\Sigma'_H$, and split completely at $\Sigma_H$.  Therefore $L' \subset L$ and we get maps
\begin{equation} \label{e:clb}
 \begin{tikzcd} 
\Cl_{\Sigma}^{\Sigma'}(H)^{-} \ar[r,"\rec_{L/H}"] & \Gal(L/H) \ar[r, two heads] &  \Gal(L'/H) \ar[r, "\kappa_H"] &  \overline{B}_p(\bpsi^{-1}). 
\end{tikzcd}
\end{equation}
Furthermore these maps are $G$-equivariant (where on the middle two terms $G$ acts by conjugation, and as indicated $G$ acts on $\overline{B}_p$ via $\bpsi^{-1}$).  
By construction, the composition of maps in (\ref{e:clb}) is given by $\fa \mapsto \overline{b}(\sigma)$, with notation as in the statement of the corollary.  It follows that the image of the composition in (\ref{e:clb}) is contained in $\overline{B}_0(\bpsi^{-1})$.
It remains to prove that if we extend scalars to $R^\#$, then the induced map 
\[ \begin{tikzcd}
 \alpha \colon \Cl_{\Sigma}^{\Sigma'}(H)_{R^\#} \ar[r] &  \overline{B}_0(\bpsi^{-1}) 
 \end{tikzcd} \]
is surjective.  Denote by $B_\alpha$ the image of $\alpha$, and write $\overline{B}^\alpha = \overline{B}_0/B_\alpha$.  By construction, $B_\alpha$ contains $\overline{b}(\sigma)$ for all $\sigma \in G_H$, so the image of $\kappa_H$ in $H^1(G_H, \overline{B}^\alpha)$ is trivial. 
 By Lemma~\ref{l:splitting}, this implies that the image of $\kappa$ in
  $H^1(G_F, \overline{B}^\alpha(\bpsi^{-1}))$,  denoted $\kappa^\alpha$, is trivial.  Yet if we write $\kappa^\alpha$  as a coboundary: \[ \kappa^\alpha(\sigma) = (1- \bpsi^{-1}(\sigma))t \] for $t \in \overline{B}^\alpha$, then evaluating at $\sigma = \tau$ shows that $t = 0$ (since $\kappa(\tau) = 0$ and $\bpsi(\tau) = -1)$.
Therefore $\kappa^\alpha$ is zero as a cocycle, not just as a cohomology class.  But the values of the cocycle $\kappa$ generate the module $B_0$, and hence $\kappa^\alpha$ generates the module 
$\overline{B}^\alpha$.  It follows that $ \overline{B}^\alpha =0$, i.e. that $\alpha$ is surjective.
\end{proof}

Next we consider the quotient $\T_\fm$-module $\overline{B}_1 = \overline{B}_p/\overline{B}_0$.  This module is generated by the $A_\fp/C_\fp$ for finite $\fp \in \Sigma$.  In fact, since the definition of $\varphi$ implies that every element of $\T_\fM$ is congruent modulo $I$ to an element of $R$, it follows that $\overline{B}_1$ is generated over $R$ by the $A_\fp/C_\fp$.

\begin{prop}  \label{p:xb}
There is a canonical $R^\#$-module surjection \[ \begin{tikzcd}
 \gamma\colon (X_{H, \Sigma})_{R^\#} \ar[r] &  \overline{B}_1(\bpsi^{-1}). 
 \end{tikzcd}
  \]
\end{prop}

\begin{proof}  By construction there is a canonical $R$-module surjection
\[\begin{tikzcd}
 \gamma \colon \bigoplus_{\fp \in \Sigma} R^\# \ar[r] &  \overline{B}_1(\bpsi^{-1}) 
 \end{tikzcd}
 \]
that sends a basis vector associated to $\fp \in \Sigma$ to $-A_\fp/C_\fp$.

Since $R^\#$ is a quotient of a component of $\cO[G]$ corresponding to an odd (and in particular nontrivial) character $\chi^{-1}$, we have 
\[ (X_{H, \Sigma})_{R^\#} \cong (Y_{H, \Sigma})_{R^\#} = \bigoplus_{\fp \in \Sigma} (\Z[G/G_\fp] \otimes_{\Z[G]} R^\#) \cong  \bigoplus_{\fp \in \Sigma} R^\#/\Delta G_\fp. \]
Here $G_\fp \subset G$ is the decomposition group at $\fp$ in $G$ and $\Delta G_\fp \subset R^\#$ is the ideal generated by $\bpsi(g) - 1$ for $g \in G_\fp$.  To show that the map $\gamma$ factors through 
$(X_{H, \Sigma})_{R^\#}$, we must show that
\[ (\bpsi(g) - 1)\frac{A_\fp}{C_\fp} \equiv 0 \quad \text{in } \overline{B}_1(\bpsi^{-1}). \]
This follows directly from (\ref{e:bpac}).  The left side of that congruence vanishes in the quotient $\overline{B}_1$ of $\overline{B}_p$.
\end{proof}

Combining Corollary~\ref{c:cb}, Proposition~\ref{p:xb}, and  the sequence (\ref{e:nablaext}), we have constructed the solid arrows in a commutative diagram as follows:
\begin{equation}
\begin{gathered}
 \begin{tikzcd}
0 \ar[r] & \overline{\Cl_\Sigma^{\Sigma'}}(H)_{R^\#} \ar[r] \ar[d,two heads,"\alpha"] & \nabla_{\Sigma}^{\Sigma'}(H)_{R^\#} \ar[r] \ar[d, dotted, "\beta"] & 
(X_{H, \Sigma})_{R^\#} \ar[r] \ar[d,two heads, "\gamma"] & 0 \\ 
0 \ar[r] & \overline{B}_0(\bpsi^{-1}) \ar[r] & \overline{B}_p(\bpsi^{-1}) \ar[r] & 
\overline{B}_1(\bpsi^{-1}) \ar[r] & 0.
\end{tikzcd}
\end{gathered}
\label{e:commdiag}
\end{equation}
Here $\overline{\Cl_\Sigma^{\Sigma'}}(H)_{R^\#}$ denotes the image of $\Cl_\Sigma^{\Sigma'}(H)_{R^\#}$ in $\nabla_{\Sigma}^{\Sigma'}(H)_{R^\#}$.

\begin{theorem} \label{t:beta}
 There exists an $R^{\#}$-module surjection $\beta \colon \nabla_{\Sigma}^{\Sigma'}(H)_{R^\#} \longtwoheadrightarrow \overline{B}_p(\bpsi^{-1})$ completing the commutative diagram (\ref{e:commdiag}).
\end{theorem}

\begin{proof}  As we now explain, the essential content of this theorem is property \ref{i:ext}, i.e.\ Lemma~\ref{l:ext},  which gives a Galois cohomological interpretation of the extension class corresponding to $\nabla_{\Sigma}^{\Sigma'}(H)^-$.
Let \begin{align*}
\eta_1 &\in \Ext^1_{R^\#}((X_{H, \Sigma})_{R^\#}, \overline{\Cl_\Sigma^{\Sigma'}}(H)_{R^\#}), \\
\eta_2 &\in \Ext^1_{R^\#}(\overline{B}_1(\bpsi^{-1}), \overline{B}_0(\bpsi^{-1}))
\end{align*}
be the extension classes corresponding to the rows of the diagram (\ref{e:commdiag}). Pushout by $\alpha$ and pullback by $\gamma$, respectively, yield classes
\[ \alpha_* \eta_1,  \gamma^*\eta_2 \in \Ext^1_{R^\#}((X_{H, \Sigma})_{R^\#}, \overline{B}_0(\bpsi^{-1})). \]
Proving that $\alpha_* \eta_1 = \gamma^*\eta_2$ will yield the desired result.  For then, if we let $Z_1, Z_2$ denote  $R^{\#}$-modules representing these extension classes, we obtain a commutative diagram:
\begin{equation}
\begin{gathered}
 \begin{tikzcd}
0 \ar[r] & \overline{\Cl_\Sigma^{\Sigma'}}(H)_{R^\#} \ar[r] \ar[d, two heads, "\alpha"] & \nabla_{\Sigma}^{\Sigma'}(H)_{R^\#} \ar[r] \ar[d,two heads] & 
(X_{H, \Sigma})_{R^\#} \ar[r] \ar[d, equal] & 0 \\ 
0 \ar[r] & \overline{B}_0(\bpsi^{-1}) \ar[r] \ar[d, equal] &Z_1 \ar[r] \isoarrow{d} & 
(X_{H, \Sigma})_{R^\#} \ar[r] \ar[d,equal] & 0 \\ 
0 \ar[r] & \overline{B}_0(\bpsi^{-1}) \ar[r] \ar[d,equal] &Z_2 \ar[r] \ar[d,two heads] & 
(X_{H, \Sigma})_{R^\#} \ar[r] \ar[d,two heads,"\gamma"] & 0 \\ 
0 \ar[r] & \overline{B}_0(\bpsi^{-1}) \ar[r] & \overline{B}_p(\bpsi^{-1}) \ar[r] & 
\overline{B}_1(\bpsi^{-1}) \ar[r] & 0.
\end{tikzcd}
\end{gathered}
\label{e:commdiag2}
\end{equation}
The desired surjection $\beta$ is  given by  composition of the middle vertical arrows.

To prove  $\alpha_* \eta_1 =  \gamma^*\eta_2 $, we interpret these extension classes in terms of  Galois cohomology using the isomorphism
\begin{equation} \label{e:extxb}
 \Ext^1_{R^\#}((X_{H, \Sigma})_{R^\#}, \overline{B}_0(\bpsi^{-1})) \cong 
\bigoplus_{v \in \Sigma} H^1(G_v, \overline{B}_0(\bpsi^{-1}))
\end{equation}
described in (\ref{e:ext}).  We will show that the component at $v$ for both 
$\alpha_* \eta_1$ and $\gamma^*\eta_2 $ is equal to the unique class whose inflation to 
$H^1(G_{F_v}, \overline{B}_0(\bpsi^{-1}))$ is  $\res_{G_{F,v}}^{G_F} \kappa$.

Lemma~\ref{l:ext} implies that under the isomorphism (\ref{e:extxb}), we have
$\alpha_* \eta_1 = (\alpha_* \lambda_v)_{v \in \Sigma}$ where $\lambda_v$ is the cohomology class defined in \S\ref{s:ext}. Reviewing this definition, the class $\alpha_* \lambda_v$ is given as follows.  Consider the class in 
\[ H^1(G_H, \overline{B}_0(\bpsi^{-1})) = \Hom_{\cont}(G_H, \overline{B}_0(\bpsi^{-1})) \] given by the composition of the homomorphisms
\[ \begin{tikzcd}
G_H \ar[r] & \Gal(L/H) \ar[r,"{\rec_{L/H}^{-1}}"] & \Cl_{\Sigma}^{\Sigma'}(H)^{-} \ar[r,"\alpha"] & 
\overline{B}_0(\bpsi^{-1}).
\end{tikzcd}
\]
By the explicit formula for $\alpha$ given in Corollary~\ref{c:cb}, this composition is simply $\sigma \mapsto \overline{b}(\sigma)$ for $\sigma \in G_H$.  Next we must lift this homomorphism to a (unique) class in $H^1(G_F, \overline{B}_0(\bpsi^{-1}))$; but of course we already have a specific lift, namely $\kappa$.  The class $\alpha_* \lambda_v \in H^1(G_v, \overline{B}_0(\bpsi^{-1}))$ is then by definition the unique class whose inflation is equal to $\res_{G_{F,v}}^{G_F} \kappa \in H^1(G_{F_v}, \overline{B}_0(\bpsi^{-1}))$.

Next we compute $\gamma^*\eta_2$ in these Galois cohomological terms using the explication of the isomorphism (\ref{e:extxb}) given in the discussion between (\ref{e:ext}) and Lemma~\ref{l:ext}.
Let \[ \gamma_v \in H^1(G_v, \overline{B}_0(\bpsi^{-1})) \] denote the component at $v \in \Sigma$ of  $\gamma^*\eta_2$.  Then by the definition of $\gamma$ and in view of (\ref{e:alphadef}), we have
\[ \gamma_v(g) =(1 - \bpsi^{-1}(g))\frac{A_v}{C_v}, \quad g \in G_v. \]
By our favorite equation (\ref{e:bpac}), the right hand side has image $\kappa(\sigma)$ in $\overline{B}_0(\bpsi^{-1})$, where $\sigma$ is any lift of $g$ to $G_{F,v}$.
In other words, $\gamma_v$ is the unique class whose inflation is equal to $\res_{G_{F,v}}^{G_F} \kappa \in H^1(G_{F_v}, \overline{B}_0(\bpsi^{-1}))$.  This was the same description of $\alpha_* \lambda$ given above.

This concludes the proof that $\alpha_* \eta_1 =  \gamma^*\eta_2 $ and completes the proof of the theorem.
\end{proof}

\subsection{Calculation of Fitting Ideal}

In this section we will prove that
 \[ \Fitt_R(\overline{B}_p) \subset (\Theta^\#)\]
(note we have not twisted by $\bpsi^{-1}$ here) and use this to conclude the desired result \[ \Fitt_R(\Sel_{\Sigma}^{\Sigma'}(H)_R) \subset (\Theta^\#). \]
\begin{lemma} The module $B_0$ can be generated over $R$ by finitely many elements $b_1, \dotsc, b_n$ that are non-zerodivisors (i.e. invertible) in $K = \Frac(\T_\fM)$. \label{l:bgen}
\end{lemma}

\begin{proof} Recall that $K = \Frac(\T_\fM) = \prod_{f \in \overline{M}} E$, with each factor corresponding to a cuspidal eigenform $f$, and $E = \Frac(\cO)$ a finite extension of $\Q_p$.  We will denote the $i$th factor $E$ in this finite product as $E_i$, and the corresponding eigenform by $f_i$.  The homomorphism $\rho$ is continuous and hence $B_0$ is a compact subset of $K$.  It is therefore finitely generated over $\cO$ and hence finitely generated over $R$.  

Suppose we start with any finite generating set $b_1, \dotsc, b_n$.  We claim we can alter these generators  such that each $b_i$ is a non-zerodivisor in $K$, i.e.\ such that the projection of each $b_i$ to each factor $E_j$ is nonzero.  We prove this by induction on the total number of zero projections of the $b_i$ onto the $E_j$.  Suppose that $b_i$ has zero projection onto some factor $E_j$.  Since the individual representations $\rho_{f_j}$ are irreducible, some other $b_k$ must have nonzero projection onto $E_j$.  If we replace $b_i$ by $b_i + t b_k$ for any nonzero  $t \in \cO$,  the new $b_i$  has nonzero projection onto $E_j$.  Furthermore, at most finitely many $t$ introduce a new zero projection of $b_i$ onto some other $E_{j'}$.  Avoiding these finitely many $t$, we can choose a $t$ that decreases the total number of zeros.  Furthermore, the replacement $b_i \mapsto b_i + t b_k$ does not change the span of the $b_i$, and hence preserves the property that they generate $B_0$ over $R$.  Continuing in this fashion, we can repeatedly reduce the number of zero projections of the $b_i$ on to the $E_j$ until there are none remaining.
This concludes the proof.
\end{proof}

\begin{theorem} \label{t:fitbp}
We have $\Fitt_R(\overline{B}_p) \subset (\Theta^\#)$.
\end{theorem}

\begin{proof} Let  $\fp_1, \dotsc, \fp_r$ denote the primes of $F$ above $p$ not contained in $\Sigma$ (i.e.\ those dividing $\fP'$). For each $\fp_i$, choose an element $\sigma_i \in G_{\fp_i} \subset G_F$ that lifts $\rec(\varpi_i) \in G_{\fp_i}^{\ab}$, where
$\varpi_i$ is a uniformizer for $F_{\fp_i}$. Set $c_i = b(\sigma_i)\bpsi(\fp_i) \epsilon_{\cyc}^{1-k}(\sigma_i)$. By (\ref{e:beq}) we have
\[ c_i = b(\sigma_i)\bpsi(\fp_i) \epsilon_{\cyc}^{1-k}(\sigma_i) = \frac{A_i}{C_i}(U_{\fp_i} - \bpsi(\fp_i) + I) \in B. \]
Here and throughout this proof, we use the notation $a = b + I$ to mean $a = b + z$ for some $z \in I$ to avoid needing to add distinct variable names for each such $z$ that appears.  We have also written $A_i/C_i$ for $A_{\fp_i}/C_{\fp_i}$.

Let $b_1, \dotsc, b_n$ be $R$-module generators of $B$ that are not zerodivisors in $K$; we can choose the generators of $B_0$ as given by Lemma~\ref{l:bgen} along with the $A_\fp/C_\fp$ for finite $\fp \in \Sigma$.
To calculate $\Fitt_R(\overline{B}_p) $ we use the generating set $c_1, \dotsc, c_r, b_1, \dotsc, b_n$ for $\overline{B}_p$.  Of course, these first $r$ generators are not necessary, but including them will aid us in proving the theorem.  Suppose we have a matrix 
\[ M \in M_{(n+r) \times (n+r)}(R) \]
such that each row of $M$ represents a relation amongst our generators, i.e.\ such that \[ M(c_1, \dotsc, c_r, b_1, \dotsc, b_n)^T \equiv 0 \text{ in } (\overline{B}_p)^{n+r}.\]
By definition of Fitting ideal, the theorem will follow if we can show that $\det(M) \in (\Theta^\#)$.

 Write $M   = (W | Z)$ in block matrix form, where \[ W = (w_{ij}) \in M_{(n+r)\times r}(R), \qquad 
Z = (z_{ij}) \in M_{(n+r)\times n}(R). \]
Note that by (\ref{e:beq}), since $\bpsi$ and $\eta_{\fp_i}$ are unramified at $\fp_i$ and $a(\sigma) \equiv \epsilon_{\cyc}^{k-1}(\sigma) \pmod{I}$, we have 
\[ b(I_{\fp_i}) \subset \frac{A_i}{C_i} I. \]
Also, since the $b_i$ generate $B$, every element of $IB$ can be written as a sum of elements of the form   $b_i t_i$ with $t_i \in I$.
Therefore each relation \[ \sum_{j=1}^r w_{ij} c_j + \sum_{j=1}^n z_{ij} b_j \equiv  0 \text{ in } \overline{B}_p \] can be expressed as an equality in $B$ as
\begin{equation} \label{e:zerorow}
  \sum_{j=1}^r  \frac{A_j}{C_j}(w_{ij}(U_{\fp_j} - \bpsi(\fp_j)) + I) +\sum_{j=1}^n (z_{ij} +I + p^m R) b_j  = 0. \end{equation}
Here, as above, we use the notation ``$\ + \ I$" as shorthand for ``$\ + \ z$ for some $z \in I$," and similarly for ``$\ + \ p^mR$." 
It  follows from (\ref{e:zerorow}) that if we define a matrix $M' \in M_{(n+r) \times (n+r)}(K)$  in block form by
\[ M' = \left( \frac{A_j}{C_j}(w_{ij}(U_{\fp_j} - \bpsi(\fp_j)) + I) \ \ \ \ |  \ \ \ \ (z_{ij} + I + p^mR) b_j \right), \]
then $\det(M') = 0$ in $K$ since it has rows that sum to 0.  We can cancel the factors ${A_j}/{C_j}$ and $b_j$ scaling the columns of $M'$, since these are non-zerodivisors in $K$.  We obtain that $\det(M'') = 0$ where 
\[  M'' = \left( (w_{ij}(U_{\fp_j} - \bpsi(\fp_j)) + I) \ \ \ \ |  \ \ \ \ (z_{ij} + I+p^mR)  \right) \in M_{(n+r)\times (n+r)}(\tilde{\T}_\fm).
\]
Recall from the notation of Theorem~\ref{t:yxt}, we have  \[ \prod_{i=1}^r (U_{\fp_j} - \bpsi(\fp_j)) = U \in \tilde{\T}_\fM. \]
Taking the determinant of $M''$ and applying $\varphi$, we obtain
 \[ 0 = \varphi( \det(M''))= \varphi(U) ( \det(M) + p^mR) \text{ in } W. \]  Therefore, by the last statement in Theorem~\ref{t:yxt}, we obtain that \begin{equation} \label{e:detm}
 \det(M) + p^mR \in (\Theta^\#). \end{equation}  Since $\Theta^\#$ divides $p^m$, $\det(M) \in (\Theta^\#)$ as desired.
\end{proof}

It is worth noting that the last statement of Theorem~\ref{t:yxt}, which allowed for the deduction of (\ref{e:detm}), was heavily dependent on the presence of the factor $x$ in our congruence (\ref{e:case1f}) in case 1a.  The fact that we are able to construct a ``stronger congruence" (i.e.\ modulo $x\Theta^\#$ rather than just $\Theta^\#$) is essential for our proof.

\begin{corollary} We have
 \[ \Fitt_R(\Sel_{\Sigma}^{\Sigma'}(H)_R) \subset (\Theta^\#). \]
\end{corollary}
 
 \begin{proof}  Theorem~\ref{t:fitbp} states that  $\Fitt_R(\overline{B}_p) \subset (\Theta^\#)$, hence 
 \[  \Fitt_{R^\#}(\overline{B}_p(\bpsi^{-1})) \subset (\Theta). \]
 Theorem~\ref{t:beta} states that there is an $R$-module surjection  $\nabla_{\Sigma}^{\Sigma'}(H)_{R^\#} \twoheadrightarrow \overline{B}_p(\bpsi^{-1})$, whence
  \[  \Fitt_{R^\#}(\nabla_{\Sigma}^{\Sigma'}(H)_{R^\#}) \subset (\Theta). \]
 Finally, by Corollary~\ref{c:selprin}, we obtain
  \[ \Fitt_R(\Sel_{\Sigma}^{\Sigma'}(H)_R) \subset (\Theta^\#) \]
as desired.
 \end{proof}
 
\appendix
\section{Appendix: Construction and Properties of $\nabla$} \label{s:rw}

Let $\Sigma, \Sigma'$ denote finite disjoint sets of places of $F$ with $\Sigma \supset S_\infty$, such that $\Sigma'$ satisfies condition (\ref{e:drcond}) from the introduction.

In this section we define the module $\nabla_{\Sigma}^{\Sigma'} = \nabla_\Sigma^{\Sigma'}(H)$ following the methods of Ritter--Weiss \cite{rw}.
We do not yet enforce any additional assumptions on the sets $\Sigma, \Sigma'$.  
 Later in this appendix we will impose  assumptions as necessary to obtain certain desirable properties
   of $\nabla_\Sigma^{\Sigma'}$.

\subsection{Construction of $\nabla$} \label{s:construction}

To define  $\nabla_\Sigma^{\Sigma'}$, we introduce an auxiliary finite set of primes $S'$ of $F$ satisfying the following properties:
\begin{itemize}
\item $S' \supset \Sigma$ and $S' \cap \Sigma' = \emptyset$. \item $S' \cup \Sigma' \supset S_{\ram}$.
\item $\Cl_{S'}^{\Sigma'}(H) = 1$.
\item $\cup_{w \in S'_H} G_w = G$, where $G_w \subset G$ is the decomposition group at $w$.
\end{itemize}
Although it is not used in this work, we prove in \S\ref{s:indep} that the construction of $\nabla_\Sigma^{\Sigma'}$ is independent of the chosen auxiliary set $S'$.

For each place $v$ of $F$, we fix a place $w$ of $H$ above $v$.  Ritter--Weiss define a $\Z[G]$-module $V_w$ sitting in an exact sequence:
\begin{equation} \label{e:vseq}
\begin{tikzcd}
 0 \ar[r] &   H_w^* \ar[r] &  V_w \ar[r] &  \Delta G_w \ar[r] &  0,
 \end{tikzcd}
\end{equation}
where as usual $\Delta G_w \subset \Z[G_w]$ denotes the augmentation ideal.
For $w$ finite, they define a $\Z[G]$-module $W_w$ sitting in an exact sequence
\begin{equation} \label{e:wseq}  \begin{tikzcd}
  0  \arrow[r] &  \cO_w^* \arrow[r] &  V_w \arrow[r] &  W_w \ar[r] &  0. 
  \end{tikzcd}
  \end{equation}

We recall the construction  of these modules. Let $H_w^{\ab} \supset L_w^{\nr}$ denote the maximal abelian and unramified extensions of $H_w$, respectively. There are canonical short exact sequences 
\begin{equation} 
\begin{tikzcd}[row sep=tiny]
0 \ar[r] & \text{W}(H_w^{\ab}/H_w) \cong H_w^* \ar[r] & \text{W}(H_w^{\ab}/F_v) \ar[r,"{\pi_V}"] & G_w \ar[r] & 0 \\
0\ar[r] &  \text{W}(H_w^{\nr}/H_w) \cong \mathbf{Z} \ar[r] & \text{W}(H_w^{\nr}/F_v) \ar[r,"{\pi_W}"] & G_w \ar[r] &  0,
\end{tikzcd}
\label{e:whw}
\end{equation}
where $\text{W}$ denotes the Weil group. Let $\Delta V$ denote the (absolute) augmentation ideal of $\text{W}(H_w^{\ab}/F_v)$ and let $\Delta(V, H_w^*)$ denote the relative augmentation ideal corresponding to $\pi_V$. Define $\Delta W$ and $\Delta(W, \Z)$ similarly from the corresponding terms in the second exact sequence in (\ref{e:whw}). Then we define
\begin{align}
\begin{split}
V_w &= V_w(H_w) = \Delta V/( \Delta V)\Delta(V, H_w^*), \\ 
\label{e:wdef} W_w & = W_w(H_w) = \Delta W/ (\Delta W)\Delta(W, \Z).
\end{split}
\end{align}

 We adopt the following notation of \cite{greither}: for a collection of $G_w$-modules $M_w$, we define 
 \[ \prod_{v}^\sim M_w := \prod_v \Ind_{G_w}^{G} M_w. \]
 
Let $U_w^1 \subset \cO_w^*$ denote the group of 1-units.  
 Define
\begin{align*}
J &= \prod_{v \not\in \Sigma \cup \Sigma'}^{\sim}  \cO_w^*  \prod_{v \in \Sigma}^{\sim}  H_w^* \prod_{v \in \Sigma'}^{\sim} U_w^1, \\
 V &= \prod_{v \not \in S'  \cup \Sigma'}^{\sim} \cO_w^* \prod_{v \in S'}^{\sim} V_w  \prod_{w \in \Sigma'}^{\sim} U_w^1, \\
  W &= \prod_{v \in S' - \Sigma}^{\sim} W_w \prod_{v \in \Sigma}^{\sim} \Delta G_w,
  \end{align*}
so that we have an exact sequence of $G$-modules
\begin{equation} \label{e:jvw}
\begin{tikzcd}
 0 \ar[r] & J \ar[r] &  V \ar[r] &  W \ar[r] &  0. 
 \end{tikzcd}
 \end{equation}
Next, we consider the canonical extension (see pg. 148 of \cite{rw}) 
\begin{equation} \label{e:cog}
\begin{tikzcd}
 0 \ar[r] &  C_H = \A_H^*/H^* \ar[r] &  \fO \ar[r] &  \Delta G \ar[r] &  0  
 \end{tikzcd}
\end{equation}
associated to the global fundamental class in $H^2(G, C_H)$.  
 
As in \cite{rw}*{Theorem 1}, there is a  map between the extensions (\ref{e:jvw}) and (\ref{e:cog}):
\begin{equation}
\begin{gathered}
\begin{tikzcd}
0 \ar[r] & J \ar[r]  \ar[d,"\theta_J"] & V \ar[r] \ar[d,"\theta"] & W \ar[r] \ar[d,"\theta_W"] & 0 \\
0 \ar[r] & C_H \ar[r] & \fO \ar[r] & \Delta G \ar[r] & 0.
\end{tikzcd}
 \label{e:rwdiagram}
\end{gathered}
\end{equation}
Our map $\theta$ is the restriction of the map $\theta$ appearing in \cite{rw}; in the context of \cite{rw}, the map $\theta$ is shown to be surjective.  We must show that it remains surjective after restricting to our module $V$.

\begin{lemma} The map $\theta$ in $(\ref{e:rwdiagram})$ is surjective.
\end{lemma}

\begin{proof} The same proof as in \cite{rw}*{Page 162} works, and for completeness we recall it.  Define
\begin{align*}
J' &= \prod_{v \not\in \Sigma' \cup S'}^{\sim}  \cO_w^*  \prod_{v \in S'}^{\sim}  H_w^* \prod_{v \in \Sigma'}^{\sim} U_w^1, \\
  W' &=  \prod_{v \in S'}^{\sim} \Delta G_w.
  \end{align*}
  We then obtain
  \begin{equation} \label{e:jpc}
 \begin{gathered}
\begin{tikzcd}
0 \ar[r] & J' \ar[r]  \ar[d,"\theta_{J'}"] & V \ar[r] \ar[d,"\theta"] & W' \ar[r] \ar[d,"\theta_{W'}"] & 0 \\
0 \ar[r] & C_H \ar[r] & \fO \ar[r] & \Delta G \ar[r] & 0,
\end{tikzcd}
\end{gathered}
\end{equation}
where the middle vertical arrow $\theta$  is the same as in (\ref{e:rwdiagram}).  Yet now $\theta_{J'}$ is surjective, since its cokernel is $\Cl_{S'}^{\Sigma'}(H)=1$, by our assumption on $S'$.  It remains to see that $\theta_{W'}$ is surjective, and this follows easily from the other assumptions on $S'$ (see the argument below diagram 3 on page 162 of \cite{rw}).
\end{proof}

Applying the snake lemma to (\ref{e:rwdiagram})  yields an exact sequence
\begin{equation} \label{e:snake}
\begin{tikzcd}
 0 \ar[r] &  \cO_{H, \Sigma, \Sigma'}^* \ar[r] &  V^\theta \ar[r] &  W^\theta \ar[r] &  \Cl_\Sigma^{\Sigma'}(H) \ar[r] &  0, 
\end{tikzcd}
\end{equation}
where $V^\theta = \ker \theta$, $W^\theta = \ker \theta_W$. 

We next construct an injection from $W$ to a free $\Z[G]$-module. 
Write $W_w^* = \Hom_\Z(W_w, \Z)$.  By  \cite{rw}*{Lemma 5}, there is a  commutative diagram 
of $\Z[G_w]$-modules with exact rows:
 \begin{equation} \label{e:ww2}
 \begin{tikzcd}
0 \ar[r] & W_w \ar[r,"{(\alpha_w,\beta_w)}"] \ar[d,"\alpha_w"] & \Z[G_w]^2 \ar[d,"\pi_1"] \ar[r] & W_w^* \ar[r] \ar[d] & 0 \\
0 \ar[r] & \Delta G_w \ar[r] & \Z[G_w] \ar[r] & \Z \ar[r] & 0.
\end{tikzcd}
\end{equation}
 Here $\pi_1$ denotes projection onto the first factor.  Let us recall the definition of the maps $\alpha_w, \beta_w$.  The map $\alpha_w$ is induced by the canonical projection $\pi_W \colon \Delta W \longrightarrow \Delta G_w$ (see (\ref{e:whw}) and (\ref{e:wdef})) and sits in a short exact sequence
 \begin{equation} \label{e:alphaw} \begin{tikzcd}
 0 \ar[r] & \Z \ar[r] & W_w \ar[r, "\alpha_w"] & \Delta G_w \ar[r] & 0 
 \end{tikzcd}
 \end{equation}
 \cite{rw}*{Lemma 5(b)}.
To define $\beta_w$, we first define a map
   \[ \beta_w^0 \colon W_w \longrightarrow \Z[G_w/I_w]. \] 
 Let $\sigma \in \text{W}(H_w^{\nr}/F_v)$ and write $\overline{\sigma}$ for the image of $\sigma$ in $G_w/I_w = \Gal(H_w^{I_w}/F_v)$.  Define the integer $n$ by
 $\sigma|_{F_v^{\nr}} = \sigma_v^n$, where $\sigma_v \in \text{W}(F_v^{\nr}/F_v) \cong \Z$ is the Frobenius element.
Writing $x = \sigma -1$, we define $\beta_w^0(x) \in \Z[G_w/I_w]$ to be the unique element whose augmentation is equal to $n$ and such that \begin{equation} \label{e:iwred}
 \overline{\alpha_w(x)} = \overline{\sigma} - 1 = (\sigma_w-1)\beta_w^0(x)
 \end{equation}
in $\Z[G_w/I_w],$
 where $\sigma_w = \overline{\sigma}_v \in G_w/I_w$ is the Frobenius element.
 To be explicit, we have
 \[ \beta_w^0(\sigma-1) =
 \begin{cases}
  1 + \sigma_w + \sigma_w^2 + \cdots + \sigma_w^{n-1} & \text{if } n > 0 \\
  0 & \text{if } n=0 \\
  -(\sigma_w^{-1} + \sigma_w^{-2} + \cdots + \sigma_w^{n}) & \text{if } n < 0.
  \end{cases} \]
  We define
 \begin{equation} \label{e:betawdef}
  \beta_w(x) = \N I_w \cdot \beta_w^0(x) \in \Z[G_w]. \end{equation}

The maps $\alpha_w, \beta_w$ allow us to give an injection from $W$ to a finite free $\Z[G]$-module.  Write
   \[ S'_\ram = S_{\ram} \setminus (\Sigma \cup \Sigma')  \subset S'. \]
Define
\[ B  =   \prod_{v \in S' \setminus S'_{\ram}} \Z[G]  \prod_{v \in S'_{\ram}} \Z[G]^2. \]
We then have an injection $\gamma\colon W \longrightarrow B$ defined componentwise as follows:
\begin{itemize} \item For $v \in \Sigma$, the map $\gamma_v$ is induced by the canonical injection $\Delta G_w \subset \Z[G_w]$.
\item For $v \in S'_\ram$, the map $\gamma_v$ is induced by the injection $(\alpha_w, \beta_w)$ in (\ref{e:ww2}).
\item For $v \in S' \setminus (\Sigma \cup S_{\ram}')$,  the map $\gamma_v$ is induced by $\beta_w$, which is an isomorphism since $v$ is unramified in $H$ (see \cite{rw}*{Lemma 5}).
\end{itemize}
Let \[ Z =  \prod_{v \in \Sigma}^{\sim} \Z \prod_{S'_\ram}^{\sim}  W_w^*. \]
We then have a commutative diagram with   exact rows:
\begin{equation}
\begin{gathered}
 \label{e:wbdiagram}
\begin{tikzcd}
0 \ar[r] & W \ar[r,"\gamma"]  \ar[d,"\theta_W"] & B \ar[r] \ar[d,"\theta_B"] & Z
\ar[r] \ar[d,"\theta_Z"] & 0 \\
0 \ar[r] & \Delta G \ar[r] & \Z[G] \ar[r] & \Z \ar[r] & 0.
\end{tikzcd}
\end{gathered}
\end{equation}
The vertical maps are defined componentwise as follows:
\begin{itemize}
\item If $v \in \Sigma$, then $\theta_W$ and $\theta_B$ are the identity map, and $\theta_Z$ is the augmentation.
\item If $v \in S'_\ram$, then $\theta_W$, $\theta_B$, and $\theta_Z$ are induced from the vertical maps in (\ref{e:ww2}).
\item If $v \in S' \setminus (\Sigma \cup S_{\ram}')$, then $\theta_W$ is again induced from the first vertical map in (\ref{e:ww2}),
namely  $\alpha_w = (\sigma_w - 1) \cdot \beta_w$.   The map $\theta_B$ is multiplication by $\sigma_w - 1$.

\end{itemize}

Since $\theta_W$ is surjective, taking kernels in (\ref{e:wbdiagram}) yields a short exact sequence
\begin{equation} \label{e:wbx}
\begin{tikzcd}
 0 \ar[r] &  W^\theta \ar[r] &  B^\theta \ar[r] &  Z^\theta \ar[r] &  0. 
 \end{tikzcd}
 \end{equation}
Since $\theta_B$ is the identity on each component corresponding to  $v \in \Sigma$, and $\Sigma \supset S_\infty$ is nonempty,
it follows that:
\begin{equation} \label{e:bfree}
 B^\theta \text{ is a free } \Z[G]\text{-module of rank } \#S' + \#S_{\ram}' - 1. \end{equation}

\begin{definition}
We define $\nabla_\Sigma^{\Sigma'}$ to be the cokernel of the composite map 
\[ \begin{tikzcd}
V^\theta \ar[r] &  W^\theta  \ar[r] & B^\theta. \end{tikzcd} \]
\end{definition}
  Comparing (\ref{e:snake}) and (\ref{e:wbx}) we obtain two exact sequences
\begin{equation} \label{e:tate}
\begin{tikzcd}
 0  \ar[r] & \cO_{H, \Sigma, \Sigma'}^*  \ar[r] & V^\theta  \ar[r] & B^\theta  \ar[r] & \nabla_\Sigma^{\Sigma'}  \ar[r] & 0,   
 \end{tikzcd}
 \end{equation}
\begin{equation}  \label{e:seltr}  \begin{tikzcd}
0  \ar[r] & \Cl_\Sigma^{\Sigma'}(H)  \ar[r] & \nabla_{\Sigma}^{\Sigma'}  \ar[r] & Z^\theta  \ar[r] & 0. \end{tikzcd}
 \end{equation}
 
Consider now the following assumption:
\begin{itemize}
\item[(A1)] $\Sigma \cup \Sigma' \supset S_{\ram}$. 
\end{itemize}
 
 If assumption (A1) holds, then $S'_{\ram} = \emptyset$ so $Z = Y_{H, \Sigma}$ and $Z^\theta = X_{H, \Sigma}$.  The exact sequence (\ref{e:seltr}) can then be written:
 \begin{equation}  \label{e:seltr2}  \begin{tikzcd}
0  \ar[r] & \Cl_\Sigma^{\Sigma'}(H)  \ar[r] & \nabla_{\Sigma}^{\Sigma'}  \ar[r] & X_{H, \Sigma} \ar[r] & 0. \end{tikzcd}
 \end{equation}

We have therefore constructed the $\Z[G]$-module $\nabla_{\Sigma}^{\Sigma'}$ satisfying property \ref{i:exact} of \S\ref{s:properties}.  We now explore the other properties.

\subsection{Independence of $S'$} \label{s:indep}

We prove in this section that the module $\nabla_{\Sigma}^{\Sigma'}$---moreover, the extension class it defines via the sequence (\ref{e:seltr})---is independent of the choice of auxiliary set $S'$ used in the construction.  This follows (by identifying the construction for two different sets $S_1'$ and $S_2'$ with the one for $S_1' \cup S_2'$) from the following lemma.

\begin{lemma}  Let $\nabla$ and $\nabla'$ be constructed as in \S\ref{s:construction} with the same sets $\Sigma$, $\Sigma'$, but different auxiliary sets $S'$ and $S' \cup \{v\}$.  Then there is an equivalence between the extensions $(\ref{e:seltr})$ associated to $\nabla$ and $\nabla'$, i.e.\ an isomorphism $\nabla \longrightarrow \nabla'$ fitting into a commutative diagram
\begin{equation} \label{e:nnp} \begin{gathered}
\begin{tikzcd}
0 \ar[r] & \Cl_{\Sigma}^{\Sigma'}(H) \ar[r]\ar[d,equal] & \nabla \ar[r] \ar[d] &Z^\theta \ar[r] \ar[d,equal] & 0 \\
0 \ar[r] & \Cl_{\Sigma}^{\Sigma'}(H) \ar[r] & \nabla' \ar[r] &Z^\theta \ar[r] & 0.
\end{tikzcd}
\end{gathered}
\end{equation}
\end{lemma}

\begin{proof}
 Let $V, W, B$ denote the modules defined above in the construction of $\nabla_{\Sigma}^{\Sigma'}$ using the auxiliary set $S'$, and let $V', W', B'$ denote these same modules when $S'$ is replaced by $S' \cup \{v\}$.
 Then it follows from the definitions that there is an exact sequence
 \[  \begin{tikzcd}
  0  \ar[r] & V  \ar[r] & V'  \ar[r] & \Ind_{G_v}^{G} W_w  \ar[r] & 0, 
  \end{tikzcd}
 \]
 with $\Ind_{G_v}^{G} W_w \cong \Z[G]$  since $v$ is unramified in $H$ (\cite{rw}*{Lemma 5}).  Since the homomorphisms $\theta\colon V, V' \longrightarrow \fO$ are surjective and compatible with the map $V \longrightarrow V'$, it follows that we obtain
 \[ \begin{tikzcd}
  0  \ar[r] & V^\theta  \ar[r] & (V')^\theta  \ar[r] & \Z[G]  \ar[r] & 0. 
  \end{tikzcd}
  \]
 It is similarly clear from the definitions that we obtain the same exact sequences with $(V, V')$ replaced by $(W, W')$ and $(B, B')$; in fact in these cases
 the exact sequences are split.  The induced maps on the quotients $\Z[G]$ associated to $(V^\theta, (V')^\theta) \longrightarrow  (W^\theta, (W')^\theta)$ and $(W^\theta, (W')^\theta) \longrightarrow  (B^\theta, (B')^\theta)$ are the identity.  It follows that the induced map $\nabla \longrightarrow \nabla'$ is an isomorphism.
 
 The fact that this isomorphism fits into the commutative diagram (\ref{e:nnp}) is a similar arrow chase.  The map $B^\theta \longrightarrow Z^\theta$ forgets the components away from $\Sigma \cup S_{\ram}'$, so commutativity of the right square of (\ref{e:nnp}) is clear.  For the left square, recall how the map $\Cl_{\Sigma}^{\Sigma'}(H) \longrightarrow V$ is defined using the snake lemma.  Fix an element $x \in C_H$ representing a class  $\overline{x} \in \Cl_{\Sigma}^{\Sigma'}(H)$.  Its image in $\fO$ may be written $\theta(y)$ for some $y \in V$, whose image $\overline{y}$ in $W$ necessarily lies in $W^\theta$.  The image of $\overline{y}$ in $\nabla$ is the definition the image of $\overline{x}$  under $ \Cl_{\Sigma}^{\Sigma'}(H) \longrightarrow \nabla$.
When making the same calculation for $\nabla'$, we may choose the lift $\theta(y')$ for the image of $x$ in $\fO$, where $y'$ is the image of $y$ under $V \longrightarrow V'$.  Then the image of $\overline{y}'$ in $\nabla'$  is the image of $\overline{y}$ in $\nabla$, and we obtain commutativity of the left square of (\ref{e:nnp}).
\end{proof}

\subsection{Projectivity of Presentation} \label{s:pp}

In this section, we show that under an appropriate assumption, the module $V^\theta$ is projective over $\Z[G]$.

\begin{itemize}
\item[(A2)] $\Sigma'$ contains no primes of wild ramification, i.e.\ for every $v \in \Sigma'$, the inertia group $I_v \subset G_v \subset G$ has prime-to-$\ell$ order, where $\ell$ is the residue characteristic of $v$.
 \end{itemize}

We also consider the following simpler condition that is useful, for instance, when working over $\Z_p$ as in the main body of the paper.

\begin{itemize}
\item[(A$2'$)]   We work over a $\Z[G]$-algebra $R$ such that for every $v \in \Sigma' \cap S_{\ram}$,
the rational prime $\ell$ below $v$ is invertible in $R$. 
 \end{itemize}

\begin{lemma}  \label{l:projective}
Assuming condition (A2), the $\Z[G]$-module $V^\theta$ is projective with constant rank equal to $\#S' - 1$. 
Assuming condition (A2$\,'$), the $R$-module $V^\theta_R = V^\theta \otimes_{\Z[G]} R$ is projective with constant rank equal to $\#S' - 1$.
\end{lemma}

\begin{proof} Recall that a $G$-module $M$ is called {\em cohomologically trivial} if the Tate cohomology $\hat{H}^i(H, M)$ vanishes for all subgroups $H \subset G$ and all integers $i$.
We first claim that $V^\theta$ is cohomologically trivial.  For this, it suffices to show that $V$ is cohomologically trivial, since it is known that $\fO$ is cohomologically trivial \cite{nsw}*{Theorem 3.1.4(i)}.  

As we now explain, the module $V$ is the product of cohomologically trivial modules. Any $v \not\in S' \cup \Sigma'$ is unramified and hence $\Ind_{G_w}^G \cO_w^*$ is cohomologically trivial \cite{cf}*{\S VI.1.2, Proposition 1}.  Ritter--Weiss show that the module $\Ind_{G_v}^{G} V_w$ is cohomologically trivial \cite{rw}*{\S3, Proposition 2}.  

It remains to show that $U_w^1$ is $G_v$-cohomologically trivial for $v \in \Sigma'$.  
The argument of \cite{cf}*{\S VI.1.2, Proposition 1} again shows that $(U_w^1)^{I_w}$ is cohomologically trivial as a $(G_w/I_w)$-module.
By inflation-restriction, it therefore suffices to show that $U_w^1$ is cohomologically trivial as an $I_w$-module.
The assumption (A2) states that $I_w$ has prime-to-$\ell$ order, where $\ell$ is the residue characteristic of $v$.  Therefore multiplication by $\#I_v$ is invertible on the pro-$\ell$ group $U_w^1$, so cohomological triviality is automatic.  This proves the claim that $V^\theta$ is $G$-cohomologically trivial.

Next we note that (\ref{e:tate}) implies that $V^\theta$ is $\Z$-torsion free, since the modules $\cO_{H, \Sigma, \Sigma'}^*$ and $B^\theta$ are $\Z$-torsion  free.  A theorem of Nakayama then implies that $V^\theta$ is $\Z[G]$-projective (\cite{nak}*{Theorem 1}).

To adapt this argument when assuming (A$2'$) instead of (A2), note that by the argument of 
\cite{cf}*{Chapter VI, Proposition 3}, $U_w^1$ contains an open subgroup $U_w'$ that is cohomologically trivial.    But the index $[U_w^1 \colon U_{w'}]$ is a power of $\ell$, which is invertible in $R$, so $(U_w^1)_R \cong (U_w')_R$.  We can therefore replace $U_w^1$ by $U_w'$ and proceed as above.

To conclude, we show that $V^\theta$ has constant rank equal to $\#S' - 1$. It suffices to show that for every character \[ \chi \colon \Z[G] \longrightarrow \overline{\Q}^* \]
we have \[ \dim_{\overline{\Q}} V^\theta_\chi = \#S' - 1, \]
where 
\[  V^\theta_\chi = V^\theta \otimes_{\Z[G]} \overline{\Q}(\chi). \]
  Here $\overline{\Q}(\chi)$ denotes the 1-dimensional $\overline{\Q}$-vector space on which $G$ acts by $\chi$.  Note that $\overline{\Q}(\chi)$ is flat over $\Z[G]$.
  
The sequence (\ref{e:seltr}) implies that 
\begin{align*}
 \dim_{\overline{\Q}} (\nabla_{\Sigma}^{\Sigma'})_\chi &= \dim_{\overline{\Q}} Z^\theta_\chi \\
 & = \dim_{\overline{\Q}}(X_{H,\Sigma})_\chi +  \sum_{v \in S'_\ram} \dim_{\overline{\Q}}(\Ind_{G_w}^{G} W_w^*)_\chi. 
 \end{align*}
Yet the Dirichlet unit theorem implies
$\dim_{\overline{\Q}} (\cO_{H, \Sigma, \Sigma'}^*)_\chi =  \dim_{\overline{\Q}}(X_{H,\Sigma})_\chi$, so (\ref{e:tate}) implies that
\begin{equation} \label{e:vbw}
\dim_{\overline{\Q}} V^\theta_\chi = \dim_{\overline{\Q}} B^\theta_\chi - \sum_{v \in S'_\ram} \dim_{\overline{\Q}}(\Ind_{G_w}^{G} W_w^*)_\chi. \end{equation}
Now combining (\ref{e:vseq}) and (\ref{e:wseq}) one obtains a short exact sequence (see also \cite{rw}*{Lemma 5})
\[  \begin{tikzcd} 0 \ar[r] & \Z \ar[r] & W_w \ar[r] & \Delta G_w \ar[r] & 0, \end{tikzcd}\]
from which it follows that each term in the sum on the right of (\ref{e:vbw}) is equal to $1$.

Therefore \[ \dim_{\overline{\Q}} V^\theta_\chi = \dim_{\overline{\Q}} B^\theta_\chi  - \#S'_\ram = \#S' - 1. \]
\end{proof}

As an immediate corollary, we find:

\begin{lemma} Assuming (A1) and (A2), the exact sequence
 \begin{equation} \label{e:pp}
\begin{tikzcd}
V^\theta  \ar[r] & B^\theta  \ar[r] & \nabla_{\Sigma}^{\Sigma'}(H)  \ar[r] & 0 
\end{tikzcd}
 \end{equation}
 is a locally quadratic presentation of $ \nabla_{\Sigma}^{\Sigma'}(H)$ over $\Z[G]$.
 
 Assuming (A1) and (A2$\,'$), the exact sequence
 \begin{equation} \label{e:pp2}
  \begin{tikzcd}
   V^\theta_R  \ar[r] & B^\theta_R  \ar[r] & \nabla_{\Sigma}^{\Sigma'}(H)_R  \ar[r] & 0 
   \end{tikzcd}
    \end{equation}
 is a locally quadratic presentation of $ \nabla_{\Sigma}^{\Sigma'}(H)_R$ over $R$.
\end{lemma}

\begin{remark}  In view of the proof of Lemma~\ref{l:projective}, perhaps the ``right" thing to do when $\Sigma'$ contains wildly ramified primes is to replace $U_w^1$ in the definition of $V$ by a $G_w$-cohomologically trivial open subgroup $U_w'$.  This will yield a different module $\nabla_{\Sigma}^{\Sigma'}$, sitting in exact sequences analogous to (\ref{e:tate}) and (\ref{e:seltr}), where $\Cl_{\Sigma}^{\Sigma'}(H)$ is replaced by a more general ray class group and $\cO_{H, \Sigma, \Sigma'}^*$ is replaced by a subgroup.  Then (\ref{e:pp}) would remain a projective presentation of $\nabla_{\Sigma}^{\Sigma'}$.  Since we have no present applications of such a construction, we do not pursue this further here.
\end{remark}

\begin{remark}  In the main body of the paper, we work over a $\Z_p[G]$-algebra $R$.  Furthermore the sets $\Sigma$ and $\Sigma'$ defined in (\ref{e:bigsigmadef}) and (\ref{e:sigmapdef}) are easily seen to satisfy (A1) and (A2$'$).  Since $\Z_p[G]$ is a product of local rings, the projective module of constant rank $V^\theta_R$ is free.
\end{remark}

\subsection{Transpose of $\nabla$} \label{s:trnabla}

In this section we assume (A2), but not (A1).  In the previous section we showed that $V^\theta$ is  projective under the assumption of (A2).
We now compute the transpose of $\nabla_{\Sigma}^{\Sigma'}(H)$ associated to the projective presentation 
(\ref{e:pp}),
namely, 
\begin{equation}\label{e:ntdef} \begin{tikzcd}
\nabla_{\Sigma}^{\Sigma'}(H)^{\tr}= \coker((B^\theta)^*  \ar[r] & (V^\theta)^*). 
\end{tikzcd}
 \end{equation}
When $\Sigma \supset S_{\ram}$, a version of the following lemma is proved in \cite{bks}.

\begin{lemma}  \label{l:trnabla}
Assume (A2).  With $\nabla_{\Sigma}^{\Sigma'}(H)^{\tr}$ defined as in $(\ref{e:ntdef})$, we have 
\[ \nabla_{\Sigma}^{\Sigma'}(H)^{\tr} \cong \Sel_\Sigma^{\Sigma'}(H). \]
Similarly if we assume (A2$\,'$) instead of (A2), the transpose of $\nabla_{\Sigma}^{\Sigma'}(H)_R$ associated to the projective presentation (\ref{e:pp2}) satisfies
\[ \nabla_{\Sigma}^{\Sigma'}(H)^{\tr}_R \cong \Sel_\Sigma^{\Sigma'}(H)_R. \]
\end{lemma}

\begin{proof}  Assume (A2).
We will relate (\ref{e:ntdef}) to the presentation for $ \Sel_\Sigma^{\Sigma'}(H)$ given in (\ref{e:sels}).
There is a natural isomorphism of functors from the category of $\Z[G]$-modules to itself
\[  \Hom_\Z( -, \Z) \cong \Hom_{\Z[G]}(-, \Z[G]), \qquad \varphi \mapsto (m \mapsto \sum_g \varphi(gm)[g^{-1}]). \]
Applying $\cF =  \Hom_\Z( -, \Z)$ to (\ref{e:wbx}) and noting that $Z^\theta$ is $\Z$-free, we see that 
\[ \begin{tikzcd} \cF(B^\theta)  \ar[r] & \cF(W^\theta)  \end{tikzcd} \] is surjective, and hence our transpose fits into a short exact sequence \begin{equation} \label{e:nwv}
   \begin{tikzcd}
   0 \ar[r] & \cF(W^\theta)  \ar[r] & \cF(V^\theta) \ar[r] &  \nabla_{\Sigma}^{\Sigma'}(H)^{\tr} \ar[r] & 0. \end{tikzcd} \end{equation}
The injectivity of the first nontrivial arrow in (\ref{e:nwv}) follows since $\Cl_{\Sigma}^{\Sigma'}(H)$ is finite.

Next we revisit (\ref{e:jpc}) and apply the snake lemma.
Since $\Cl_{S'}^{\Sigma'}(H)$ is trivial, we extract an exact sequence
\begin{equation} \label{e:tate2} 
\begin{tikzcd}
0  \ar[r] & \cO_{H, S', \Sigma'}^*  \ar[r] & V^\theta  \ar[r] & (W')^\theta  \ar[r] & 0. \end{tikzcd} 
\end{equation}
Since $(W')^\theta$ is $\Z$-free, we obtain
\begin{equation} \label{e:wvo}
\begin{tikzcd}
0  \ar[r] & \cF((W')^\theta)  \ar[r] & \cF( V^\theta)  \ar[r] &  \cF(\cO_{H, S', \Sigma'}^*)  \ar[r] & 0. \end{tikzcd}
 \end{equation}
Now, the map $V^\theta \! \longrightarrow \! (W')^\theta$ factors through $W^\theta$.
Indeed, this map is the composition of $V^\theta  \longrightarrow  W^\theta$ with the map
$W^\theta \! \longrightarrow \! (W')^\theta$ induced by $\alpha_w$ on the components corresponding to $v \in S' - \Sigma$ and the identity on the components corresponding to $v \in \Sigma$.
 By inducing (\ref{e:alphaw}) from $G_w$ to $G$ and taking the product over $v \in S' - \Sigma$, we find that this latter map sits in a short exact sequence
 \begin{equation} \label{e:yww}
  \begin{tikzcd}
  0  \ar[r] & Y_{H,S' - \Sigma}  \ar[r] & W^\theta  \ar[r] & (W')^\theta  \ar[r] & 0. \end{tikzcd} \end{equation}
Since $(W')^\theta$ is $\Z$-free, applying $\cF$ to (\ref{e:yww}) gives another short exact sequence that fits together with (\ref{e:nwv}) and (\ref{e:wvo}) in the following commutative diagram.
\[ 
\begin{tikzcd}
 & 0 \ar[d] & 0 \ar[d] & \\
0  \ar[r] & \cF((W')^\theta) \ar[d]  \ar[r] & \cF( V^\theta)  \ar[r]  \ar[d,equal] &  \cF(\cO_{H, S', \Sigma'}^*)  \ar[r] \ar[d] & 0 \\
 0 \ar[r] & \cF(W^\theta)  \ar[r] \ar[d, two heads] & \cF(V^\theta) \ar[r]  \ar[d] &  \nabla_{\Sigma}^{\Sigma'}(H)^{\tr} \ar[r] & 0 \\
  & \cF(Y_{H,S' - \Sigma}) & 0 & 
\end{tikzcd}
\]
The snake lemma therefore yields an isomorphism
\begin{equation} \label{e:nyo}
\begin{tikzcd}
\nabla_{\Sigma}^{\Sigma'}(H)^{\tr} \cong \coker(\cF(Y_{H,S' - \Sigma}) \ar[r,"\alpha"] & \cF(\cO_{H,S', \Sigma'}^*) ). 
\end{tikzcd}
\end{equation}
 It is easy to  explicitly describe the map $\alpha$ appearing in (\ref{e:nyo}).
Given $\varphi \in   \cF(Y_{H,S' - \Sigma})$, we have \[ \alpha(\varphi)(x) = \varphi((\ord_w(x))_{w \in (S' - \Sigma)_H}). \]  Therefore, (\ref{e:nyo}) is exactly the description of $\Sel_{\Sigma}^{\Sigma'}(H)$ given in  (\ref{e:sels}).

The statement for (A2$'$) follows similarly.
\end{proof}

\subsection{Extension class via Galois cohomology} \label{s:extgal}

In this section we assume (A1) but not (A2).
As in \S\ref{s:ext} we set $M = \Cl_{\Sigma}^{\Sigma'}(H)^-$ and let $L$ denote the field extension of $H$ corresponding to $M$ via class field theory.
In Lemma~\ref{l:splitting} we gave a formal proof that the Artin reciprocity map 
\[ \begin{tikzcd}
 \varphi\colon G_H  \ar[r] & \Gal(L/H) \cong M, 
 \end{tikzcd} \] viewed as an element of $H^1(G_H, M)$, lifts to a unique class  $\lambda \in H^1(G_F, M)$.
A cocycle representing $\lambda$ is given by
\begin{equation} \label{e:kappadef2}
 \lambda(g) = \varphi(g{c}g^{-1} {c}^{-1})^{1/2}, \end{equation}
where ${c}$ is any fixed complex conjugation in $G_F$ and $m^{1/2}$ denotes the unique square root of the element $m$ in the finite abelian group of odd order $M$.
It is elementary to check that the function defined by (\ref{e:kappadef2}) is a well-defined cocycle representing a class in $H^1(G_F, M)$.
Furthermore, if $g \in G_H$, then since complex conjugation acts as inversion on $M$ we have $\varphi({c}g^{-1} {c}^{-1}) = \varphi(g)$ and hence $\lambda(g) = \varphi(g^2)^{1/2} = \varphi(g)$.

Recall that in \S\ref{s:ext} we explained how restriction to the decomposition group at $v$ in $G_F$ gives rise to classes $\lambda_v \in H^1(G_v, M)$.  We now prove Lemma~\ref{l:ext}, restated below.

\begin{lemma} The extension class in $\Ext_{\Z[G]^{-}}^{1}( X_{H,\Sigma}^{-} , M)$ determined by $\nabla_{\Sigma}^{\Sigma'}(H)^-$ corresponding to the minus part of the exact sequence $(\ref{e:nablaext})$ is equal to $(\lambda_v)_{v \in \Sigma}$ under the isomorphism $(\ref{e:ext})$.
\end{lemma}

\begin{proof}  Recall the explicit description of the isomorphism (\ref{e:ext}) given in \S\ref{s:ext}.  With $\gamma_v$ defined as in (\ref{e:alphadef}), it suffices to show that we can choose $x$ such that $\gamma_v = \lambda_v$.

This requires the explicit construction of $\nabla_{\Sigma}^{\Sigma', -}$ in \S\ref{s:construction}.   Recall that $S'$ was chosen so that the $G_{v'}$ for $v' \in S'$ cover $G$; in particular, there exists $v' \in S'$ such that $c \in G_{v'}$ (here $c$ is complex conjugation).  For notational simplicity we assume that the Frobenius of $v'$ in $G$ is equal to $c$ (we are of course free to add a $v'$ with this property to the set $S'$).  The restriction of $\theta_B$ to the factor corresponding to $v'$ is therefore $y \mapsto y \cdot (c - 1)$.  Hence an explicit element $\tilde{x} \in B^{\theta,-} \subset \prod_{S'} \Z[G]^-$ lifting $\frac{1}{2}(w - \overline{w}) \in X_{H,\Sigma}^-$ is the tuple having coordinate at $v$
equal to $(1 - c)/2$, coordinate at $v'$ equal to $1/2$, and all other coordinates equal to 0. For $g\in G_v$,
\[ \gamma_v(g) = (g - 1)\tilde{x} = ((1 - c)(g-1))/2,  (g-1)/2, 0)_{v, v', v''\neq v, v'}. \]
This is an element of $W^{\theta, -}$ and to conclude we must compute its image in $M$ under the snake map $W^\theta\longrightarrow M$ associated to  (\ref{e:rwdiagram}).  This snake map was described explicitly in \cite{rw}*{Theorem 5},  as follows.
Write $\tilde{G} = \Gal(L/F)$, where as above $L$ is the extension of $H$ corresponding to $M$ via class field theory.
Write $(\Delta \tilde{G})$ for the augmentation ideal of $\Z[\tilde{G}]$ and let $\Delta(\tilde{G}, M)$ denote the kernel of the canonical map
$\Z[\tilde{G}] \longrightarrow \Z[G]$.  There is a canonical short exact sequence
\begin{equation} \label{e:rw} 
\begin{tikzcd}
0  \ar[r] & M \cong   \dfrac{\Delta(\tilde{G},M)}{\Delta(\tilde{G}, M) \Delta\tilde{G}} 
 \ar[r] &  \dfrac{\Delta\tilde{G}}{\Delta(\tilde{G}, M) \Delta\tilde{G}}  \ar[r] & \Delta G \ar[r] & 0. 
 \end{tikzcd}
  \end{equation} 
Ritter  and Weiss associate to an element $w \in W$ an element $\rho(w) \in \Delta\tilde{G}/\Delta(\tilde{G}, M) \Delta\tilde{G}$.  When  $w \in  W^\theta$, the element $\rho(w)$ has trivial image in $\Delta G$ and hence gives rise to an element of $M$; this is the explicit description of the snake map $W^\theta\longrightarrow M$.

The components $\rho_v, \rho_{v'}$ of the map $\rho$ have slightly different definitions in the case $v \in \Sigma$, when the corresponding component of $W$ is
\[ \Ind_{G_v}^G \Delta G_v = \Z[G] \otimes_{\Z[G_v]} \Delta G_v \subset \Z[G] \otimes_{\Z[G_v]}  \Z[G_v] = \Z[G], \]
and the case $v' \in S' - \Sigma$ when  the corresponding component of $W$ is \[ \Z[G] \otimes_{\Z[G_v]}  \Z[G_v] = \Z[G].\]
 To describe these,  let $\tilde{g}$ and $\tilde{c}$ represent lifts to $\tilde{G}$ of $g$ and $c$ lying in the decomposition group associated to $v$ and $v'$, respectively.  For $v$, we write the corresponding  component of $2\gamma_v(x) = 2(g-1)\tilde{x}$ as $(1 - c) \otimes(g-1)$, and then
\[ \rho_v((1 - c) \otimes(g-1)) =  (1 - \tilde{c})(\tilde{g} - 1). \]
Meanwhile for $v'$ the corresponding component of  $2(g-1)\tilde{x}$ is simply $(g-1) \otimes  1$ and 
\[ \rho_{v'}((g-1)\otimes 1) = (\tilde{g} -  1)(\tilde{c} - 1). \] 
Adding  these, we obtain
\[ \rho(2(g-1)\tilde{x}) = \tilde{g}\tilde{c} - \tilde{c}\tilde{g}. \]

The explicit description of the isomorphism $\Delta(\tilde{G},M)/\Delta(\tilde{G}, M) \Delta\tilde{G} \cong M$ given in \cite{rw}*{Page 155} shows that the element  
\[ \tilde{g}\tilde{c} - \tilde{c}\tilde{g} = (\tilde{g}\tilde{c}\tilde{g}^{-1}\tilde{c}^{-1} - 1) (\tilde{c}\tilde{g}) \]
corresponds to $\tilde{g}\tilde{c}\tilde{g}^{-1}\tilde{c}^{-1}\in M$.
Therefore \[ \gamma_v(g) = (\tilde{g}\tilde{c}\tilde{g}^{-1}\tilde{c}^{-1})^{1/2} = \lambda_v(g) \] as desired (see (\ref{e:kappadef2})).
\end{proof}

\section{Appendix: Kurihara's Conjecture} \label{s:kurihara}

In this section, we prove Kurihara's Conjecture on the Fitting ideal of $\Cl^T(H)^{\vee,-}$, bootstrapping from the partial version proven in Theorem~\ref{t:ks}.  We first recall the statement of the conjecture, starting with notation from Lemma~\ref{l:sku}.
For $S_\infty \subset J \subset S_\infty \cup S_{\ram}$, write $\overline{J} = S_{\ram} \setminus J$.  
Let $H^{\overline{J}}$ denote the maximal subextension of $H/F$ that is unramified at primes in $\overline{J}$. This is the field $H^{I_{\overline{J}}}$, where $I_{\overline{J}}$ is the subgroup of $G$ generated by the inertia groups $I_v$ for $v \in \overline{J}$.
Note that the extension $H^{\overline{J}}/F$ is unramified outside ${J}$, and hence \[ \Theta_{{J}, T}(H^{\overline{J}}/F) \in \Z[G/I_{\overline{J}}].\] 
Since $\N I_{\overline{J}}$ divides $\prod_{v \in \overline{J}} \N I_v$, multiplication by $\prod_{v \in \overline{J}} \N I_v$ yields a well-defined map \[ \Z[G/I_{\overline{J}}] \longrightarrow \Z[G].\]
As noted in (\ref{e:skalt}), the version of Kurihara's conjecture stated in the introduction is equivalent to the equality
\begin{equation} \label{e:kc}
\Fitt_{\Z[G]^-} \Cl^T(H)^{\vee,-} = \bigg( \prod_{v \in \overline{J}} \N I_v \cdot \Theta_{J, T}(H^{\overline{J}}/F)^\# \colon S_\infty \subset J \subset S_\infty \cup S_{\ram}\bigg) \subset \Z[G]^-. 
\end{equation}

\subsection{Functorial properties}

  We begin with some functorial properties of the
construction of $\nabla_{\Sigma}^{\Sigma'}$.
Throughout this section we assume (A2).   In our application to Kurihara's conjecture we will have $\Sigma' = T$, which contains no ramified primes, so (A2) is satisfied.

\begin{lemma}  \label{l:vbasis}
 For any subgroup $\Gamma \subset G$, we have 
\[ V^\theta(H)^\Gamma = (\N\Gamma) V^\theta(H) \cong V^\theta(H^\Gamma). \]
\end{lemma}

\begin{proof}  As $V^{\theta}(H)$ is projective over $\Z[G]$, by \cite{bourbaki}*{Chapter I, Proposition 10} it follows that $V^{\theta}(H)^{\Gamma} = \Z[G]^\Gamma \otimes_{\Z[G]} V^{\theta}(H)$. The equality $V^\theta(H)^\Gamma = (\N\Gamma) V^\theta(H)$ follows from the fact that $\Z[G]^{\Gamma} = (\N\Gamma) \Z[G]$.  We must show $V^\theta(H)^\Gamma \cong V^\theta(H^\Gamma)$.  

We first check that $V(H)^\Gamma \cong V(H^\Gamma)$, which we can do componentwise over all places $v$.  
Each component of $V(H)$ is of the form $\Ind_{G_w}^G N_w(H_w)$ for a $G_w$-module $N_w(H_w)$.  If we write $\Gamma_w = \Gamma \cap G_w$, then we claim that
\begin{equation} \label{e:ind}
 (\Ind_{G_w}^{G} N_w(H_w))^\Gamma \cong \Ind_{G_w/\Gamma_w}^{G/\Gamma} N_w(H_w)^{\Gamma_w} 
 \end{equation}
as $G/\Gamma$-modules. To see this note that by \cite{weibel}*{Lemma 6.3.4}, the induced modules can be identified with co-induced modules, and therefore the isomorphism (\ref{e:ind}) is equivalent to the following natural isomorphism 
\[
\Hom_{\Gamma}(\Z, \Hom_{G_w}(\Z[G], N_w(H_w))) \cong \Hom_{G_w/ \Gamma_w}(\Z[G/\Gamma], \Hom_{\Gamma_w}(\Z, \N_w(H_w))).
\]
So it suffices to prove that in each case 
\[ N_w(H_w)^{\Gamma_w} \cong N_w(H_w^{\Gamma_w}). \] 
as $G_w/\Gamma_w$-modules.

For $v \not\in S'$, we have $N_w(H_w) = H_w^*$ or $\cO_w^*$, so this holds trivially.
 For $v \in S'$ there is a map 
\[
\begin{tikzcd}
V_w(H_w^{\Gamma_w}) \ar[r,"{\N\Gamma_w}"] & V_w(H_w)^{\Gamma_w}
\end{tikzcd}
\] 
given by
\[
(\sigma -1) \mapsto \N\Gamma_w(\tilde{\sigma}-1)
\]
for any $\sigma \in \text{W}((H_w^{\Gamma_w})^{\ab}/F_v)$ and any lift $\tilde{\sigma} \in \text{W}(H_w^{\ab}/F_v)$ of $\sigma$. This map is well-defined because of the isomorphism
\[ \text{W}(H_w^{\ab}/(H_w^{\Gamma_w})^{\ab}) \cong \ker(\N\Gamma_w \colon H_w^* \longrightarrow (H_w^{\Gamma_w})^*). \]
We have a
 commutative diagram connecting the exact sequences (\ref{e:vseq}) for $H_w$ and $H_w^{\Gamma_w}$:
\begin{equation} \label{e:vseq2}
\begin{tikzcd}
 1  \ar[r] &   (H_w^{\Gamma_w})^* \ar[r] \ar[d, equal] &  V_w(H_w^{\Gamma_w}) \ar[r] \ar[d,"\N\Gamma_w"] &  \Delta(G_w/\Gamma_w) \ar[r] \ar[d, "\N\Gamma_w" ] &  1 \\
  1  \ar[r] &   (H_w^*)^{\Gamma_w} \ar[r]  &  V_w(H_w)^{\Gamma_w} \ar[r] &  (\Delta G_w)^{\Gamma_w} \ar[r] &  1.
 \end{tikzcd}
\end{equation}
The exactness of the bottom row follows from Hilbert's Theorem 90.  The flanking vertical arrows are easily seen to be isomorphisms, so the central vertical arrow is as well.  The right square is cartesian and we use this below.

We have therefore proven that $V(H)^\Gamma \cong V(H^\Gamma)$.  To conclude we claim that there is a commutative diagram
\begin{equation} \label{e:vodiagram} \begin{tikzcd}  V(H^\Gamma) \ar[r, "\sim"] \ar[d, "\theta"] & V(H)^\Gamma  \ar[d, "\theta"] \\
\fO(H^\Gamma) \ar[r, "\sim"] & \fO(H)^\Gamma,
\end{tikzcd}
\end{equation}
from which the desired isomorphism $V^\theta(H)^\Gamma = V^\theta(H^\Gamma)$ follows. 

To prove the claim we need to construct the bottom arrow giving a commutative square. Taking $\Gamma$-invariants of the sequence in equation (\ref{e:cog}) and noting that $H^1(\Gamma, C_H) =1$ (see for example, \cite{neukirch}*{Theorem III.4.7}), we get the short exact sequence
\begin{equation} \label{e:gammainv}
\begin{tikzcd}
1 \ar[r] & C_{H}^{\Gamma} \ar[r] & \fO(H)^{\Gamma}  \ar[r] &  \Delta(G)^{\Gamma}  \ar[r]  & 1.
\end{tikzcd}
\end{equation}
Using the isomorphisms $\N\Gamma \colon \Delta(G/\Delta) \longrightarrow \Delta(G)^\Gamma$ and $C_{H}^{\Gamma} \cong C_{H^{\Gamma}}$  (see \cite{neukirch}*{Theorem III.2.7} for the latter), we can write this as 
\begin{equation} \label{e:gammainv2}
\begin{tikzcd}
1 \ar[r] & C_{H^{\Gamma}} \ar[r] & \fO(H)^{\Gamma}  \ar[r] &  \Delta(G/\Gamma)  \ar[r]  & 1.
\end{tikzcd}
\end{equation}
Let $\alpha \in H^2(G/\Gamma, C_{H^{\Gamma}})$ be the cohomology class representing the extension class of (\ref{e:gammainv2}). Note that the image of $\alpha$ under the inflation map 
\[
\text{infl}: H^2(G/\Gamma, C_{H^{\Gamma}}) \longrightarrow H^2(G, C_H)
\]
is given by first pulling back the short exact sequence in equation (\ref{e:gammainv}) along the map $\Delta(G) \xrightarrow{\N\Gamma} \Delta(G)^{\Gamma}$ and then pushing it forward along the inclusion $C_{H^\Gamma} \longrightarrow C_H$. Denote the fundamental classes in $H^2(G, C_H)$ and $H^2(G/\Gamma, C_{H^{\Gamma}})$ by $u_G$ and $u_{G/\Gamma}$, respectively. We have 
\[
\text{infl}(\alpha) = u_G^{\#\Gamma} = \text{infl}(u_{G/\Gamma}).
\]
The first equality follows since  $\N\Gamma$ acts as multiplication by $\#\Gamma$ on $\Gamma$-invariants.  For the second equality see \cite{neukirch}*{Proposition I.1.6}. As $\text{infl}$ is injective, we find that $\alpha = u_{G/\Gamma}$. Hence we obtain a commutative diagram 
\begin{equation}
\begin{tikzcd}
 1  \ar[r] &  C_{H^{\Gamma}} \ar[r] \ar[d] &  \fO(H^{\Gamma}) \ar[r] \ar[d] &  \Delta(G/\Gamma) \ar[r] \ar[d, "\N\Gamma" ] &  1 \\
  1  \ar[r] &   C_{H}^{\Gamma} \ar[r]  &  \fO(H)^{\Gamma} \ar[r] &  (\Delta G)^{\Gamma} \ar[r] &  1,
\end{tikzcd} \label{e:ohgdiagram}
\end{equation} 
with square on the right cartesian and all vertical arrows isomorphisms.
The commutativity of (\ref{e:vodiagram}) follows since the right squares in  (\ref{e:vseq2}) and (\ref{e:ohgdiagram}) are cartesian.  We must only note the following commutative diagram, whose  vertical arrows are isomorphisms:
\[ \begin{tikzcd}
\Ind_{G_w}^{G} (\Delta G_w/{\Gamma_w}) \ar[r] \ar[d,"\N \Gamma_w"] &
 \Delta(G/\Gamma) \ar[d,"\N\Gamma"] \\
\Ind_{G_w}^G (\Delta G_w)^{\Gamma_w} \ar[r,"\N(\Gamma/\Gamma_w)"] & \Delta(G)^{\Gamma}.
\end{tikzcd}
\]
This completes the proof.
\end{proof}

Recall from (\ref{e:ww2}) that we have an injection \[ \begin{tikzcd} (\alpha_w,\beta_w) \colon 
W_w(H_w)\ar[r] & \Z[G_w]^2. \end{tikzcd} \]  
We put $W_v(H)=  \Ind_{G_w}^{G} W_w(H_w)$ and denote by $(\alpha_v, \beta_v)$ the induced map
\begin{equation} \label{e:alphabeta}
 \begin{tikzcd} (\alpha_v,\beta_v) \colon W_v(H) \ar[r] & \Z[G]^2. \end{tikzcd}
 \end{equation}

The following lemma follows immediately from the definition of $\alpha_w, \beta_w^0,$ and $\beta_w$ given
in (\ref{e:ww2})--(\ref{e:betawdef}).

\begin{lemma} \label{l:gammared}
  Let $\Gamma \subset G$ be a subgroup.  Let $x \in W_v(H)$, and let $\overline{x}$ denote the image of $x$
  under the canonical map $W_v(H) \longrightarrow W_v(H^\Gamma)$ induced by restriction.
\begin{itemize} \item 
We have \begin{equation} \label{e:alphacong}
 \overline{\alpha_v(x)} = \alpha_v(\overline{x}) 
 \end{equation}
 in $\Z[G/\Gamma]$.  Here the left side denotes the reduction modulo $\Gamma$ of $\alpha_v(x)$.  On the right side, the map $\alpha_v$ is the map of (\ref{e:alphabeta}) with $(H, G)$ replaced by $(H^\Gamma, G/\Gamma)$.
\item Suppose that $\Gamma \supset I_w$.  We have
\begin{equation} \label{e:betacong}
 \overline{\beta_v^0(x)} = \beta_v(\overline{x}) 
 \end{equation}
in $\Z[G/\Gamma]$, with notation as in the previous item.
\end{itemize}  
\end{lemma}

The lemma below follows since the isomorphism $\N\Gamma \cdot V^\theta(H) \cong V^\theta(H^\Gamma)$ sends $\N \Gamma \cdot x$ to the restriction of $x$, denoted  $\overline{x} \in V^\theta(H^\Gamma)$.

\begin{lemma} \label{l:gammared2}  Let $\Gamma \subset G$ be a subgroup.
 Let $x \in V^\theta(H)$, and let $\overline{x}$ denote the image of $\N\Gamma \cdot x$ under the isomorphism  $\N\Gamma \cdot V^\theta(H) \cong V^\theta(H^\Gamma)$
  described in Lemma~\ref{l:vbasis}.  Then the congruences (\ref{e:alphacong}) and (\ref{e:betacong}) hold, with the latter under the assumption $\Gamma \supset I_w$.
\end{lemma}

\subsection{Proof of Kurihara's Conjecture}

For a nonnegative integer $i$ and finitely presented $R$-module $M$, we write $\Fitt^i_R(M)$ for the $i$th Fitting ideal of $M$.  Throughout this text, $\Fitt_R(M)$ has denoted $\Fitt^0_R(M)$, and we continue this convention.
The connection between $\Cl^T(H)^\vee$ and Ritter--Weiss modules is provided by the following lemma.

\begin{lemma}  \label{l:clnab} Let $s = \#S_{\ram}$.  We have
\[ \Fitt_{\Z[G]^-} (\Cl^T(H)^{\vee,-})= (\Fitt_{\Z[G]^-}^s \nabla_{S_\infty}^T(H)^-)^\#. \]
\end{lemma}

\begin{proof}  By Lemma~\ref{l:trnabla}, the transpose of $ \nabla_{S_\infty}^T(H)$ associated to the presentation (\ref{e:pp}) is $\Sel_{S_\infty}^{T}(H)$.  Since $\Sigma = S_\infty$, we have $S'_{\ram} = S_{\ram}$.  Therefore (\ref{e:bfree}) and Lemma~\ref{l:projective} imply that this presentation of  $ \nabla_{S_\infty}^T(H)$ has precisely $s$ more generators than relations.  It follows that
\begin{equation} \label{e:selnab}
 \Fitt_{\Z[G]} (\Sel_{S_\infty}^T(H))  = (\Fitt_{\Z[G]}^s \nabla_{S_\infty}^T(H))^\#. 
 \end{equation}
It remains to observe that we showed in (\ref{e:selt}) an isomorphism 
\[ \Sel_{S_\infty}^T(H)^- \cong \Cl^T(H)^{\vee,-} \]
of $\Z[G]^-$-modules.
\end{proof}

To prove Kurihara's conjecture (\ref{e:kc}), it therefore remains to prove that
\begin{equation} \label{e:knew}
\Fitt_{\Z[G]^-}^{s} \nabla_{S_\infty}^T(H)^- = \bigg( \prod_{v \in \overline{J}} \N I_v \cdot \Theta_{J, T}(H^{\overline{J}}/F) \colon S_\infty \subset J \subset S_\infty \cup S_{\ram}\bigg). 
\end{equation}

 It suffices to prove (\ref{e:knew}) after tensoring with $\Z_p[G]^-$ over $\Z[G]^-$ for every odd prime $p$.  We therefore fix an odd prime $p$ and write $R = \Z_p[G]$. 

  As $R$ is a product of local rings, $V^\theta_p = V^\theta \otimes_{\Z} \Z_p$ is free of rank $t = \#S' - 1$ over $R$ by Lemma~\ref{l:projective}.
We fix an $R$-basis $v_1, \dots, v_t$ of $V^\theta_p$.  We denote the canonical basis of \[ B_p = R^{s+t+1} = \prod_{v \in S'\setminus S_{\ram}} R \prod_{v \in S_{\ram}} R^2 \]
by $\{ e_v \}_{v \in S' - S_{\ram}} \cup \{e_{v,0}, e_{v,1}\}_{v \in S_{\ram}}.$ 
Fix an infinite place $\infty$ of $F$.  Recall that $\theta_B(e_\infty)=1$. 
We can therefore define a basis of $B^\theta_p$,
\[ \{ f_v  \}_{v \in S'\setminus S_{\ram}, v \neq \infty} \cup \{f_{v,0} , f_{v,1}\}_{v \in S_{\ram}},\] 
by
\[ 
f_v = e_v - \theta_B(e_v)e_\infty
\]
and similarly for the $f_{v,0}, f_{v,1}$.
The purpose of this basis is that the coordinates of an element of $B_p^\theta$ with respect to the $f$'s are the same as its coordinates with respect to the $e$'s, with the coordinate at $\infty$ ignored.

Let $A$ denote the matrix of the map $V^\theta_p \longrightarrow B_p^\theta$ with respect to our chosen  bases.
By definition, $\Fitt_{R^-}^{s} {\nabla}_{S_\infty}^T(H)_p^-$ is the ideal generated by the determinants of the submatrices 
of $A$ determined by selecting any $t$ of its columns.  These columns are indexed by the basis vectors of $B_p^\theta$.  We first show that if the columns associated to the basis vectors $f_{v,0}, f_{v,1}$ of a place $v \in S_{\ram}$ are selected, then the resulting determinant vanishes.

\begin{lemma} \label{l:detz} Let $x_1, x_2 \in W_v(H)$.  Then
\[ \det\mat{\alpha_v(x_1)}{\beta_v(x_1)}{\alpha_v(x_2)}{\beta_v(x_2)} = 0 \]
in $\Z[G]$.
\end{lemma}

\begin{proof}  It suffices to prove that if $x_1, x_2 \in W_w(H_w)$ then
\[ \det\mat{\alpha_w(x_1)}{\beta_w(x_1)}{\alpha_w(x_2)}{\beta_w(x_2)} = 0 \]
in $\Z[G_w]$.  We have
\[ \det\mat{\alpha_w(x_1)}{\beta_w(x_1)}{\alpha_w(x_2)}{\beta_w(x_2)} = 
\N I_w \cdot \det\mat{\overline{\alpha_w(x_1)}}{\beta_w^0(x_1)}{\overline{\alpha_w(x_2)}}{\beta_w^0(x_2)},\]
where the bar denotes reduction modulo $I_w$.
The determinant on the right vanishes since $\overline{\alpha_w(x)} = (\sigma_w - 1) \beta_w^0(x)$, as we noted in (\ref{e:iwred}).
\end{proof}

Lemma~\ref{l:detz} allows for the following calculation.

\begin{lemma} \label{l:fittbd} Let $S_\infty \subset J \subset S_\infty \cup S_{\ram}$ be any subset and write $j = \#(J \setminus S_\infty)$. Recall the notation $\overline{J} = S_{\ram} \setminus J$.  We have
\[  \Fitt^{s-j}_{R} \nabla^T_{J}(H)_p = \bigg( \prod_{v \in \overline{J \cup J'}} \N I_v \cdot \Fitt_{R} \nabla^T_{J \cup J' }(H^{\overline{J \cup J'}})_p \colon J' \subset \overline{J} \bigg).
\]
\end{lemma}

\begin{proof} 
The matrix $A_J$ for the presentation $V^\theta_p \longrightarrow B^\theta_p$ of $\nabla_{J}^T(H)_p$ is simply the matrix $A$ with the columns corresponding to the basis vectors $f_{v,1}$ removed for $v \in J$.  It has dimension $t \times (t+s-j)$. The $(s-j)$th Fitting ideal is computed by choosing $t$ of the columns, computing the determinant, and taking the ideal generated by all such choices.
 The columns of $A_J$ can be partitioned into $t - (s-j)$ columns corresponding to the 
  $v \in S'\setminus \overline{J}$, $v \neq \infty$, and $(s-j)$ pairs of columns corresponding to 
  the $v \in \overline{J}$.  Lemma~\ref{l:detz} implies that if we choose both columns in 
  the pair corresponding to some $v \in \overline{J}$, then the resulting determinant vanishes. 
   It follows that 
 \begin{equation} \label{e:fittjj}
  \Fitt^{s-j}_R \nabla^T_{J}(H)_p = (\det(A_{J, J'}) \colon J' \subset \overline{J}) 
  \end{equation}
   where $A_{J,J'}$ is the square matrix obtained by choosing the following $t$ columns of $A_J$:
 \begin{itemize}
 \item All $t - (s-j)$ columns corresponding to the $v \in S' \setminus \overline{J}$.
 \item The first column of the pair corresponding to the $v \in J'$.
 \item The second column of the pair corresponding to the $v \in \overline{J \cup J'}$.
 \end{itemize}
The second column for $v \in \overline{J \cup J'}$ is the column vector $(\beta_v(v_i))_{i=1}^t$, where we recall that $v_1, \dotsc, v_t$ is our basis for $V^\theta(H)$.  Since $\beta_v(x) = \N I_v \cdot \beta_v^0(x)$, we pull out the factors $\N I_v$ from these columns and find that
\begin{equation} \label{e:ajj}
 \det(A_{J, J'}) = \bigg(\prod_{v \in \overline{J \cup J'}} \N I_v \bigg) \det(A_{J,J'}^0), 
 \end{equation}
where $A_{J, J'}^0$ is the matrix $A_{J,J'}$ with $\beta_v(x)$ replaced by $\beta_v^0(x)$ for $v \in \overline{J \cup J'}$.  Note that $A_{J, J'}^0$ is well-defined as a matrix over $\Z_p[G/I_{\overline{J \cup J'}}]$, and multiplication by $\prod_{v \in \overline{J \cup J'}} \N I_v$ yields  a well-defined element of $R$.

\medskip
By Lemma~\ref{l:vbasis}, the module $V^\theta(H^{\overline{J \cup J'}})_p$ has a $\Z_p[G/I_{\overline{J \cup J'}}]$-module basis  \[ \overline{v}_1 = \N I_{\overline{J \cup J'}} \cdot v_1, \dotsc, \overline{v}_t =  \N I_{\overline{J \cup J'}} \cdot v_t. \]
 It then follows from Lemma~\ref{l:gammared2} that the matrix \[ A_{J,J'}^0 \in M_{t \times t}(\Z_p[G/I_{\overline{J \cup J'}}]) \] is precisely the square matrix for the presentation $V^\theta_p \longrightarrow B^\theta_p$ of the module $\nabla_{J \cup J'}^T(H^{\overline{J \cup J'}})_p$.  For this, note that $H^{\overline{J \cup J'}}$ is unramified at $v \in \overline{J \cup J'}$, so by definition the corresponding column in the matrix of the presentation is $(\beta_v(\overline{v}_i))_{i=1}^t$.  Meanwhile by definition the other columns are $(\alpha_v(\overline{v}_i))_{i=1}^t$.  Therefore,
 \begin{equation} \label{e:ajfitt} (\det(A_{J,J'}^0)) = \Fitt_R \nabla_{J \cup J'}^T(H^{\overline{J \cup J'}})_p.
 \end{equation}
Combining (\ref{e:fittjj}), (\ref{e:ajj}), and (\ref{e:ajfitt}) yields the desired result.
\end{proof}

Note that Lemma~\ref{l:fittbd} did not require $p$ to be odd, or to project to the minus side; in particular the result holds over $\Z[G]$.  In what follows we do require $p$ to be odd, and where necessary we project to the minus side.

As in \S\ref{s:fmr}, let 
\[ \Sigma = S_\infty \cup \{ v \in S_{\ram}, v \mid p \}. \]
The following is the major input from the main text of the paper, namely Theorem~\ref{t:ks}.

\begin{lemma} \label{l:mtapp}
Let $S_\infty \subset \Sigma_0 \subset \Sigma$, and let $s_0 = \# (\Sigma_0 \setminus S_\infty)$. Let $R_0 = \Z_p[G/I_{\Sigma \setminus \Sigma_0}]$.
\[ \Fitt_{R_0^-}^{s-s_0} \nabla_{\Sigma_0}^T(H^{\Sigma \setminus \Sigma_0})_p^- = 
\bigg(\Theta_{\Sigma_0 \cup J_0, T}(H^{\overline{\Sigma_0 \cup J_0}}) \prod_{v \in \overline{\Sigma \cup J_0}} \N I_v \colon J_0 \subset \overline{\Sigma} \bigg). \]
\end{lemma}

\begin{proof}
By the same argument as in (\ref{e:selnab}), we have
\[ \Fitt_{R_0^-}^{s-s_0} \nabla_{\Sigma_0}^T(H^{\Sigma \setminus \Sigma_0})^-_p = (\Fitt_{R_0^-} \Sel_{\Sigma_0}^T(H^{\Sigma \setminus \Sigma_0})_p^-).^\#\]
The result then follows directly from Theorem~\ref{t:ks}.
\end{proof}

We can now prove Kurihara's conjecture, which in view of Lemma~\ref{l:clnab}, is equivalent to the following statement.

\begin{theorem}  We have
\begin{equation} \label{e:kc2}
\Fitt_{R^-}^s \nabla_{S_\infty}^T(H)_p^- = \bigg( \prod_{v \in \overline{J}} \N I_v \cdot \Theta_{J, T}(H^{\overline{J}}/F) \colon S_\infty \subset J \subset S_\infty \cup S_{\ram}\bigg). 
\end{equation}
\end{theorem}

\begin{proof}
By Lemma~\ref{l:fittbd}, we have
\begin{equation}
\label{e:fj1}
  \Fitt^{s}_{R} \nabla^T_{S_\infty}(H)_p = \bigg( \prod_{v \in \overline{J}} \N I_v \cdot \Fitt_{R} \nabla^T_{J}(H^{\overline{J}})_p \colon S_\infty \subset J \subset S_\infty \cup S_{\ram} \bigg).
\end{equation}
We partition each set $J = \Sigma_0 \cup J_0$, where
\[  \Sigma_0 = J \cap \Sigma, \qquad J_0 = J \setminus \Sigma.  \]
   Then (\ref{e:fj1}) can be written
\begin{equation} \label{e:fj2}
  \Fitt^{s}_{R} \nabla^T_{S_\infty}(H)_p = \bigg( \prod_{v \in \overline{\Sigma_0 \cup J_0}} \N I_v \cdot \Fitt_{R} \nabla^T_{\Sigma_0 \cup J_0}(H^{\overline{\Sigma_0 \cup J_0}})_p \colon S_\infty \subset \Sigma_0 \subset \Sigma, J_0 \subset \overline{\Sigma} \bigg).
\end{equation} 
Now apply Lemma~\ref{l:fittbd} with $J=\Sigma_0$ and $H$ replaced by 
$H^{\Sigma\setminus \Sigma_0}$.  
Note that \[ S^{\ram}(H^{\Sigma\setminus\Sigma_0}/F) \subset \overline{\Sigma} \cup \Sigma_0. \]
Writing $s_0 = \#(\Sigma_0 \setminus S_\infty)$ and $R_0 = \Z_p[G/I_{\Sigma \setminus\Sigma_0}]$, we obtain
\[  \Fitt^{s-s_0}_{R_0} \nabla^T_{\Sigma_0}(H^{\Sigma\setminus\Sigma_0})_p = \bigg( \prod_{v \in \overline{\Sigma \cup J_0}} \N I_v \cdot \Fitt_{R_0} \nabla^T_{\Sigma_0 \cup J_0 }(H^{\overline{\Sigma \cup J_0}})_p \colon J_0 \subset \overline{\Sigma} \bigg) \subset R_0.
\]
If we multiply by $\prod_{v \in \Sigma\setminus \Sigma_0} \N I_v$, we obtain exactly the terms in (\ref{e:fj2}) corresponding to $\Sigma_0$.
We therefore obtain
\begin{equation} \label{e:fj3}
  \Fitt^{s}_{R} \nabla^T_{S_\infty}(H)_p = \bigg( \prod_{v \in \Sigma\setminus \Sigma_0} \N I_v  \cdot \Fitt^{s-s_0}_{R_0} \nabla^T_{\Sigma_0}(H^{\Sigma\setminus\Sigma_0})_p: S_\infty \subset \Sigma_0 \subset \Sigma \bigg).
  \end{equation}
  To conclude, we project to the minus side and apply Lemma~\ref{l:mtapp}:
\[ 
  \Fitt^{s}_{R^-} \nabla^T_{S_\infty}(H)^-_p = 
  \bigg( \prod_{v \in \Sigma\setminus \Sigma_0} \N I_v  \prod_{v \in \overline{\Sigma \cup J_0}} \N I_v \cdot
  \Theta_{\Sigma_0 \cup J_0, T}(H^{\overline{\Sigma_0 \cup J_0}})  \colon
  S_\infty \subset \Sigma_0 \subset \Sigma,  J_0 \subset \overline{\Sigma} \bigg).
  \]
  Writing $J = \Sigma_0 \cup J_0$, we obtain the expression (\ref{e:kc2}).
\end{proof}

\end{document}